\documentclass[review,onefignum,onetabnum]{siamart171218}

\usepackage{enumerate}
\usepackage{amsmath}
\usepackage{mathrsfs}
\usepackage{xcolor}	
\usepackage{amssymb}
\usepackage{bm}
\usepackage{graphicx}
\usepackage{subcaption}
\usepackage{adjustbox}
\usepackage{soul}
\usepackage{listings}
\usepackage{array}  
\usepackage{mathtools}
\usepackage{float}
\usepackage{mdframed}
\newmdenv[linecolor=green!50!black, fontcolor=green!50!black, backgroundcolor=green!20, linewidth=2pt, roundcorner=10pt]{gnote}

\definecolor{amber}{rgb}{1.0, 0.49, 0.0}
\definecolor{darkpastelgreen}{rgb}{0.01, 0.75, 0.24}
\definecolor{darkred}{rgb}{0.64, 0.0, 0.0}
\definecolor{amberhl}{rgb}{1.0, 0.75, 0.0}

\newcommand{\red}[1]{\textcolor{black}{#1}}
\newcommand{\rededit}[1]{\textcolor{black}{#1}}

\newmdenv[linecolor=blue!50!black, fontcolor=blue!50!black, backgroundcolor=blue!20, linewidth=2pt, roundcorner=10pt]{anote}

\newcolumntype{L}{>{$}l<{$}}

\newcommand{\bb}{{\bm b}}

\newcommand{\br}{{\bm r}}

\newcommand{\bu}{{\bm u}}
\newcommand{\bv}{{\bm v}}

\newcommand{\bx}{{\bm x}}
\newcommand{\by}{{\bm y}}

\newcommand{\bfz}{{\bm 0}}

\newcommand{\cB}{\mathcal{B}}

\newcommand{\cN}{\mathcal{N}}

\newcommand{\cT}{\mathcal{T}}

\newcommand{\secref}[1]{Section~\ref{#1}}

\headers{Analysis of ZFP Compression in Iterative Methods}{A. Fox, J. Diffenderfer, J. Hittinger, G. Sanders, P. Lindstrom}

\title{Stability Analysis of Inline ZFP Compression for Floating-Point Data in Iterative Methods\thanks{Submitted to the editors \red{DATE}
		\funding{This work was performed under the auspices 
    of the U.S. Department of Energy by Lawrence Livermore National
    Laboratory under Contract DE-AC52-07NA27344 and was supported by the
    LLNL-LDRD Program under Project No. 17-SI-004, LLNL-JRNL-769679-DRAFT.}}}
\author{ Alyson Fox\thanks{Lawrence Livermore National Laboratory,
		Livermore, CA (\email{fox33@llnl.gov})}
	\and James Diffenderfer\thanks{The University of Florida,
    Gainesville, FL (\email{jdiffen1@ufl.edu})}
\and Jeffrey Hittinger\thanks{Lawrence Livermore National Laboratory,
    Livermore, CA (\email{hittinger1@llnl.gov})}
\and Geoffrey Sanders\thanks{Lawrence Livermore National Laboratory,
    Livermore, CA (\email{sanders29@llnl.gov})} 
\and Peter Lindstrom\thanks{Lawrence Livermore National Laboratory,
    Livermore, CA (\email{lindstrom2@llnl.gov})}}

\ifpdf
\hypersetup{
  pdftitle={Stability Analysis of inline ZFP Compression in Iterative Methods},
  pdfauthor={A. Fox, J. Diffenderfer, J. Hittinger, G. Sanders, P. Lindstrom}
}
\fi

\begin{document}

\maketitle

\begin{abstract}
  Currently, the dominating constraint in many high performance computing
  applications is data capacity and bandwidth, in both inter-node
  communications and even more-so in on-node data motion.  A new
  approach to address this limitation is to make use of data
  compression in the form of a compressed data array.  Storing data in
  a compressed data array and converting to standard IEEE-754 types as
  needed during a computation can reduce the pressure on bandwidth and
  storage.  However, repeated conversions (lossy compression and
  decompression) introduce additional approximation errors, which need
  to be shown to not significantly affect the simulation results.
  We extend recent work~\cite{errorzfp} that analyzed the error of a
  single use of compression and decompression of the ZFP compressed data
  array representation~\cite{zfp,zfp-doc} to the case of time-stepping
  and iterative 
  schemes, where an advancement operator is repeatedly applied in addition to the
  conversions. \red{We show that the accumulated error for iterative methods involving fixed-point
  and time evolving iterations is bounded under 
  standard constraints.}  An upper bound is established on the number of
  additional iterations required for the convergence of stationary
  fixed-point iterations.  An additional analysis of traditional forward and backward error of stationary iterative methods using ZFP compressed
  arrays is also presented.  The results of several 1D, 2D, and 3D
  test problems are provided to demonstrate the correctness of the
  theoretical bounds.    
\end{abstract}


\section {Introduction}
\label{sec:introduction}
Applications in scientific computing often result in extremely large
volumes of floating-point data, e.g, a simulation of turbulent fluid
dynamics can produce on the order of 6 TB per timestep~\cite{fluidsim}.
Thus, a significant order cost is often the off-node and on-node data
motion~\cite{Xcutting2010,Williams:2009:RIV:1498765.1498785}.  Lossless
and lossy compression algorithms have been considered to reduce the
amount of data, thus reducing time in data transfer.
Compression algorithms are most frequently considered for I/O operations (data and restart  
files) and in-memory storage of static data (e.g., tabulated material
properties), however there is potential for significant performance
gains by using compressed data types within a simulation.  In this approach, 
the solution state data could be stored in a compressed format, be
decompressed, be operated on, and be recompressed inline during each
time step or iteration of a numerical simulation.  The trends in
computer hardware are that processing power (the number of FLOPS) is
growing much more rapidly than memory and network bandwidth or memory
capacity.  Thus, compressed data types would make more efficient use of
bandwidth and storage.  

The seemingly preferred choice would be a compressed data type based on
lossless compression algorithms, 
such as FPC~\cite{Ratanaworabhan:2006:FLC:1126009.1126035},
SPDP~\cite{8416606}, BLOSC~\cite{Blosc},  
FPZIP~\cite{fpzip} and other variants, as they reproduce the original
data with no degradation.  Unfortunately, lossless compression
algorithms struggle to produce significant
compression ratios for floating-point
data~\cite{Ratanaworabhan:2006:FLC:1126009.1126035,fpzip}.
Lossy floating-point compression, e.g., SZ~\cite{sz,sz17,sz18} and
ZFP~\cite{zfp}, on the other hand, only guarantees an approximation of
the original data but produces a much higher ratio 
of data reduction than lossless compression.
Since the solution state from a
numerical simulation already contains a host of numerical approximation
errors (e.g., floating-point round-off, truncation error, and
iteration error), one can legitimately question the logic of requiring
a lossless compressed data type.  Data lost in lossy compression
may not be meaningful, i.e., it represents errors, and a lossy
compressed data type can be thought of as just another approximate
representation of real numbers with a finite number of bits that
introduces error that must be controlled. 

The use of lossy compression in numerical simulation has been suggested
before: it has been considered for checkpointing numerical
simulations~\cite{calhoun_exploring_2018} and for inline compression of
the solution state in~\cite{Laney}. In both cases, it was demonstrated 
that lossy compression can be used without causing significant changes to
important physical quantities.  However, neither application provided
theoretical backing to ensure numerical stability, that is, to control
the accumulated error resulting from repeated compression and
decompression. In~\cite{tao_improving_2018}, the authors 
investigated the convergence of iterative methods when only one or
possibly two checkpoints are used with lossy compression.

In this paper, we will consider the use of a lossy compressed data type
in place of the standard IEEE-754 \cite{IEEE754} floating point representation and analyze
the impact of the associated compression error on two common
applications in numerical simulation. 
Specifically, we consider the ZFP compressed data array, which was first
described in~\cite{zfp} and modified in~\cite{zfp-doc}, as a lossy compression
data type that individually compresses and decompresses small blocks
of $4^d$ values from $d$-dimensional data.
Due to the locality of the independently compressed blocks, ZFP
compressed data arrays are ideal for storing simulation data, since only
the block containing a particular data value needs to be uncompressed,
which gives ZFP arrays a behavior similar to standard random access
arrays.
In prior work~\cite{errorzfp}, extensive error analysis
for a single cycle of ZFP compression and decompression was performed
for all three of ZFP's encoding modes (fixed-rate, fixed-accuracy, or
fixed-precision).
Here, we extend these results to the use of ZFP compressed data arrays
under the action of an \rededit{{\it advancement operator}, an operator that advances the numerical solution whether it be for an iterative or time-stepping numerical method.}  Specifically, we consider the situation
where data is stored in a ZFP compressed array, converted (decompressed) to IEEE
double precision data types as needed, updated arithmetically under the
action of the advancement operator, and then converted (compressed) back
to the ZFP representation.  This cycle will be repeated numerous times
in a time-stepping or iterative algorithm.

\rededit{Without loss of generality, we assume that the advancement operator is bounded, a common assumption for a wide class of problems. Either we assume the advancement operator is Lipschitz continuous, which is useful for both linear and nonlinear operators, or it is a Kreiss bounded linear operator (see Section \ref{sec:bounds} for more details). In either case, without a bounded advancement operator, the numerical method may be unstable resulting in poor approximations to the true solution. Under these assumptions,} our results determine an upper bound on the accumulation of error introduced by the use of ZFP compressed arrays. Furthermore, our results
allow the formulation of criteria for the parameters of ZFP that will
ensure that the compression error remains below a user-defined tolerance.
We also establish an upper bound on the number of additional 
iterations required to achieve convergence of a stationary iterative
method when ZFP compression error accumulates. \rededit{We also extend our error bounds for successive displacement fixed-point operators that are block diagonally dominant.} Finally, we provide a
forward and backward error analysis for stationary iterative methods
using ZFP compressed arrays and compare the accumulated floating-point
round-off error to the accumulated compression error caused by ZFP.

The remainder of this paper is structured as follows.
In~\secref{sec:notation}, we will summarize the required
notation and bounds established
in~\cite{errorzfp}. In~\secref{sec:bounds}, we provide the main results, 
which are upper bounds for the error introduced by using ZFP compressed
arrays under repeated application of an advancement operator. Following the main
results, we demonstrate the validity of the error bounds numerically
in~\secref{sec:results}.

\section{Preliminary Notation, Definitions, and Theorems}
\label{sec:notation}
We first outline the ZFP compression algorithm as documented in \cite{zfp-doc} and provide the required notation and preliminary theorems, as defined in \cite{errorzfp}, that are necessary for this paper.

 Consider a $d$-dimensional scalar array. This array is partitioned into {\it blocks} of $4^d$ scalars. Each block is then compressed independently by the following steps. The floating-point values, typically represented by IEEE, from each block are converted to a block-floating-point representation \cite{MitrablockRoundingError} using a common exponent then shifted and rounded into two's complement signed integers. A near orthogonal decorrelating transform, similar to the discrete cosine transform, is applied to the $4^d$ block. 
The idea is that the decorrelating transform removes redundancies in the data by a change in basis. From the change of basis, the magnitude of the coefficients tend to correlate with the location within the block, thus, ZFP reorders the coefficients by total sequency \cite{zfp-doc}. A 2-$d$ example can be seen in Figure \ref{fig:totalsequency}. 
\begin{figure}
	\centering
	\includegraphics[width=.5\textwidth]{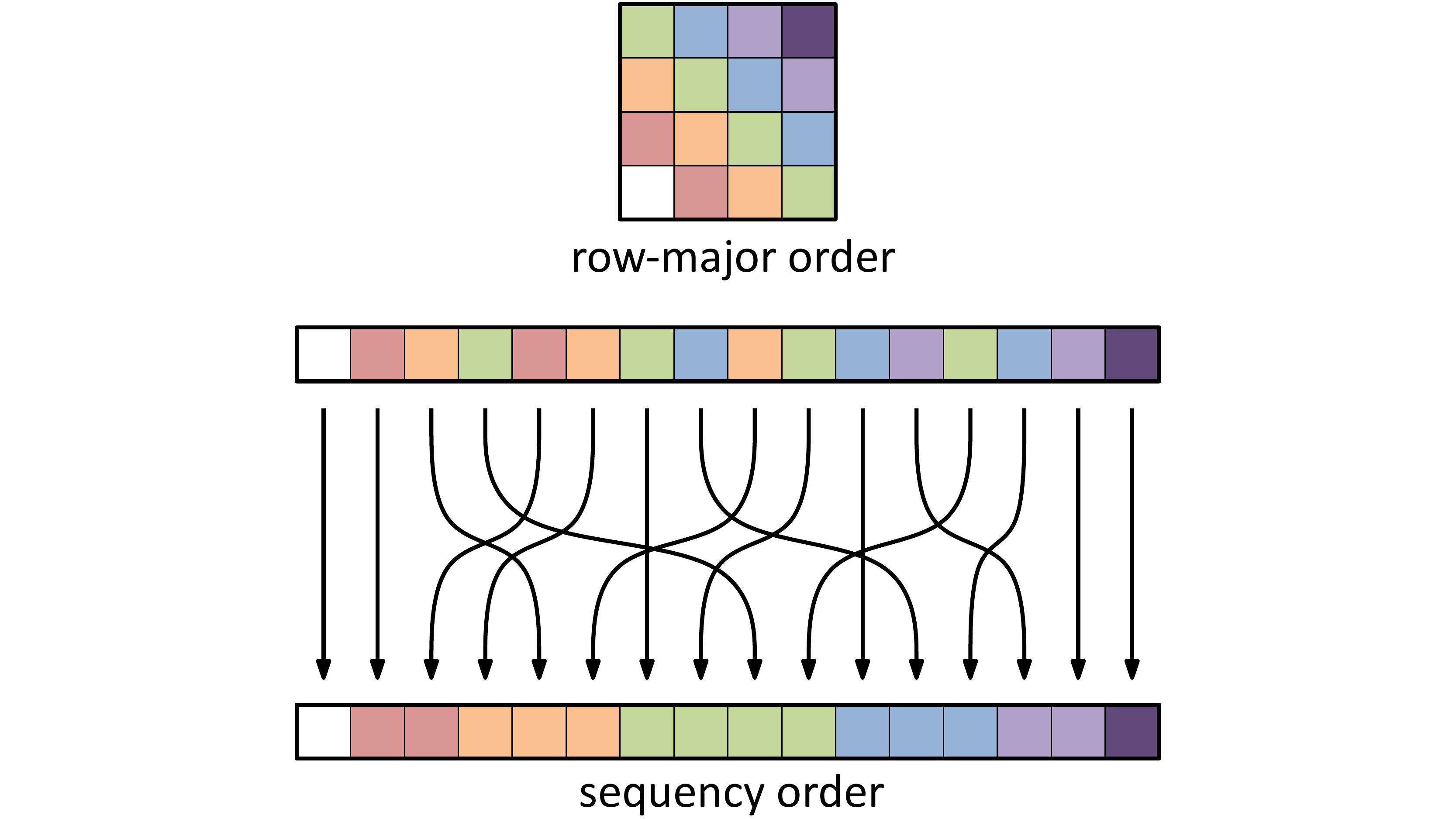}
	\caption{Total sequency ordering for a 2-dimensional array, which groups the diagonal elements together. }
	\label{fig:totalsequency}
\end{figure}
As the sign-bit in two's complement does not provide any useful information without the leading one-bit, ZFP converts each integer into a negabinary representation \cite{Knuth} where the leading one-bit encodes both the sign and approximate magnitude of the value. The block is then transposed so that it is ordered by bit-plane instead of by coefficient, from most to least significant bit; see Figure \ref{fig:bit-plane}.
\begin{figure}
	\centering
	\includegraphics[width=.5\textwidth]{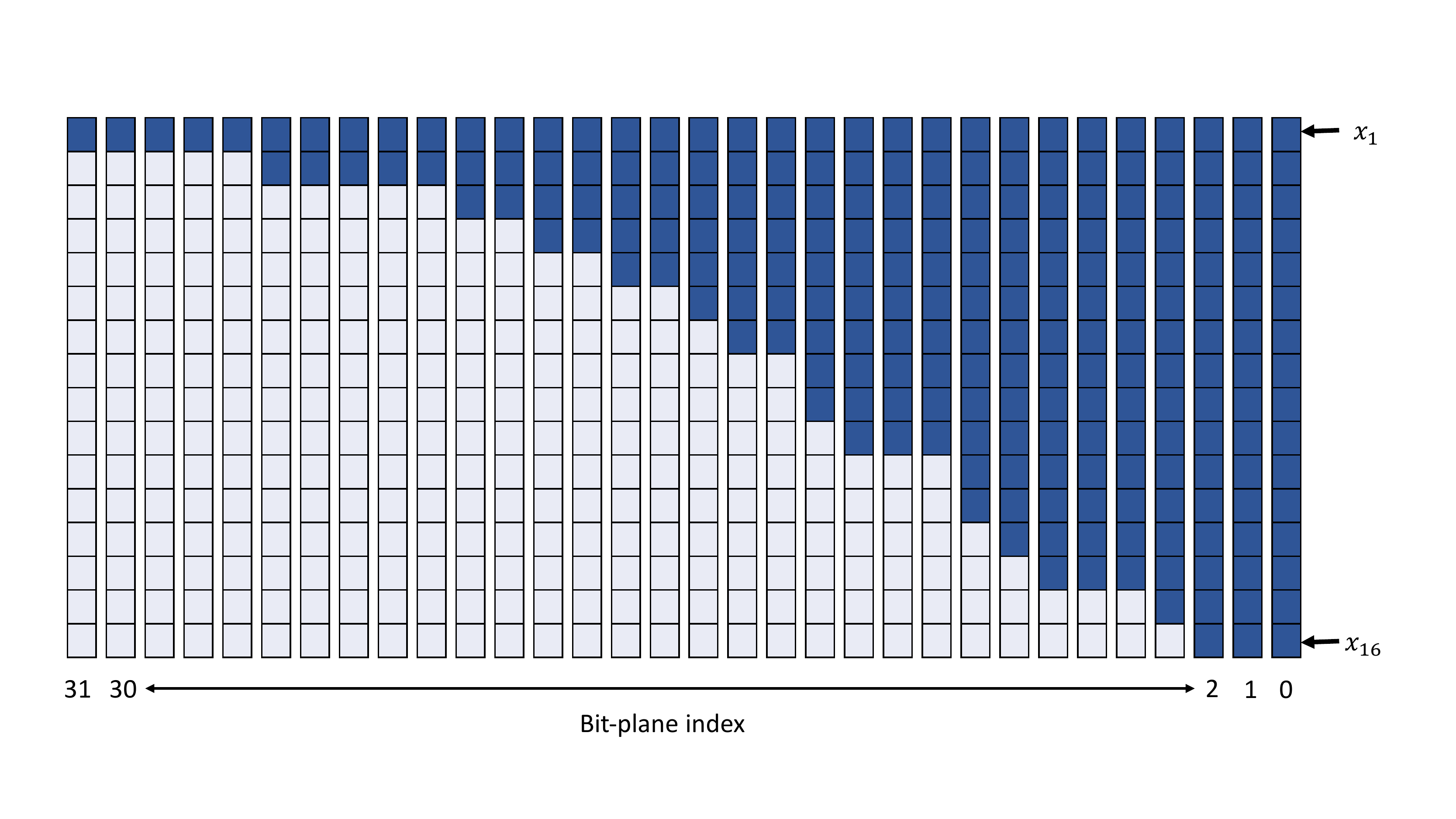}
	\caption{A 2-dimensional array represented with a 32-bit negabinary integer representation ordered by bit-plane, from most to least significant bit. Typically, the bits containing the information are concentrated in the higher transform coefficients as illustrated by the dark blue coloring (referred to as the energy compaction property \cite{Rao1990}), where the light blue represents 0 and the dark blue represents either 1 or 0.  }
	\label{fig:bit-plane}
\end{figure}
Each bit-plane is then losslessly compressed using an embedded coding scheme \cite{zfp} which emits one bit at a time until some stopping criterion is satisfied. The stopping criterion for ZFP has three modes: fixed-precision, fixed-rate, and fixed-accuracy. For more detailed information on ZFP see \cite{zfp} and \cite{zfp-doc}. 

The bounds presented in this paper are only for the fixed-precision mode, which retains a fixed number of bit-planes; however, the theorems can be extended to other modes using Theorem 5.3 and Theorem 5.4 in \cite{errorzfp}. Each block of values can then be converted back to the source format (e.g., IEEE) by applying a decompression operator. Before discussing the round-off error bound between the compression of IEEE values to ZFP established by \cite{errorzfp}, we summarize a key vector space used in the analysis in \cite{errorzfp}.

\subsection{Signed Binary and Negabinary Bit-Vector Spaces}
In \cite{errorzfp}, a vector space was introduced (which we will refer to as the \emph{infinite bit vector space}) in order to express each step of the ZFP compression algorithm as an operator on binary or negabinary representations  \cite{Knuth} of $4^d$ vectors. The components of each vector in this vector space are infinite sequences of zeros and ones that are restricted in a manner such that each real number has a unique representation as a binary sequence. For example, given $x \in \mathbb{R}$, there exist $c, d \in \mathbb{B}^{\infty}$ and $s \in \mathbb{B}$ such that $x$ can be represented in signed binary and negabinary as 
\begin{align}
{\text{Signed Binary: }} x = (-1)^s \sum_{i = - \infty}^{\infty} c_i 2^i \ \ \ \text{ and}  \ \ \ {\text{ Negabinary: }} x = \sum_{i = - \infty}^{\infty} d_i (-2)^i, \label{CrepX}
\end{align}
where $\mathbb{B} = \{0,1\}$. Placing certain restrictions on the choice of $c$, $d$, and $s$ such that each $x$ has a unique representation in the form described in (\ref{CrepX}), we are able to form the infinite bit vector spaces for signed binary and negabinary representations, which we denote by $\mathcal{B}$ and $\mathcal{N}$, respectively.  To imitate floating-point representations, \cite{errorzfp} defines subspaces $\cB_k$ and $\cN_k$ of $\mathcal{B}$ and $\mathcal{N}$, respectively, where $k$ represents the maximum number of nonzero bits allotted for each representation, excluding the sign bit in the signed binary representation. As $\cB_k$ and $\cN_k$ are subsets but \emph{not} subspaces of $\cB$ and $\cN$, all the analysis in \cite{errorzfp} took place in $\cB$, $\cN$, or $\mathbb{R}$, using operators to imitate working with a fixed number of bits. For the full definition and underlying concepts of the infinite bit vector space, see Section 3 in \cite{errorzfp}.

\subsection{Single Use Error Bound}
An upper bound on the round-off error of a single use of ZFP compression from IEEE values was established in \cite{errorzfp} by constructing an operator for each step of ZFP, that acts on elements of an infinite bit vector space and by composing these operators to make a ZFP compression and decompression operator. Notation introduced in \cite{errorzfp} will be used in this paper and will be summarized here with respect to each step of ZFP (de)compression. 

First, assume the $d$-dimensional data has been partitioned into arrays of dimension $4^d$, which we refer to as \emph{blocks}. Let a block  $\bx \in \mathbb{R}^{4^d}$ be given. For some precision $k \in \mathbb{N}$, assume every element in $\bx$ can be represented with at most $k$-consecutive bits, including the leading one bit. For frequently used IEEE single and double precision floating-point types, $k \in \{24, 53\}$, \rededit{i.e., the number of mantissa bits}. Each component in $\bx$ is then converted to a block floating-point representation with respect to a common exponent. Note, that error may occur when converting the block into signed integers and is taken into account in \cite{errorzfp}. Let $q \in \mathbb{N}$ represent the maximum number of nonzero consecutive bits that can be used to represent each component in the block floating-point representation. In the ZFP implementation, $q = k+e-2$, where $e \in \mathbb{N}$ is the maximum number of bits that represent the exponent in IEEE single or double format, i.e., $e  \in \{8,11\}$ and $q \in \{30, 62\}$. Define, $e_{max, \mathcal{B}}(\bm{x})$ and $e_{max, \mathcal{N}}(\bm{x})$ as the index of the leading nonzero bit in each corresponding infinite bit vector space (see Section 3.4 in \cite{errorzfp}). \rededit{From Lemma 3.6 in \cite{errorzfp}, we have $\|\bx\|_\infty \geq 2^{e_{max, \mathcal{B}}(\bm{x})}$, a useful relation between the magnitude of $\|\bx\|_\infty$ and $e_{max, \mathcal{B}}(\bm{x})$. }

 Once the block is converted to signed integers represented in two's complement, the block is acted on by a linear
transformation. In $d$-dimensions, the transform operator, $\cT$, is
applied to each dimension separately, and the operator can be
represented as a Kronecker product. For $A \in \mathbb{R}^{n_1,m_1}$ and
$B \in \mathbb{R}^{n_2,m_2}$, the Kronecker product is defined as  
\begin{align*}
A\otimes B =\begin{bmatrix}a_{1,1}B &a_{1,2} B & \cdots a_{1,m_1}B \\  \vdots & \ddots& \vdots  \\ 
a_{n_1,1}B &a_{n_1,2} B & \cdots a_{n_1,m_1}B \end{bmatrix}. 
\end{align*}
Then the total forward transform operator used in of ZFP is defined as 
\begin{align*}
\cT_d = \underbrace{\cT\otimes \cT \otimes \cdots \otimes \cT}_\text{$(d - 1)$-products},
\end{align*}
where $\cT \in \mathbb{R}^{4\times 4}$ is defined by
\begin{align}
\cT = \frac{1}{16} \begin{bmatrix}
\begin{array}{rrrr}
4 & 4 & 4 & 4 \\
5 & 1 & -1 & -5 \\
-4 & 4 & 4 & -4 \\
-2 & 6 & -6 & 2
\end{array}
\end{bmatrix} \quad \text{and} \quad \cT^{-1}= \frac{1}{4} \begin{bmatrix}
\begin{array}{rrrr}
4 & 6 & -4 & -1 \\
4 & 2 & 4 & 5 \\
4 & -2 & 4 & -5 \\
4 & -6 & -4 & 1
\end{array}
\end{bmatrix}. 
\end{align}
The transformation
could result in typical fixed-point round-off. Thus, the operator used in the
implementation of the algorithm \cite{zfp-doc} will be defined as $\tilde{\cT}_d$.  From \cite{errorzfp}, we have the associated lemma bounding the round-off error. 
\begin{lemma}[Lemma 4.4 \cite{errorzfp}]
	\label{lemma:boundT}
	Suppose $\bx \in \mathbb{Z}^{4^d}$, such that $e_{max, \mathcal{B}}(\bx) = q \in \mathbb{N}$ and $\bx \neq \bfz$. Then 
	\begin{align}
	\|\cT_d\bx - \tilde{\cT}_d \bx \|_\infty & \leq k_\cT\epsilon_q \|\bx\|_\infty, 
	\end{align}
	where $k_\cT = \frac{7}{4} \left( 2^d - 1 \right)$ and $\epsilon_q = 2^{1-q}$. 
\end{lemma}
After the forward transformation, each component in the block is converted losslessly from a two's complement representation to a negabinary representation. That is, a $p$-bit two's complement integer requires $p+1$ bits in negabinary; however, as ZFP no longer needs the guard bit that was used for the transformation, there is an extra bit of precision for the negabinary representation. Once each component of the block is converted to a negabinary representation, a deterministic permutation acts on the components of the block. The final compression step is then the bit stream truncation. For the fixed-precision mode, the fixed-precision parameter, denoted $\beta \geq 0$, is provided by the user to ZFP. Here, $\beta$ represents the number of most significant bit planes to keep and any discarded bit plane is mathematically equivalent to replacing the bits with all-zero bits. Note that the encoding step is omitted from the discussion as it is a lossless mapping. For the decompression operator,  each block is decompressed by applying the inverse of each compression step in reverse order. It was shown in \cite{errorzfp} that, if $\beta\leq q-2d+2$, then no additional loss will occur in the decompression steps with the exception of potential round-off error due the conversion back to the original source format (see Section 4.2 in \cite{errorzfp}). 

For $r \in \mathbb{Z}$, we define the constant $\epsilon_r = 2^{1-r}$. Let the nonlinear operators $C: \mathbb{R}^{4^d} \rightarrow \mathcal{N}^{4^d}$ and $D: \mathcal{N}^{4^d} \rightarrow \mathbb{R}^{4^d}$ denote the ZFP compression and decompression operators, as defined in \cite{errorzfp}. We can now note the following point-wise error bound resulting from compressing and then decompressing a $4^d$ block using ZFP: 

\begin{theorem}[Theorem 5.2 \cite{errorzfp}]\label{thm:diffDCandDC} 
	Assume $\bx \in \mathbb{R}^{4^d}$ with $\bx \neq 0$ such that for some $k$ and $e$, each element in $\bx$ is representable by $k$ mantissa bits and $e$ exponent bits as defined in \cite{errorzfp}. Let $0\leq \beta \leq  q- 2d+2$ be the fixed-precision parameter. Then 
	\begin{align}
		\| {D}{C} (\bx) - \bx \|_\infty &\leq K_\beta \|\bx\|_\infty,
	\end{align}
	where $q \in \mathbb{N}$ is the precision for the block-floating point representation,
	\begin{align}
		K_\beta := \left( \frac{15}{4} \right)^d \left(  (1+\epsilon_k)\left( \frac{8}{3} \epsilon_\beta+ \epsilon_q \left(1+ \frac{8}{3} \epsilon_\beta \right) \left(k_\cT(1+\epsilon_q)+1 \right )\right)+\epsilon_k \right), \label{eqn:kbeta}
	\end{align}
	and the constant $k_\cT = \frac{7}{4}(2^d-1)$.
\end{theorem}

Theorem \ref{thm:diffDCandDC} can easily be extended to any $d$-dimensional array of floating point values by taking the largest error from all blocks. Let $\hat{C}$ and $\hat{D}$ represent the  compression and decompression operator for whole $d$-dimensional array, then 
\begin{align}
\| \hat{D}\hat{C}(\bx) - \bx \|_\infty = \max_i \|{D}{C}(\bx_i) - \bx_i\|_\infty,
\end{align}
where each $\bx_i$ represents a non-overlapping $4^d$ block. When it is clear \emph{from context} that the entire data field is being (de)compressed, the $\hat{D}$, $\hat{C}$ notation will be dropped. Now that the required notation and theorems have been presented, we proceed with our discussion of error bounds for inline use of ZFP compressed arrays.

\section{Error Bounds for Inline use of ZFP Conversion}
\label{sec:bounds}
\begin{definition} \label{def:lip}An operator $g(\cdot): \mathbb{R}^n \rightarrow\mathbb{R}^n$ is {\bf{Lipschitz continuous}} with constant $L_l<\infty$ if 
		\begin{align}
			\|g(\bx) - g(\by)\|_\infty \leq L_l \|\bx - \by\|_\infty,  \qquad \text{for all} \quad  \bx, \by \in \mathbb{R}^n.
				\end{align}
\end{definition}
\begin{definition} \label{def:kreiss}A linear operator $A \in \mathbb{R}^{n \times n}$ is {\bf{Kreiss bounded}} \cite{kreiss} with constant $L_k<\infty$ if 
	\begin{align}	\|A^i\|_2 \leq L_k,   \qquad \text{for all} \ i \in \mathbb{N}.	\end{align}  
\end{definition}

\subsection{Lipschitz Continuous Advancement Operator}
Suppose that we generate a sequence $\{ \by_t \}_{t = 0}^{\infty}$ by
\begin{align}
\by_{t+1} = g(\by_t), \label{FPscheme}
\end{align} 
where the advancement operator, $g: \mathbb{R}^n \to \mathbb{R}^n $, and initial value, $\by_0 \in \mathbb{R}^n$, are given. Additionally, consider the sequence $\{ \bx_{t} \}_{t = 0}^{\infty}$ generated by 
\begin{align}
\bx_{t+1} = g ( {D} {C}( \bx_{t})), \label{ZFPscheme}
\end{align} 
where the initial point of the sequence is $\bx_0 = \by_0$. \red{In applications, the solution would be stored in the compressed state, $C(\bx_t)$, however, the solution must be converted back to IEEE in order to analyze the error. We chose to consider the error after the application of the advancement operator, however, we could have instead considered the sequence $\bx_{t+1} = {D} {C}(g ( \bx_{t}))$ and obtained similar results.} The scheme given by (\ref{FPscheme}) can either be viewed as a fixed-point or a time-stepping method, depending on the properties of $g(\cdot)$. For example, $g(\cdot)$ could represent a finite difference method of a PDE or a Jacobi iteration method. First, we will  determine a bound on the value of $\| \bx_{t+1} - \by_{t+1} \|_{\infty}$ for various properties of the advancement operator. \red{Then, we will investigate specifically the convergence for fixed-point iterative methods; we will determine the convergence properties for which the sequence $\{ \bx_{t} \}_{t = 0}^{\infty}$ converges to the same fixed-point as the sequence $\{ \by_{t} \}_{t = 0}^{\infty}$. Additionally, }we will determine how many extra iterations, $m$, are needed for the sequence $\{ \bx_{t} \}_{t = 0}^{T+m}$ to converge to the same accuracy as $\{ \by_{t} \}_{t = 0}^{T}$ with respect to the fixed-point, given $T$ iterations. 

%
%
%

Theorem \ref{thm:lip} provides a bound for the difference between the $j$-th element of the sequences $\{ \by_{t} \}_{t= 0}^{j}$ and $\{ \bx_{t} \}_{t = 0}^{j}$ for a general Lipschitz continuous advancement operator by using the single use error bound from Theorem \ref{thm:diffDCandDC}. Since the only assumption on the advancement operator $g(\cdot)$ is Lipschitz continuity, Theorem \ref{thm:lip} is not restricted to linear functions and is useful for applications involving
Lipschitz continuous nonlinear functions.
\begin{theorem}[Lipschitz Continuous Advancement Operator] \label{thm:lip}
Suppose that $g(\cdot)$ is Lipschitz continuous with Lipschitz constant $L_l$. Then
\begin{align}\label{eqn:lipbound}
\| \bx_{t+1} - \by_{t+1} \|_{\infty} \leq \sum_{j = 0}^{t} L_l^{t-j+1} K_{\beta_j} \|\bx_j\|_\infty ,
\end{align}
where $\beta_j$ is the number of bit planes kept at step $j$.
\end{theorem}
\begin{proof}
Since $g(\cdot)$ is Lipschitz continuous it follows that
	\begin{align}
\nonumber	\| \bx_{t+1} - \by_{t+1} \|_{\infty} &= \| g ( {D} {C} (\bx_{t})) - g(\by_t) \|_\infty, \\
	& \leq L_l \| {D} {C} (\bx_{t}) - \by_t \|_\infty \leq L_l \left( \| {D} {C} (\bx_{t}) - \bx_t \|_\infty +\|\bx_t -\by_t \|_\infty\right), 
	\end{align}
for all $t \geq 0$. Applying Theorem \ref{thm:diffDCandDC} yields the inequality 
\begin{align}
\| \bx_{t+1} - \by_{t+1} \|_{\infty} &\leq L_l \left( K_{\beta_t} \| \bx_{t} \|_\infty +\|\bx_t -\by_t \|_\infty\right). \label{5.3eq1}
\end{align}
Since $\bx_0 = \by_0$ and (\ref{5.3eq1}) holds for all $t \geq 0$, the
desired result now follows. 
\end{proof}
It should be noted that if the advancement operator given by $g(\cdot)$ has a Lipschitz constant less than one, then (\ref{eqn:lipbound}) is a convergent sum. Otherwise, even if the advancement operator is stable (e.g., finite difference equations of hyperbolic PDE's), the terms $L_l^{t-j+1}$ in (\ref{eqn:lipbound}) will grow exponentially and will no longer provide a meaningful bound. 

Simulations are fundamentally about approximation. For example, to solve a PDE one must use a finite difference or finite element method, which produces truncation error. Truncation error is typically the dominating error, and if the order of magnitude of the error is known, then one could find an approximate $\beta$ value for all iterations that will ensure the truncation error remains the dominating error instead of the error caused by the repeated application of ZFP compression. Here we present a result that provides a lower bound on the choice of $\beta$ to ensure \red{that truncation error, or any other type of error resulting from the simulation,} remains the dominant cost when using ZFP compressed data-types.

\begin{lemma}
	\label{lemma:truncerror}
	Suppose that $g(\cdot)$ is Lipschitz continuous with Lipschitz constant $L_l<1$. \red{Assume the sequence $\{\bx_t\}_{t=0}^\infty$ defined by (\ref{ZFPscheme}) is bounded. Assume $\epsilon \geq \frac{K_q\gamma}{1-L_l}$, where $q$ is the number of bits used in the negabinary representation in ZFP and $\gamma = \displaystyle \max_{0\leq j \leq t} {\|\bx_j\|_\infty}$. Then }
	\begin{align}\label{eqn:lipbound2}
	\| \bx_{t+1} - \by_{t+1} \|_{\infty} \leq \epsilon
	\end{align}
	if 
	\begin{align}
\tilde{\beta} \geq 1- \log_2\left(  \left(\frac{4}{15}\right)^d \frac{3\epsilon(1-L_l)}{8\gamma (2^d+1) L_l(1-L_l^{t+2}) }  \right)
	\end{align}
	and  $\tilde{\beta} \geq\beta_j$ for $1 \leq j \leq t$, where $\beta_j$ is the number of bit planes kept at step $j$. 
\end{lemma}
\begin{proof}
	\red{By dropping the lower order terms, we can approximate $K_{\beta_j }$ with respect to $\beta_j$ as $K_{\beta_j } \approx  \left(\frac{15}{4}\right)^d \frac{8}{3} (2^d+1)\epsilon_{\beta_j}$.}  Let $\gamma =\displaystyle \max_{0\leq j \leq t} {\|\bx_j\|_\infty}$ and $K_{max} =\displaystyle \max_j{K_{\beta_j}}$.  From Equation \ref{eqn:lipbound}, we have   
\begin{align}
\nonumber \| \bx_{t+1} - \by_{t+1} \|_{\infty} &\leq \gamma K_{max} \sum_{j = 0}^t L_l^{t-j+1}, \\
&\leq \gamma \left(\frac{15}{4}\right)^d \frac{8}{3} (2^d+1) \epsilon_{\beta'} \frac{L_l(1-L_l^{t+2})}{1-L_l}, 
\end{align} where $\beta' \leq \min_j \beta_j$. Let $\tilde{\beta}$ be the minimal $\beta'$ such that Equation \ref{eqn:lipbound2} is true. Then Equation \ref{eqn:lipbound2} implies 
\begin{align} \gamma\left(\frac{15}{4}\right)^d \frac{8}{3} (2^d+1) \epsilon_{\tilde{\beta}}\frac{L_l(1-L_l^{t+2})}{1-L_l} \leq \epsilon\end{align}
and the desired result follows. 
\end{proof}	

It is also of interest to consider convergence properties of the fixed-point iteration defined by (\ref{ZFPscheme}). As the operator $g \circ D \circ C : \mathbb{R}^n \to \mathbb{R}^n$ may not be differentiable or even continuous, we cannot directly apply the \red{fixed-point convergence theorem.} However, we are able to use the error bound established in Theorem \ref{thm:diffDCandDC} for the compression and decompression operators to prove a similar \red{convergence result when ZFP is used as a data-type.}


\begin{theorem} \label{ZFPFixedPointThm}
		Let $B$ be a compact subset of $\mathbb{R}^n$. Suppose that  $g(\cdot)$ is Lipschitz continuous on B with constant $L_l<1$. 
		 Furthermore, if $L_l \in \left(0,\frac{1}{1+\tilde{K}_\beta}\right)$, where $\tilde{K}_{\beta} =\displaystyle{ \max_{0\leq j\leq t}{K_{\beta_j}}}$, then 
	\begin{align}
	\lim_{t \to \infty}\| \bm{x}_t- \by^*  \| \leq \frac{L_l \tilde{K}_{\beta} \| \by^*  \|}{1 - \theta}, 
	\end{align}
	where $ \theta := L_l(1 + K_{\tilde{\beta}})$ and $\by^* $ is the unique fixed-point of $g(\cdot)$ in $B$.
\end{theorem}
\begin{proof}
	As $g$ is Lipschitz continuous with constant $L_l<1$ on a compact set there exists a unique point $\by^*  \in B$ such that $g(\by^* ) = \by^* $ \cite{burden2011numerical}. Thus, the scheme in (\ref{FPscheme}) will converge to $\by^* $ for all $\bm{y}_0 \in B $.
	
	
	Now observe that 
	\begin{align}
	\| \bm{x}_{t} - \by^*  \| &= \| g (DC (\bm{x}_{t-1})) - g (\by^* ) \| \nonumber, \\
	&\leq L_l \| DC (\bm{x}_{t-1}) - \by^*  \| \nonumber  \tag{Definition \ref{def:lip}}, \\
	&\leq L_l \| DC (\bm{x}_{t-1}) - \bm{x}_{t-1} \| + L_l \| \bm{x}_{t-1}  - \by^*  \| \nonumber,\\
		&\leq L_l \tilde{K}_{\beta} \| \bm{x}_{t-1} \| +L_l \|\bm{x}_{t-1} - \by^* \| \nonumber \tag{Theorem \ref{thm:diffDCandDC}}, \\
				&= L_l \tilde{K}_{\beta} \| \bm{x}_{t-1} -\by^*+\by^*\| + L_l\|\bm{x}_{t-1} - \by^* \| \nonumber ,\\
	&\leq \theta\| \bm{x}_{t-1} -\by^* \| +L_l \tilde{K}_{\beta}\| \by^*  \|,  \label{ZFPFixedPointThm1}
	\end{align}
	where  $ \theta := L_l (1 + K_{\tilde{\beta}})$. 
	Applying (\ref{ZFPFixedPointThm1}) recursively yields that 
	\begin{align}
	\| \bm{x}_{t} - \by^*  \| \leq \theta^t \| \bm{x}_{0} -\by^*  \| + L_l \tilde{K}_{{\beta}} \| \by^*  \| \sum_{j = 0}^{t-1} \theta^j,
	\end{align}
	 If $L_l \in \left(0,\frac{1}{1+\tilde{K}_\beta}\right)$, then   $0 < \theta < 1$. We conclude that
	\begin{align}
	\lim_{t \to \infty} \| \bm{x}_{t} - \by^*  \| \leq \frac{L_l \tilde{K}_{{\beta}} \| \by^*  \|}{1 - \theta}.
	\end{align}
\end{proof}

Theorem \ref{ZFPFixedPointThm} provides sufficient conditions under which the fixed-point for the sequence (\ref{ZFPscheme}) will be within some finite distance of the fixed-point for the sequence (\ref{FPscheme}). As another point of interest, we now consider the following result, which enumerates how many additional iterations are necessary for the sequence generated by (\ref{ZFPscheme}) to achieve the same accuracy as the sequence generated by (\ref{FPscheme}). 
\begin{theorem} \label{thm:extraits}
Let $B$ be a compact subset of $\mathbb{R}^n$. Suppose that  $g(\cdot)$ is Lipschitz continuous on B with constant $L_l<1$. Let $\by^* $ denote the fixed-point of (\ref{FPscheme}). Let the initial point of sequence (\ref{FPscheme}) and (\ref{ZFPscheme}) be equal, i.e., $\bx_0 = \by_0$, such that $\by_0 \in B$. Let $t\in \mathbb{N}$ define the number of iterations for the sequence (\ref{FPscheme}) and let $m \in \mathbb{N}$ be the number of additional iterations for sequence (\ref{ZFPscheme}), i.e., $t+m$ iterations.  Define $\tilde{K}_{{\beta}} =\displaystyle \max_j K_{\beta_{j}}$ for all $j = 1, \dots, t+m$ and  $C: =\frac{\tilde{K}_{{\beta}}}{1-L_l}$. If $\frac{\tilde{K}_{{\beta}}}{(1-L_l)}\leq L_l^{t}$ and 
\begin{align} 
m &\geq  \log_{L_l} \frac{ L_l^{t+1}-C}{1-C}  - (t+1),
\end{align}
then
\begin{align} 
\| \bx_{t+m+1} - \by^* \|_\infty \leq L_l^{t+1}\|\by_0 -\by^* \|_\infty.  
\end{align} 
\end{theorem}
\begin{proof}Let $\by^*$ satisfy $\by^*  = g(\by^* )$. Using standard fixed-point analysis techniques, we have that $\|\by_{t+1} -\by^* \|_\infty  \leq  L_l^{t+1}\|\by_0 -\by^* \|_\infty$.
	Similarly we can find an expression for $\| \bx_{t+m+1} -\by^* \|_\infty$ as follows:
		\begin{align}
	\nonumber	\| \bx_{t+m+1} - \by^* \| _\infty &=  \|g (D C (\bx_{t+m}))  - g (\by^*) \| _\infty, \\
\nonumber		&\leq L_l \|g(\bx_{t+m}) -  g(\by^*) \|_\infty  + L_l K_{\beta_t} \|\bx_{t+m} \|_\infty ,   \\ 
		& \leq L_l^{t+m+1} \|\bx_0 -  \by^* \|_\infty  + \sum_{j = 0}^{t+m} L_l^{t+m-j} K_{\beta_j }\|\bx_j\| _\infty.
		\end{align}
Now define $\displaystyle \alpha := \displaystyle\max \left\{ \displaystyle \max_{1 \leq j \leq t + m + 1} \|\bx_j\|_\infty, \|\by_0-\by^*\|_\infty \right\}$. By our choice that $\bx_0 = \by_0$, we have $\|\bx_0 - \by^*\|_\infty = \|\by_0-\by^*\|_\infty$, which yields
		\begin{align}
		\| \bx_{t+m+1} - \by^* \|_\infty &\leq \alpha \left( L_l^{t+m+1}  + \tilde{K}_{{\beta}}\sum_{j = 0}^{t+m} L_l^{t+m-j} \right).
		\end{align}
		As we wish to determine the additional number of iterations $m$ such that $\bm{x}_{t+m+1}$ is sufficiently close to $\bm{y}^*$, assume that $\alpha \left( L_l^{t+m+1}  + \tilde{K}_{{\beta}}\sum_{j = 0}^{t+m} L_l^{t+m-j} \right) \leq L_l^{t+1} \|\by_0 -\by^* \|_\infty$. As $\|\by_0 -\by^* \|_\infty \leq \alpha$, by definition, we can consider the simplified inequality
		\begin{align}
		L_l^{t+m+1} + \tilde{K}_{{\beta}}\sum_{j = 0}^{t+m} L_l^{t+m-j} &\leq L_l^{t+1},
		\end{align}
		which, when solved for $m$, yields
	\begin{align}
	m  &\geq \log_{L_l} \left(\frac{L_l^{t+1}  - C}{1-C}\right) -(t+1) .  \\ 
	\end{align}
	The assumption  $\frac{\tilde{K}_{{\beta}}}{(1-L_l)}\leq L_l^{t+1}$ ensures that $\frac{L_l^{t+1}  - C}{1-C} >0$, which implies $\log_{L_l} \left(	\frac{L_l^{t+1}  - C}{1-C}\right)$ is real valued. \red{Finally, we need to show that $m>0$.} By way of contradiction, assume $\log_{L_l} \left(	\frac{L_l^{t+1}  - C}{1-C}\right) -(t+1) <0$. Then 
	\begin{align}
	\log_{L_l} \left(	\frac{L_l^{t+1}  - C}{1-C}\right)  <t+1, \ \\
\nonumber	\frac{L_l^{t+1}  - C}{1-C} > L_l^{t+1} , \\ 
	L_l^{t+1}  - C > L_l^{t+1}(1-C)
	\end{align}
which implies $L_l^{t+1}>1$, a contradiction to our hypothesis that $L_l<1$. Thus, 
	\begin{align}
	\log_{L_l} \left(\frac{L_l^{t+1}  - C}{1-C}\right) -(t+1) >0, \end{align}
	\red{implying $m>0$.}
\end{proof}

Given $t$ iterations, Theorem \ref{thm:extraits} lets us determine how many extra iterations are needed when a compressed ZFP data-type is used to ensure the \rededit{error} is no larger than the worst case error bound for the traditional fixed-point method. It should be noted that the number of iterations may be many fewer as we have only determined a meaningful bound with respect to the traditional error bound for stationary iterative methods. Additionally, there is no guarantee that we can get as close to the original solution as desired due to the finite precision yielding from the lower bound from the choice of $\tilde{K}_\beta$.

\subsection{Kreiss Bounded Linear Advancement Operator}
\red{Full discretization of many time-dependent PDE's take the from of (\ref{FPscheme}), where the iteration index plays the role of the time-step.} A common property in many \red{time-stepping and fixed-point methods} is the use of a linear advancement operator, say $A \in \mathbb{R}^{n \times n}$. \red{This section will analyze iterative methods using compressed ZFP data-types with respect to a bounded linear advancement operator. However, the action of (de)compression is a nonlinear operation, similar to traditional floating-point computations. As such, a similar analysis to floating-point error analysis will} be used when considering the propagation of an error caused by ZFP in conjunction with a linear operator. The (de)compression operation produces a small perturbation from the original data, e.g., for $\bx \in \mathbb{R}^n$ there exists $\delta \bx \in \mathbb{R}^n$ such that 
	\begin{align}
	 DC({\bx})= {\bx} +{\delta\bx}.\end{align}
	  From Theorem \ref{thm:diffDCandDC}, we can quantify the magnitude of the perturbation vector $\delta \bx$. 
\begin{lemma}[Nonlinear Error] \label{lemma:nonlinearerror}
	Assume ${\bx} \in \mathbb{R}^{n}$ with ${\bx} \neq \bfz$. Then there exists ${\delta\bx} \in \mathbb{R}^{n} $ such that
		\begin{align}
		 DC({\bx}) = {\bx} +{\delta\bx} \ \ \ \ {\text{and}} \ \ \ \ \|{\delta \bx}\|_\infty \leq K_\beta \|{\bx}\|_\infty.
		 \end{align}
\end{lemma}

Now suppose that we are given the initial value, $\by_0$, and we generate the sequences $\{ {\by}_t \}_{t = 0}^{\infty}$ and $\{ {\bx}_t \}_{t = 0}^{\infty}$ by  
\begin{align*}
	{\by}_{t+1} = A({\by}_t) \quad {\text{and} } \quad \bx_{t+1} = A ( {D} {C} ({\bx}_{t})
),
\end{align*} 
where $A: \mathbb{R}^n \to \mathbb{R}^n$ is used to denote a linear advancement operator and ${\bx}_0= {\by}_0$. Theorem \ref{thm:boundedIterative} provides an error bound between the sequences  $\{ \by_{j} \}_{j = 0}^{t+1}$ and $\{ \bx_{j} \}_{j = 0}^{t+1}$ at some iterate $t+1$, i.e., a bound on the error introduced by applying inline ZFP (de)compression after $t$ time steps.
\begin{theorem}[Bounded Linear Advancement Operator] \label{thm:boundedIterative}
	Suppose that $A$ satisfies the hypothesis of the Kreiss Matrix Theorem \cite{kreiss} with Kreiss constant $L_k$. Then
	\begin{align}
		\| \bx_{t+1} - \by_{t+1} \|_{\infty} \leq L_k \sum_{j = 0}^{t} K_{\beta_j} \| \bx_j \|_\infty ,
	\end{align}
	where $\beta_j$ is the number of bit planes kept at step $j$.
\end{theorem}
 \begin{proof}
By the Kreiss Matrix Theorem, there exists a constant $L_k < \infty$ such that $\| A^i \|_2 \leq L_k$ for all $i \in \mathbb{N}$. 	From Lemma \ref{lemma:nonlinearerror}, ${\bx}_{t+1}$ can be decomposed as  
\begin{align}
\bx_{t+1} =  A \left( DC( \bx_{t}) \right) = A^{t+1} \bx^0 + \sum_{j=0}^t A^{t-j+1} \delta \bx_j.
\end{align}
By our choice of $\bx_0 = \by_0$, it now follows that 
\begin{align}
\bx_{t+1} - \by_{t+1} = \sum_{j=0}^t A^{t-j+1} \delta \bx_j,
\end{align}
since $\by_{t+1} = A^{t+1} \bx_0$. Finally, as $\| A^i \|_\infty \leq \| A^i \|_2 \leq L_k$ for all $i \in \mathbb{N}$, we conclude that
\begin{align}
\| \bx_{t+1} - \by_{t+1} \|_\infty &= \left\| \sum_{j=0}^t A^{t-j+1} \delta \bx_j \right\|_\infty \leq \sum_{j=0}^t \| A^{t-j+1} \|_\infty \| \delta \bx_j\|_\infty \leq L_k \sum_{j=0}^t K_{\beta_j} \| \bx_j \|_\infty.
\end{align}
\end{proof}
Similarly, Lemma \ref{lemma:truncerror} can be generalized for Kreiss bounded linear operators. 
\begin{lemma}\label{lemma:truncerror2}
	Suppose that $A$ is Kreiss bounded with Kreiss constant $L_k$, the sequence $\{\bx_t\}_{t=0}^\infty$ defined by (\ref{ZFPscheme}) is bounded, and $\epsilon \geq \frac{K_q\gamma}{1-L_k}$, where $q$ is the number of bits used in the negabinary representation in ZFP. Then 
	\begin{align}
	\| \bx_{t+1} - \by_{t+1} \|_{\infty} \leq \epsilon,
	\end{align}
provided $\displaystyle \tilde{\beta} \geq 1- \log_2\left(  \left(\frac{4}{15}\right)^d \frac{3\epsilon}{8(2^d+1) \gamma L_k t} \right)$ and $\tilde{\beta} \geq \beta_j$ for all j, where $\beta_j$ is the number of bit planes kept at step $j$, and where $\gamma = \displaystyle\max_{0\leq j\leq t}{\|\bx_j\|_\infty}$. 
\end{lemma}
\begin{proof}
	Similar to proof of Lemma \ref{lemma:truncerror}. 
	\end{proof}
The implication of Theorems \ref{thm:lip} and \ref{thm:boundedIterative} is that the error introduced by repeated compression and decompression in conjunction with an advancement operator can be bounded by factors dependent only on the advancement operator, the error introduced by ZFP, and the dimensionality of the data.  Lemmas \ref{lemma:truncerror} and \ref{lemma:truncerror2} can be used to determine suitable choices for magnitude of various  parameters  in ZFP to ensure the error, relative to the solution represented in IEEE floating point, is within some tolerance. 

\subsection{Successive Displacement Fixed-point Methods}
 \rededit{So far we have assumed that the solution state is completely decompressed before applying the advancement operator simultaneously to attain the next iterate. However, for fixed-point methods we could easily apply a Gauss-Seidel type of operator, where each new element of the solution state is used in the computation of the next element. \red{That is, depending on the ordering of the updates, the components of the new iterate will change. This is also known as successive displacement.} If ZFP is used in conjunction with successive displacement methods, the additional compression error from each element will propagate into the computation of the next element during the iteration. In an effort to bound the error in a successive displacement setting, we will first study the simplest case for applying ZFP, i.e., where the successive displacement updates are aligned with our ZFP partitioned blocks.} First, we will assume the solution states are formatted as a one dimensional problem (i.e., $d = 1$). Additionally, to simplify the analysis, we assume that the linear advancement operator, $B$, is partitioned into $4$ by $4$ blocks so that 

\begin{align*}
B= \begin{bmatrix}
B_{11} & B_{12} & B_{13} & \cdots & B_{1n_1} \\ 
B_{21} & B_{22} & B_{23} & \cdots & B_{2n_1} \\ 
\vdots &  & \ddots & & \vdots  \\ 
B_{n_11 } & B_{n_12}  & \cdots & & B_{n_1n_1}  \\ 
\end{bmatrix}, 
\end{align*}
where $B \in \mathbb{R}^{4n_1\times4n_1}$ and $B_{ij} \in \mathbb{R}^{4 \times 4}$. \rededit{If the operator $B$ cannot be partitioned exactly into 4 by 4 blocks, one can pad the operator with up to three additional ghost variables, by including additional zeros in the off diagonal entires and a one along the diagonal. By aligning the   structure of the advancement operator with the ZFP blocks, the successive displacement steps can be applied without overlapping the advancement operators with the ZFP blocks, allowing for a simplified error analysis.} Assuming that each element is updated by successive displacement, we generate the sequences $\{ {\by}_t \}_{t= 0}^{\infty}$ and $\{ {\bx}_t \}_{t = 0}^{\infty}$ by  
\begin{align}
\by_{i,t+1} = \sum_{j = 1}^{i-1} B_{ij}\by_{j,t+1} +\sum_{j=i}^{n_1} B_{ij}\by_{j,t} 
\end{align}
and 
\begin{align}\label{eqn:successivecomp}
\bx_{i,t+1} = \sum_{j = 1}^{i-1} B_{ij}DC(\bx_{j,t+1}) +\sum_{j=i}^{n_1}  B_{ij}DC(\bx_{j,t}), 
\end{align}
respectively, where $\by_{j,t}$ denotes the $j$ block of the solution state $\by_t$ at iterate $t$ and where we assume $\bx_0 = \by_0$.  Using Theorem \ref{thm:boundedIterative}, Equation (\ref{eqn:successivecomp}) can be rewritten as 
\begin{align}
\bx_{i,t+1} = \sum_{j = 1}^{i-1} B_{ij}(\bx_{j,t+1}+\delta\bx_{j,t+1}) +\sum_{j=i}^{n_1} B_{ij}(\bx_{j,t}+\delta\bx_{j,t}),
\end{align}
where $\|\delta \bx_{i,t}\|_\infty \leq \max_j \|\delta \bx_{j,t}\|_\infty \leq\|\delta \bx_{t}\|_\infty \leq K_{\beta_t}\|\bx_t\|_\infty$.  \rededit{Lastly, we will also assume the advancement operator $B$ is block diagonally dominant (\cite{Block}, Definition 1), defined as 
\begin{align} \label{eqn:blockmat}
\|B_{kk}^{-1}\|_\infty^{-1} \geq \sum_{j = 1, j\neq k}^{n_1} \|B_{jk}\|_\infty
\end{align} 
for all $1\leq k\leq n_1$, implying a stationary fixed-point method with a Lipschitz constant less than one. Additional analysis is needed to generalize Theorem \ref{thm:block} for more general operators and for $d >1$.\footnote{\rededit{At first glance, Assumption~\ref{thm:block} seems highly restrictive.  However, we claim this is a less restrictive assumption when blocks of size $4^d$ are chosen spatially (internally connected), much as they often are for block Gauss-Seidel,  than common classical theory that assumes symmetric positive definiteness or strict diagonal dominance. Padding with columns/rows of the identity is used for problems when it is not easy to find spatially connected blocks of exactly size $4^d$.   For a specific class of advancement operator, meeting the assumption could be analyzed, but this is beyond the scope of this paper. 
%
	}} } 
Finally, it will be useful for the following analysis to define
\begin{align}
\alpha_i := \sum_{j=1}^{i-1} \|B_{ij}\|_\infty, \qquad \qquad \gamma_i :=  \sum_{j=i+1}^{n_1}  \|B_{ij}\|_\infty, \qquad \qquad \mu_i:= \alpha_i+\gamma_i,
\end{align}	
and $\displaystyle \mu := \max_{1\leq i\leq {n_1}  } \mu_i$. Note that $\alpha_1 = \gamma_{n_1}  = 0$.
\begin{theorem} \label{thm:block}
Assume that $B$ is block diagonally dominant. If $\max_k \|B_{kk}^{-1}\|_\infty^{-1} <1$ for all $k$, then 
\begin{align}
\|\bx_{t+1} - \by_{t+1}\|_\infty \leq \eta  \sum_{j=0}^{t+1} \nu^{t-j+1} K_{\beta_j}\|\bx^j\|_\infty ,  \label{eqn:block}
\end{align}
where $\displaystyle \nu = \max_{1\leq i\leq {n_1}  }  \frac{\gamma_i}{1-\alpha_i}$ and $\displaystyle \eta = 1 + \max_{1\leq i \leq {n_1} } \frac{\alpha_i}{1-\alpha_i}.$ 
\end{theorem}
\begin{proof}
As $B$ is block diagonally dominant, we have
	\begin{align}
	\left (\|B_{jj}^{-1}\|_\infty \right)^{-1}  \geq \sum_{k = 1, k\neq j} \|B_{jk}\|_\infty = \mu_j   \quad \forall \quad j , 
	\end{align} 
implying that $\mu <1 $. Let $\hat{i} \in \{1,\dots,{n_1} \}$ and $\hat{k} \in \{1,\dots,4\}$ be the indices such that \begin{align} \left|\left[\bx_{\hat{i},t+1} - \by_{\hat{i},t+1}\right] _{\hat{k}} \right|=\max_{i,k} \left|\left[\bx_{i,t+1} - \by_{i,t+1}\right]_k \right| =  \|\bx_{t+1} - \by_{t+1}\|_\infty.\end{align} Then
			\begin{align}
			 \|\bx_{t+1} - \by_{t+1}\|_\infty &=\max_{i,k} |\left[\bx_{i,t+1} - \by_{i,t+1}\right]_k| ,\\
		&=\max_{i} \left | \left[\sum_{j=1}^{i-1} B_{ij}( \bx_{j,t+1} -\by_{j,t+1}+\delta \bx_{j,t+1}) + \sum_{j=i}^{n_1}  B_{ij}(\bx_{j,t} - \by_{j,t}+\delta \bx_{j,t})\right ]_{\hat{k}}\right |, \\
	&\leq  \sum_{j=1}^{{\hat{i}}-1} \|B_{{\hat{i}}j}\|_\infty \| \bx_{t+1} -\by_{t+1}+\delta \bx_{t+1}\|_\infty + \sum_{j={\hat{i}}}^{n_1}  \|B_{{\hat{i}}j}\|_\infty\|\bx_t - \by_t+\delta \bx_t\|_\infty , \\
	&\leq \frac{\gamma_{\hat{i}}}{1-\alpha_{\hat{i}}}\|\bx_t - \by_t\|_\infty   +\frac{\alpha_{\hat{i}}}{1-\alpha_{\hat{i}}} \|\delta \bx_{t+1}\|_\infty +\frac{\gamma_{\hat{i}}}{1-\alpha_{\hat{i}}} \|\delta \bx_t\|_\infty ,\\
	&\leq \left(1+ \frac{\alpha_{\hat{i}}}{1-\alpha_{\hat{i}}} \right) \sum_{j=0}^{t+1} \left( \frac{\gamma_{\hat{i}}}{1-\alpha_{\hat{i}}} \right)^{t-j+1} K_{\beta_j}\|\bx_j\|_\infty, \\  	
		 & \leq\eta  \sum_{j=0}^{t+1} \nu^{t-j+1} K_{\beta_j}\|\bx_j\|_\infty \label{eqn:boundsuccessive}.
			\end{align}
		Hence, for each $i$,		
		\begin{align}
		\alpha_i +\gamma_i - \frac{\gamma_i}{1-\alpha_i} = \frac{\alpha_i}{1-\alpha_i} (1 - (\alpha_i+\gamma_i)) \geq  \frac{\alpha_i}{1-\alpha_i} (1 - \mu) \geq 0, 
		\end{align}
		which yields that $\nu \leq \mu <1$. Thus,  equation (\ref{eqn:boundsuccessive}) is bounded provided that the solution $\|\bx_t\|_\infty$ is bounded for all time steps. 
\end{proof}

In the previous section, we were assuming exact arithmetic when applying the advancement operators, however, \red{traditional floating-point error analysis assumes  that round-off will occur in both the number representation and the arithmetic operations}. In the next section, we will provide a detailed discussion of the forward and backward error analysis of the accumulated round-off error of both floating-point and ZFP round-off error. \red{The analysis will follow traditional floating-point analysis, allowing us to study the relationship between the two types of error. All our analysis thus far, have assumed that the error present in a simulation is aggregated together. Instead, we can further break down the round-off error into the floating-point arithmetic error and ZFP compressed round-off error.  }

\subsection{Forward and Backward Error Analysis for Linear Stationary Iterative Methods} This section presents the forward and backward error analysis for linear stationary iterative methods with compressed data-types for solving the system $A\bx = \bb$, where $A \in \mathbb{R}^{n_1\times n_1}$  is nonsingular and $\bb\in \mathbb{R}^{n_1}$. Decompose $A$ such that $A = M-N$ and $M$ is nonsingular. A stationary iterative method has the form 
\begin{align} M\by_{t+1} = N\by_{t} +\bb. \label{schemeSatationary} \end{align}
 Define the matrices $G = M^{-1}N$ and $H = NM^{-1}$. Define $\sigma := \|G\|_\infty$ and $\omega := \|H\|_\infty$ and assume the spectral radius of $G$ and $H$ are bounded by one, i.e., $ \sigma <1$. As $H$ and $G$ are similar matrices, we have $ \omega= \sigma $.  Since $\sigma<1$, one can show in exact arithmetic the sequence $\{\by_t\}_{t=0}^\infty$ converges for an appropriate starting vector, $\by_0$. Detailed forward and backward error analysis for the sequence $\{\by_t\}_{t=0}^\infty$ can be found in \cite{higham2002accuracy}. We would like to perform a similar analysis for the sequence $\{\bx_t\}_{t=0}^\infty$ defined by 
\begin{align}
M\bx_{t+1}  = N(DC(\bx_{t})) +\bb, \label{ZFPschemeSatationary}
\end{align} 
where the solution vector $\bx_t$ is compressed and decompressed before applying the update and where $\bx_0=\by_0$. Traditionally, in perturbation theory, the computed vectors satisfy the equality 
\begin{align}
(M+\Delta M_{t+1})\by_{t+1} =  N\by_{t}+\bb+\bm{f}_t, \label{eqn:forward_exact}
\end{align}
where $\Delta M_t$ and $\bm{f}_t$ account for floating-point arithmetic errors in solving the linear system and forming right hand side, $N\by_t +\bb$. Typically, $\|\Delta M_t\|_\infty \leq \epsilon_k b_{n_1}' \|M\|_\infty$ and $\|\bm{f}_t\|_\infty \leq   \epsilon_k b_{n_1} \left(\|N\|_\infty \|\by_t\|_\infty +\|\bb\|_\infty \right)$, for some constants $b_{n_1}$ and $b_{n_1}'$ (see \cite{higham2002accuracy}, Section 7 and Section 17 for details.) For simplicity, define $c_{n_1} = \max\{b_{n_1}, b_{n_1}'\}$. The magnitude of $c_{n_1}$ represents the accumulation of floating-point arithmetic errors, which is dependent on the size of the matrix as well as the conditioning of $M^{-1}$. In most cases, it can be assumed that \red{the order of magnitude of} $c_{n_1} $ is with respect to ${n_1}$, where $n_1$ is the number of unknowns in linear system. \red{In \cite{wilkinson}, the constant $c_{n_1}$ is dependent on the type of norm, however, it is shown to always be greater than one} \cite{Nicholas1993ComponentwiseEA}.  Similarly, the sequence $\{\bx_t\}_{t=0}^\infty$ will satisfy the equality
\begin{align}
(M+\Delta \tilde{M}_{t+1})\bx_{t+1} =  N\left(DC(\bx_{t})\right)+\bb + \tilde{\bm{f}}_t.  \label{eqn:forward}
\end{align}
Note that $\Delta \tilde{M}_{t+1}$ and $\tilde{\bm{f}}_t$ in Equation (\ref{eqn:forward}) differ slightly from the corresponding quantities in Equation (\ref{eqn:forward_exact}). It follows that there exists a constant $\tilde{c}_{n_1}$ such that $\|\Delta \tilde{M}_t\|_\infty \leq \epsilon_k \tilde{c}_{n_1} \|M\|_\infty$ and $\|\tilde{\bm{f}}_t\|_\infty \leq   \epsilon_k \tilde{c}_{n_1} \left(\|N\|_\infty \|\bx_t\|_\infty +\|\bb\|_\infty \right)$. However, we now have  
 \begin{align}
 \| \tilde{\bm{f}}_t\|_\infty &\leq \epsilon_k\tilde{c}_{n_1}\left(\|N\|_\infty \|DC(\bx_{t})\|_\infty +\|\bb\|_\infty \right ),\\
 &\leq  \epsilon_k \tilde{c}_{n_1} \left( (1+K_{\beta_t}) \|N\|_\infty\|\bx_{t}\|_\infty +\|\bb\|_\infty \right ) \tag{Theorem \ref{thm:diffDCandDC}}.
 \end{align}
%
%
%
Using Lemma \ref{lemma:nonlinearerror}, Equation (\ref{eqn:forward}) can be written as 
\begin{align}
M\bx_{t+1} =  N\bx_{t}+\bb + \zeta_t + N\delta \bx_t,  
\end{align}
where $\zeta_t =\tilde{\bm{f}}_t-  \Delta \tilde{M}_{t+1} \bx_{t+1}$. Let $\bx^*$ represent the exact fixed-point solution of Equation (\ref{schemeSatationary}) and assume $\bx^*\neq \bfz$. Then define $\gamma =\displaystyle \sup_i \frac{\|\bx_i\|_\infty}{\|\bx^*\|_\infty}$. Recall the constant $K_{\beta_t}$, defined in Equation (\ref{eqn:kbeta}), bounds the relative error caused by ZFP compression at iterate $t$  and $\tilde{K}_{\beta}:=\displaystyle \max_iK_{\beta_i}$. Then, we have  
\begin{align} 
\|\zeta_t\|_\infty &\leq\|\tilde{\bm{f}}_t\|_\infty + \|\Delta\tilde{M}_{t+1}\|_\infty \|\bx_{t+1}\|_\infty, \\ 
\nonumber &\leq\epsilon_k  \tilde{c}_{n_1} \left((1+K_{\beta_t})  \|N\|_\infty \|\bx_n\|_\infty +\|M\|_\infty + \|A\bx^*\|_\infty \right) ,\\ 
 \nonumber &\leq\epsilon_k  \tilde{c}_{n_1}\left( (1+\gamma)\left(\|N\|_\infty +\|M\|_\infty\right) +K_{\beta_t}\gamma \|N\|_\infty \right) \|\bx^*\|_\infty. \label{zetabound}
\end{align}
By applying the recurrence relation we obtain 
\begin{align}
\bx_{t+1} = G^{t+1} \bx_0+ \sum_{j = 0}^tG^{t-j}M^{-1}(\bb + \zeta_j)  + \sum_{j = 0}^t G^{t-j+1}  \delta \bx_j.  \label{eqn:xn}
\end{align}
Since $\bx^*$ is the fixed-point, we have  
\begin{align}
\bx^* = G^{t+1} \bx^*+ \sum_{j = 0}^tG^{t-j}M^{-1}\bb,
\end{align}
and the error at $t+1$ is represented as 
\begin{align}
\bx_{t+1} - \bx^* = G^{t+1} (\bx_0 - \bx^*)+ \sum_{j = 0}^tG^{t-j}M^{-1}\zeta_j  + \sum_{j = 0}^t G^{t-j+1} \delta \bx_j.
\end{align}
Define $\theta_{\text{f},t} =  \frac{ 1-\sigma^{t+1} }{1-\sigma}$. By applying norms and using the inequality (\ref{zetabound}),
\begin{align}
\nonumber \|\bx_{t+1} - \bx^*\| &\leq  \sigma^{t+1} \|\bx_0 - \bx^*\|_\infty + \max_i \|\zeta_i\|_\infty\|M^{-1}\|_\infty \sum_{j = 0}^t\|G^{t-j}\|_\infty  \\
\nonumber & \quad \quad \quad \quad \quad \quad \quad \quad \quad  \quad \quad \quad  \quad \quad \quad \quad  \quad \quad  \quad \quad  \quad \quad  \quad + \tilde{K}_\beta \gamma \left(\sum_{j = 0}^t  \|G^{t-j+1}\|_\infty \right)\|\bx^*\|_\infty,\\  
\nonumber&\leq  \sigma^{t+1} \|\bx_0 - \bx^*\|_\infty +\left(   \max_i \|\zeta_i\|_\infty\|M^{-1}\|_\infty  + \tilde{K}_\beta \gamma \|G\|_\infty \right)\left(\sum_{j = 0}^t  \|G^{t-j+1}\|_\infty \right) \|\bx^*\|_\infty,\\  
&\leq  \sigma^{t+1} \|\bx_0 - \bx^*\|_\infty  + \Bigg( \underbrace{\epsilon_k \tilde{c}_{n_1} (1+\gamma)\|M^{-1}\|_\infty\left(\|N\|_\infty+\|M\|_\infty\right) }_{\text{Traditional Floating-Point Error}} \\ 
\nonumber&  \quad \quad  \quad \quad  \quad \quad  \quad \quad  \quad \quad  \quad \quad  \quad \quad  \quad \quad  \quad \quad+  \underbrace{\tilde{K}_\beta \gamma\left(\epsilon_k  \tilde{c}_{n_1} \|M^{-1}\|_\infty\|N\|_\infty + \sigma\right)}_{\text{ZFP Round-off}} \Bigg)\theta_{\text{f},t} \|\bx^*\|_\infty.
\end{align}
\noindent 
The bound can be seen as a sum of two error terms, one reflecting the traditional error caused by floating-point arithmetic,
\begin{align}
\Phi_{\text{FP}} :=  \epsilon_k \tilde{c}_{n_1} (1+\gamma)\|M^{-1}\|_\infty\left(\|N\|_\infty+\|M\|_\infty\right),
\end{align}
 and the other term reflecting the error caused by ZFP round-off,
 \begin{align}
\Phi_{\text{ZFP}} :=  \tilde{K}_\beta \gamma\left(\epsilon_k \tilde{c}_{n_1} \|M^{-1}\|_\infty\|N\|_\infty + \sigma\right). 
 \end{align}
To ensure the additional round-off error caused by applying ZFP repeatedly is less than the traditional forward error, we must have $\Phi_{\text{FP}} \geq \Phi_{\text{ZFP}}$. Lemma \ref{lemma:forwarderrorbeta} provides the minimum value of $\tilde{\beta}$ so that $\Phi_{\text{FP}} \geq \Phi_{\text{ZFP}}$. 
\begin{lemma}\label{lemma:forwarderrorbeta}
	Let $k \in \mathbb{N}$ represent the traditional floating-point precision and define $\kappa_\infty(M) := \|M\|_\infty \|M^{-1}\|_\infty$ as the condition number of $M$. Assume $ \tilde{c}_{n_1} \leq \frac{1}{\epsilon_k \kappa_\infty(M) }$, then  $\Phi_{\text{FP}}\geq  \Phi_{\text{ZFP}}$ if 
	\begin{align}
	\tilde{\beta} \geq k - \log_2\left( \frac{1+\sigma}{\sigma}  \tilde{c}_{n_1}\left(\frac{4}{15}\right)^d \left(\frac{3}{8(2^d+1)}\right) \right) \label{eqn:betaforwarderror}
	\end{align}
	provided $\tilde{\beta} \geq \beta_j$ for all j, where $\beta_j$ is the number of bit planes kept at step $j$.
\end{lemma}
\begin{proof}
By applying the inequality $\|AB\| \leq \|A\| \|B\|$ for matrices $A$ and $B$ and the assumption $\tilde{c}_{n_1} \leq \frac{1}{\epsilon_k \kappa_\infty(M)}$, we have that 
 \begin{align}
 \Phi_{\text{FP}}& \geq  \epsilon_k \tilde{c}_{n_1} \gamma \|M^{-1}\|_\infty\left(\|N\|_\infty+\|M\|_\infty\right) \geq   \epsilon_k \tilde{c}_{n_1} \gamma \left(\|G\|_\infty+1\right) \geq   \epsilon_k \tilde{c}_{n_1} \gamma \left(\sigma+1 \right) 
 \end{align}
 and 
  \begin{align}
 \Phi_{\text{ZFP}} &\leq \tilde{K}_\beta \gamma\left(\epsilon_k \tilde{c}_{n_1} \|M^{-1}\|_\infty \|N\|_\infty + \sigma\right) \leq   \tilde{K}_\beta \gamma\left(\epsilon_k \tilde{c}_{n_1} \|M^{-1}\|_\infty\|M\|_\infty \|M^{-1}N\|_\infty + \sigma\right ) \leq   2\tilde{K}_\beta \gamma \sigma .  
 \end{align}
 Thus, if 
  \begin{align} 
2  \tilde{K}_\beta \gamma \sigma &\leq \epsilon_k \tilde{c}_{n_1} \gamma \left(\sigma+1 \right), 
 \end{align}
 it follows that $\Phi_{\text{FP}}\geq  \Phi_{\text{ZFP}}$. Without loss of generality, $\tilde{K}_\beta = \left(\frac{15}{4}\right)^d\frac{8}{3} (2^d+1) \epsilon_{\tilde{\beta}}$, and the desired result follows. 
 \end{proof}
To serve as a visual aid and to help in the interpretation of Lemma \ref{lemma:forwarderrorbeta}, consider the contour plot in Figure \ref{fig:contour} of $\tilde{\beta}$ given by Equation (\ref{eqn:betaforwarderror}) as a function of  $\tilde{c}_{n_1}$ and $\sigma$. We would expect that, as the conditioning and the size of the matrix $A$ grows, represented by the growth of $\tilde{c}_{n_1}$, ZFP is able to use a lower $\tilde{\beta}$ to compress more aggressively and still have the traditional floating-point arithmetic error dominate ZFP round-off error. 
We see from Figure \ref{fig:contour} that this is indeed true. If we have a poorly conditioned linear system, it is known that the floating-point round-off error will accumulate significantly. Lemma \ref{lemma:forwarderrorbeta} and Figure \ref{fig:contour} show that if $\tilde{\beta}$ is chosen appropriately, ZFP can represent the solution with fewer bits while retaining double-precision accuracy. For a well conditioned system, typically, the floating-point round-off error will remain small, and Lemma \ref{lemma:forwarderrorbeta} indicates that ZFP will be unable to compress aggressively if the only condition is that the ZFP error remains below the floating-point round-off error. However, in most computational simulations, errors other than floating-point round-off error will dominate, e.g., truncation error. In this situation, Lemma \ref{lemma:truncerror} and \ref{lemma:truncerror2} would be more useful, as $\tilde{\beta}$ can be determined from an approximation of the dominating simulation error. 

We will now consider the backwards error of the sequence $\{\bx_t\}_{t=0}^\infty$. The backward error (\cite{higham2002accuracy}, Chapter 7) with respect to $A$ and $\bf{b}$ is given by 
\begin{align}
\eta_{A,\bb}(\by)  = \frac{\|\br\|_\infty}{\|A\|_\infty\|\by\|_\infty+\|\bb\|_\infty} , 
\end{align}
where $\br$ is the residual defined by $\br := \bb- A\by$, which contains both traditional floating-point arithmetic error and ZFP round-off error. Using (\ref{eqn:xn}) we have 
\begin{align}
\br_{t+1} &= AG^{t+1}A^{-1} (\bb - A\bx_0) -  \sum_{j = 0}^tAG^{t-j}M^{-1}(\bb + \zeta_j)  - \sum_{j = 0}^tA G^{t-j+1}  \delta \bx_j ,\\ 
\nonumber &= H^{t+1}(\bb - A\bx_0) - \sum_{j = 0}^tH^{t-j}(I-H)\zeta_j  + \sum_{j = 0}^tH^{n-j+1}A \delta \bx_j. \\
\end{align}
Define $\theta_{\text{b},t} := \frac{1-\omega^{t+1}}{1-\omega}$. Then, by applying norms 
\begin{align}
\nonumber \|\br_{t+1}\|_\infty &\leq \omega^{t+1}\|\br_0\|_\infty+\Bigg(\underbrace{\|I-H\|_\infty\epsilon_k\tilde{c}_{n_1}(1+\gamma)\left(\|M\|_\infty+\|N\|_\infty\right)}_{\text{Traditional Floating-Point Error}}\\
 & \quad \quad  \quad \quad  \quad \quad  \quad \quad  \quad \quad  \quad \quad  \quad \quad  +  \underbrace{ \tilde{K}_\beta\gamma\left( \epsilon_k\tilde{c}_{n_1}\|I-H\|_\infty\|N\|_\infty+\|A\|_\infty  \omega\right)}_{\text{ZFP Round-off}}\Bigg) \theta_{\text{b},t}\|\bx^*\|_\infty.
\end{align}
Similar to the forward error bound, the backwards error bound can be seen as a sum of two error terms, one reflecting the traditional error caused by floating-point arithmetic,
\begin{align}
\Omega_{\text{FP}} := \|I-H\|_\infty\epsilon_k\tilde{c}_{n_1}(1+\gamma)\left(\|M\|_\infty+\|N\|_\infty\right),
\end{align}
and the other error caused by ZFP round-off,
\begin{align}
\Omega_{\text{ZFP}} :=\tilde{K}_\beta\gamma\left( \epsilon_k\tilde{c}_{n_1}\|I-H\|_\infty\|N\|_\infty+\|A\|_\infty \omega\right). 
\end{align}
Similar to Lemma \ref{lemma:forwarderrorbeta}, Lemma \ref{lemma:backwarderrorbeta} provides the minimum $\tilde{\beta}$ so that $\Omega_{\text{FP}} \geq \Omega_{\text{ZFP}}$. 
\begin{lemma}\label{lemma:backwarderrorbeta}
		Let $k \in \mathbb{N}$ represent the traditional floating-point precision. Assume $\tilde{c}_{n_1} \leq \frac{1}{\epsilon_k}$, then  $\Phi_{\text{FP}}\geq  \Phi_{\text{ZFP}}$ if 
	\begin{align}
	\tilde{\beta} \geq k - \log_2\left( \frac{1-\omega}{(1+2\omega)} \tilde{c}_{n_1}\left(\frac{4}{15}\right)^d \left(\frac{3}{8(2^d+1)}\right)\right) \label{eqn:betabackwardserror}
	\end{align}
	provided $\tilde{\beta} \geq \beta_j$ for all j, where $\beta_j$ is the number of bit planes kept at step $j$.
\end{lemma}
\begin{proof}
	By applying the inequality $\|A\| \leq \|M\|+ \|N\|$ and apply the assumptions $\tilde{c}_{n_1} \leq \frac{1}{\epsilon_k}$ and $\omega<1$, we have that 
	\begin{align}
	\Omega_{\text{FP}}& \geq \|I-H\|_\infty\epsilon_k\tilde{c}_{n_1}\gamma\|A\|_\infty  \geq \left| \|I\|_\infty - \|H\|_\infty\right| \epsilon_k\tilde{c}_{n_1}\gamma\|A\|_\infty  \geq (1 - \omega) \epsilon_k\tilde{c}_{n_1}\gamma\|A\|_\infty 
	\end{align}
	and 
	\begin{align}
\Omega_{\text{ZFP}}\geq \tilde{K}_\beta\gamma\|A\|_\infty \left( \epsilon_k\tilde{c}_{n_1}(1+\omega) + \omega\right)\geq \tilde{K}_\beta\gamma\|A\|_\infty \left(1+2\omega \right)
	\end{align}
	Thus, if $
\tilde{K}_\beta \left(1+2\omega \right) \leq  (1 - \omega) \epsilon_k\tilde{c}_{n_1}
$,
	it follows that $\Omega_{\text{FP}}\geq  \Omega_{\text{ZFP}}$. Without loss of generality, $\tilde{K}_\beta = \left(\frac{15}{4}\right)^d\frac{8}{3}(2^d+1) \epsilon_{\tilde{\beta}}$, and the desired result follows. 
\end{proof}
Figure \ref{fig:contourbackward} represents $\tilde{\beta}$ as described in Equation (\ref{eqn:betabackwardserror}) as a function of $\tilde{c}_{n_1}$ and $\omega$ for the backward error analysis. Similar conclusions from Figure \ref{fig:contour} for the forward error analysis can be concluded for the backward error. However, for the backward error, we see that $\tilde{\beta}$ is much more sensitive to the spectral radius.
\begin{figure}
	\centering
	\begin{subfigure}[b]{0.4\textwidth}
		\includegraphics[width=\textwidth]{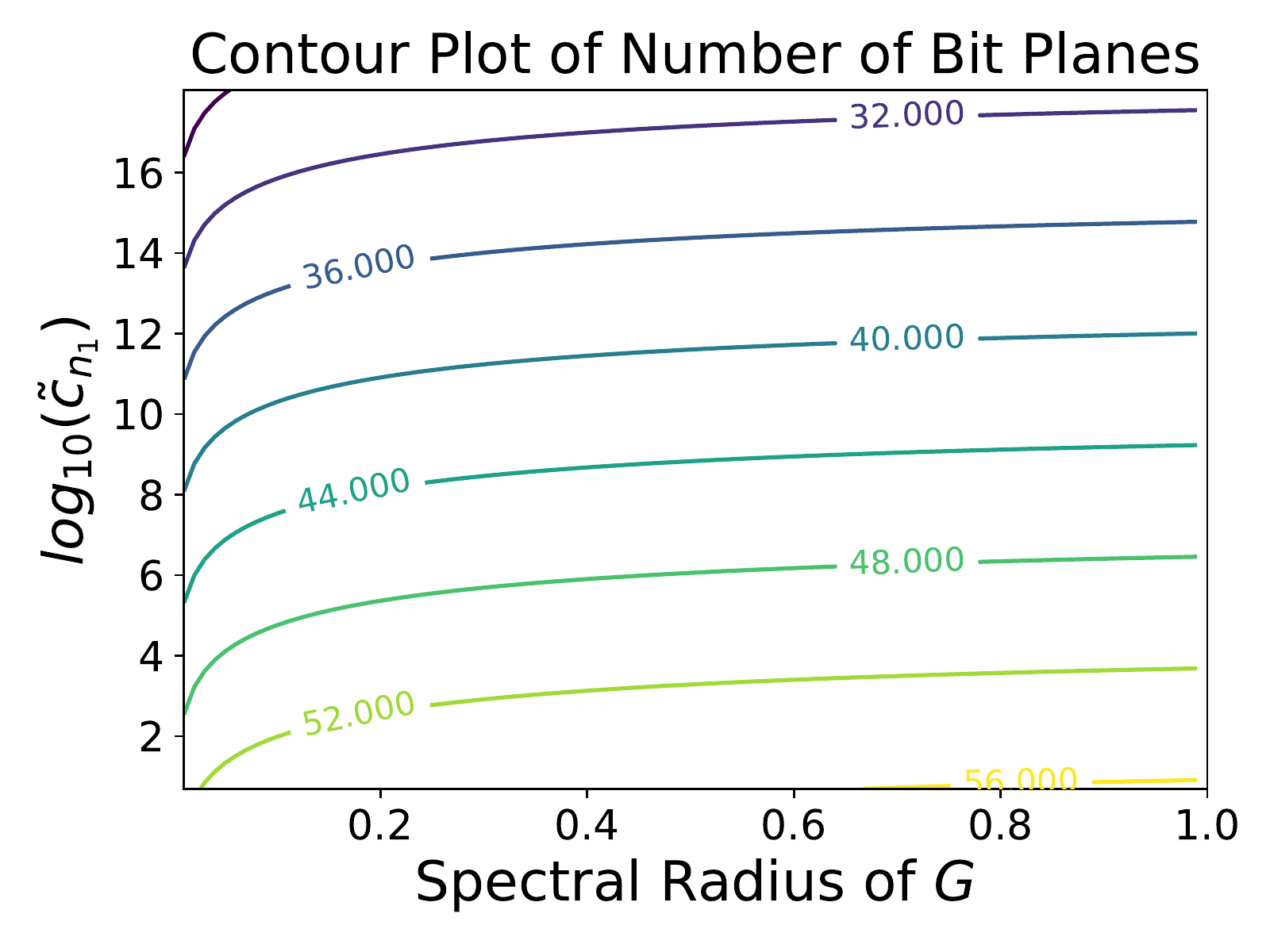}
		\caption{Forward Error}
	\label{fig:contour}
\end{subfigure}
\begin{subfigure}[b]{0.4\textwidth}
	\includegraphics[width=\textwidth]{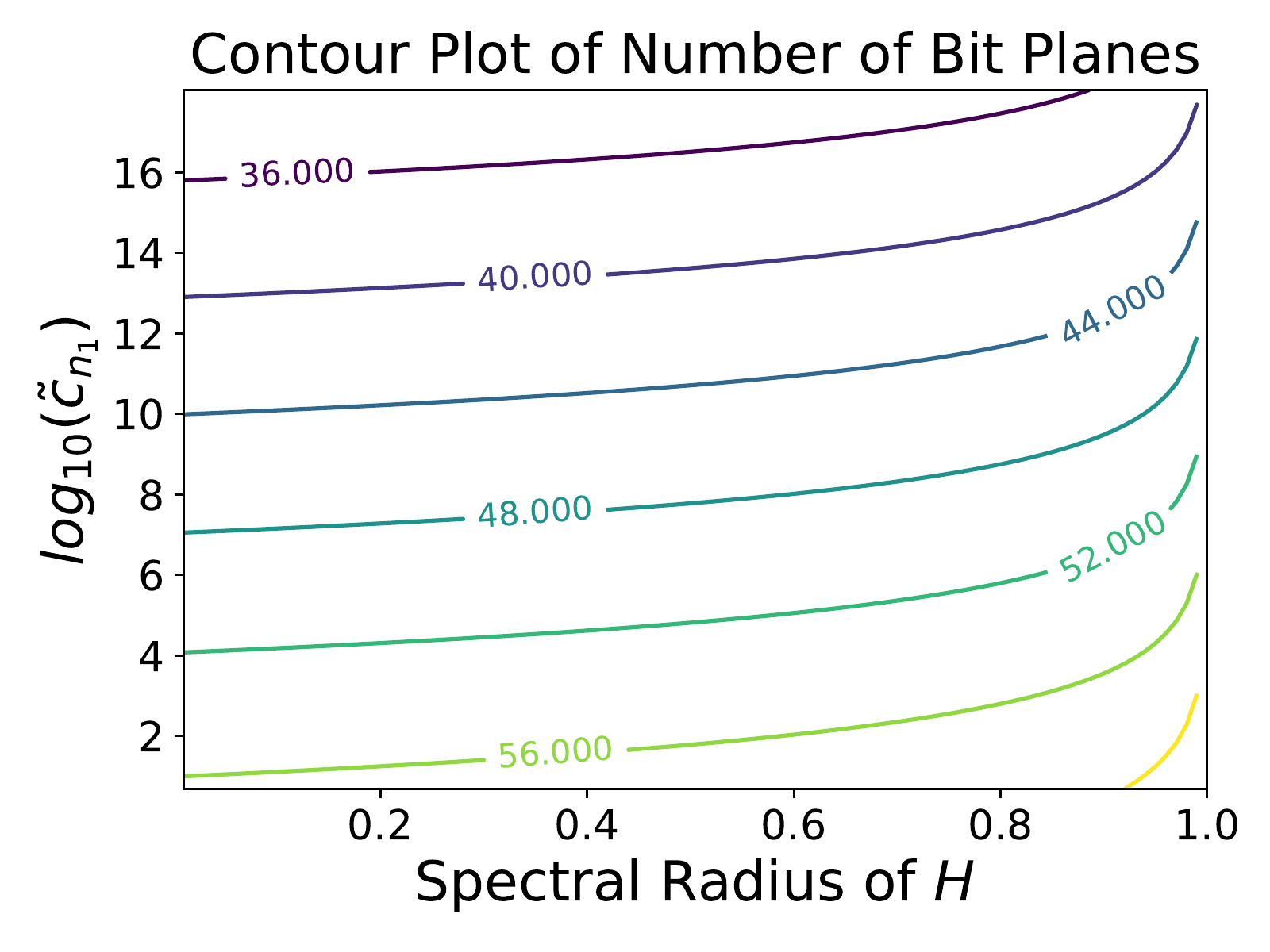}
		\caption{Backward Error}
	\label{fig:contourbackward}
\end{subfigure}
\caption{Contour plot of $\tilde{\beta}$, the number of bit-planes for ZFP, from Lemma \ref{lemma:forwarderrorbeta} as a function of $\tilde{c}_{n_1}$, the conditioning of the linear operator, and the spectral radius of either $G$ or $H$, respectively, as determined by the forward error or backwards analysis in double precision ($k = 53$). (a) Contour plot of $\tilde{\beta}$, as a function of $\tilde{c}_{n_1}$ and $\sigma$, the spectral radius of $G$. (b) Contour plot of $\tilde{\beta}$, as a function of $\tilde{c}_{n_1}$ and $\omega$, the spectral radius of $H$.}
\end{figure}
For both cases, the forward and backwards error for the sequence defined by (\ref{eqn:xn}) contains traditional floating-point arithmetic error and round off error caused by ZFP. Lemmas \ref{lemma:forwarderrorbeta} and \ref{lemma:backwarderrorbeta} provide the $\tilde{\beta}$ value needed to use ZFP in a stationary iterative methods while keeping the additional round-off error caused by ZFP less than the traditional round-off error caused by floating-point arithmetic. \red{The contour plots, shown in Figure \ref{fig:contour} and \ref{fig:contourbackward}, tell us a familiar story; the accuracy of the solution is dependent on the conditioning of the problem, thus, requiring a higher precision for our data representations than what is theoretically possible is fruitless. This is not a new phenomenon, but one that has been well-documented from studying floating-point arithmetic error \cite{higham2002accuracy}. Again, this section specifically studied the relationship between floating-point arithmetic error and ZFP round-off error. In many simulations, the floating-point arithmetic error is not the dominating error. We have shown that despite the possibilities of ZFP round-off error dominating floating-point arithmetic error, that as long as the iterative method is stable, we can bound the ZFP round-off error ensuring the method remains stable when used in conjunction with ZFP compressed data-types. } 

 Now that we have established bounds on the accumulated round-off error introduced by ZFP compression and decompression, we consider several numerical experiments to illustrate the tightness of these bounds. 

\section{Numerical Results}
\label{sec:results}
As observed in Section~\ref{sec:bounds}, repeated inline applications of ZFP with iterative methods will generate additional errors that can accumulate over time. The first two numerical experiments are designed to test how well the bounds established in Theorems \ref{thm:lip}  and \ref{thm:boundedIterative} capture the compression error introduced by ZFP for different values of $\beta$, the fixed precision parameter. In each simulation, the value of the fixed precision parameter is chosen ahead of time and is held constant throughout each simulation. \red{Again, due to hardware limitations, in order to apply the update, arithmetic operations are performed in IEEE double precision; that is, each ZFP solution vector is rounded to IEEE doubles, updated, and then the result is converted back to ZFP format.} \rededit{In each of the experiments, we will compare the additional error to the total numerical error to discuss the impact of the dominant error on the numerical solution, as an exact analytic solution is known. In particular, the total numerical error is determined by subtracting the exact analytic solution from the IEEE double precision solution. For our chosen experiments, the total numerical error will be dominated by truncation error.} We will observe that the round-off error introduced by ZFP will remain below the total numerical error for the correct choice of the fixed precision parameter and that our bounds fully encapsulate the additional error. The final numerical experiment is designed to illustrate Theorem \ref{thm:extraits}. Depending on the Lipschitz constant, the fixed precision parameter, and the number of iterations required to solve the problem without inline compression, we can bound the number of extra iterations required by the ZFP scheme in (\ref{ZFPscheme}) to achieve an iterate that is of similar accuracy to the exact solution of the non-compressed fixed point scheme in (\ref{FPscheme}). 

In the interest of using notation common to the PDE literature, the
notation we use in this section differs from the rest of the paper:
$x$ is a $d$-dimensional spatial variable, $t$ is the temporal variable, $n$ is the current iterate, $u$ the
continuous solution, and ${\bu}$ is an array used for
the uncompressed solution to the discretized diffusion equations. For each example, the finite difference scheme can be written as $\bu^{n+1} = A\bu^{n}$,
where $\bu^{n}$ is the solution calculated in double precision with no inline compression at time step $n$ and $A$ is the advancement operator.   Let
$\bv^{n+1} = A({D}{C}(\bv^{n}) )$, where $\bv^{0} = \bu^{0}$, denote the inline ZFP compression sequence and let 
$\hat{\bu}(x,t)$ denote the exact solution evaluated in double precision.  In Figures~\ref{fig:heat},~\ref{fig:heat3d64}-\ref{fig:heat3d64},~\ref{fig:laxwend},~\ref{fig:laxwend2d}, and~\ref{fig:poisson} the blue lines depict the error introduced by repeated compression of the solution at each time step relative to the true solution, 
$\| \bv^{n+1} - \bu^{n+1} \|_{\infty}/\|\hat{\bu}^{n+1}\|_\infty$. 
The green lines represent the total numerical error relative to the true solution, which, for the mesh
resolutions chosen, is dominated by truncation error, $\| \hat{\bu}^{n+1}- \bu^{n+1}  \|_{\infty}/\|\hat{\bu}^{n+1}\|_\infty$
, where 
$\hat{\bu}(x,t)$ is the exact solution 
evaluated in double precision. The red lines represent the theoretical
bound from either Theorem~\ref{thm:lip} or \ref{ZFPFixedPointThm}. Finally, the black line represents the round-off error caused by IEEE double precision representation relative to the true solution. Let $\tilde{\bu}^n$ represent the numerical solution calculated in 80-bit precision at iterate $n$; then the black line represents $\|\bu^{n+1} -\tilde{\bu}^{n+1}\|_\infty/\|\hat{\bu}^{n+1}\|_\infty$. We chose to present the total numerical error that is dominated by the truncation error as it is important to understand when the error caused by floating-point round off error or by the repeated use of ZFP, with appropriate parameters, is less than or greater than the dominating error. If the error caused by ZFP remains less than the total numerical error in our example the bits that represent the ZFP error are meaningless for the solution. Additionally, as ZFP is a compression algorithm, we also present the compression ratio, i.e., the ratio between the uncompressed size and compressed size at time $n$, for varying $\beta$ values. 
\subsection{Diffusion Example}
\label{sec:heat}
In the first example, we wish to test the bounds of the propagation of
the additional error introduced by ZFP in an iterative method, as
derived in Theorem~\ref{thm:lip}. The specific PDE we solve is 
\begin{equation}
\begin{array}{rclcrcl}
\partial_t u(x,t)   & = &a \nabla^2 u(x,t)   && (x,t) &\in& [0,1]^d \times (0,1] \equiv \Omega, \\
\end{array}
\end{equation}
with initial and boundary conditions, 
\begin{equation}
\begin{array}{rclcrcl}
u(x,t) &=& 0 \quad  \forall x  \in \delta\Omega, & \text{and} &u(x, t=0) =  \begin{cases} 
1 & x=\hat{x}  \\
0 & \text{otherwise}\\
\end{cases} & &
\end{array}
\end{equation}
for $d = 2$, $\hat{x}  = (1/2, 1/2)$, and constant $a = 1 $. Using a forward time and central space finite
difference approximation, the update is given by
\begin{align}u_{i,j}^{n+1} = u_{i,j}^n +a \Delta t \left( \frac{u_{i,j+1}^n-2u_{i,j}^n+u_{i,j-1}^n }{\Delta x_1^2 } +\frac{u_{i+1,j}^n-2u_{i,j}^n+u_{i-1,j}^n}{\Delta x_2^2 }\right ).  \end{align}
We choose a reasonable set of discretization parameters, $i,j = 1, ..., 99$, $\Delta x_1 = \Delta x_2 = 1/100$, $\Delta t= 0.0003125$, and the Dirichlet periodic boundary conditions are enforced by $u_{i,0}^n=u_{0,j}^n = 0$, for all $i,j$. The finite difference scheme can be written as $\bu^{n+1} = A\bu^{n}$,
where $\bu^{n} = [ u_{i,j}^n ]^{{T}}$ is the solution calculated in double precision with no inline compression at time step $n$ with Lipschitz
constant $L_l = \|A\|_\infty = 1$. Let
$\bv^{n+1} = A({D}{C}(\bv^{n}) )$, where $\bv^{0} = \bu^{0}$, denote the inline ZFP compression sequence.  In this example, shown in Figure \ref{fig:heat}, the red lines represent the theoretical
bound from Theorem~\ref{thm:lip},
\begin{equation}
\sum_{j = 0}^{n} L_l^{n-j+1} K_{\beta_j} \|\bv^{j}\|_{\infty},
\end{equation}
where $\beta_j = \{64, 59,44,32,16\}$ is constant for each experiment
presented.

\begin{figure}[h!]
	\centering
	\begin{subfigure}[b]{0.4\textwidth}
		\includegraphics[width=\textwidth]{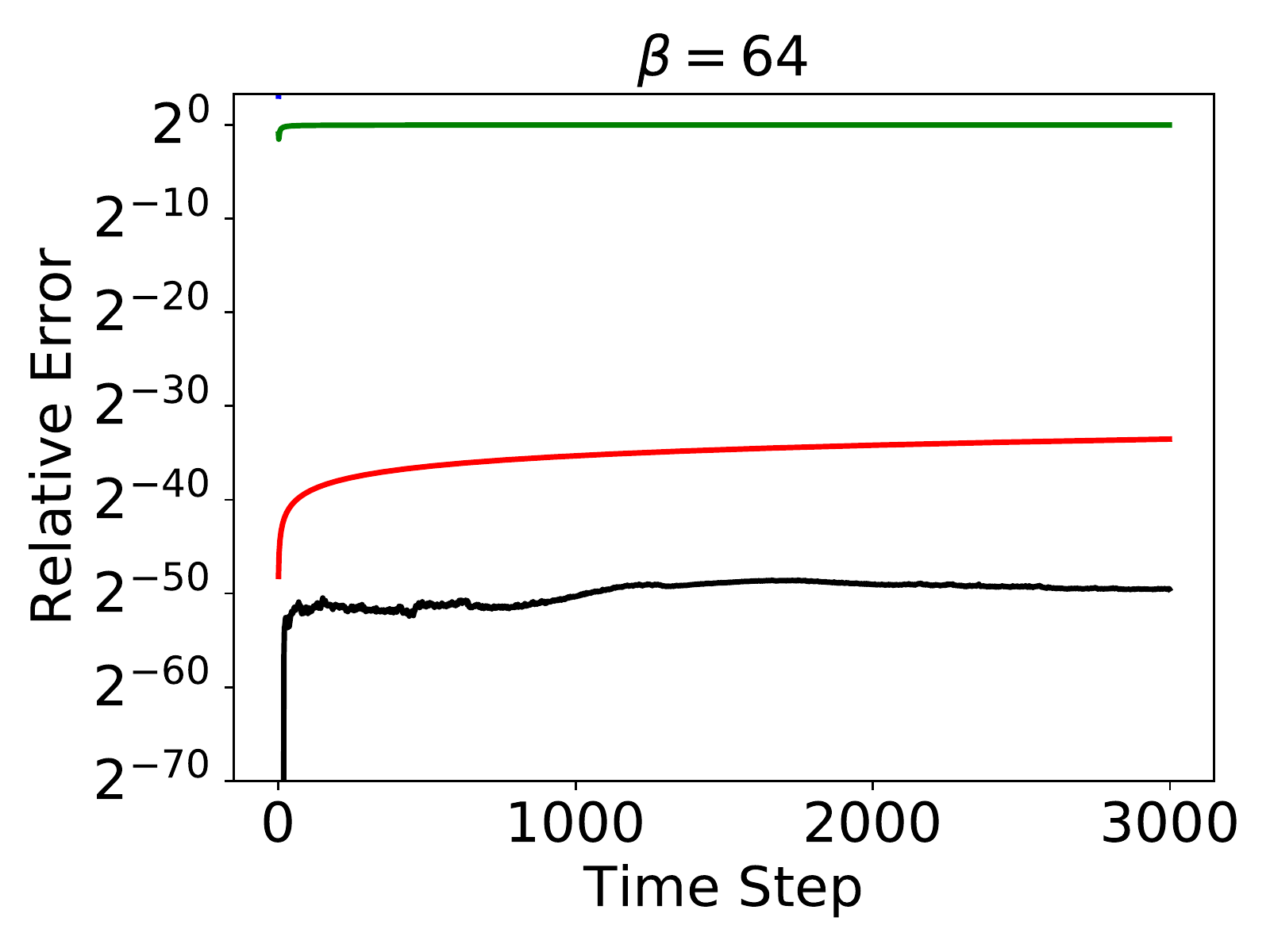}
		\caption{$\beta= 64$}
		\label{fig:heat64}
	\end{subfigure}
	\begin{subfigure}[b]{0.4\textwidth}
		\includegraphics[width=\textwidth]{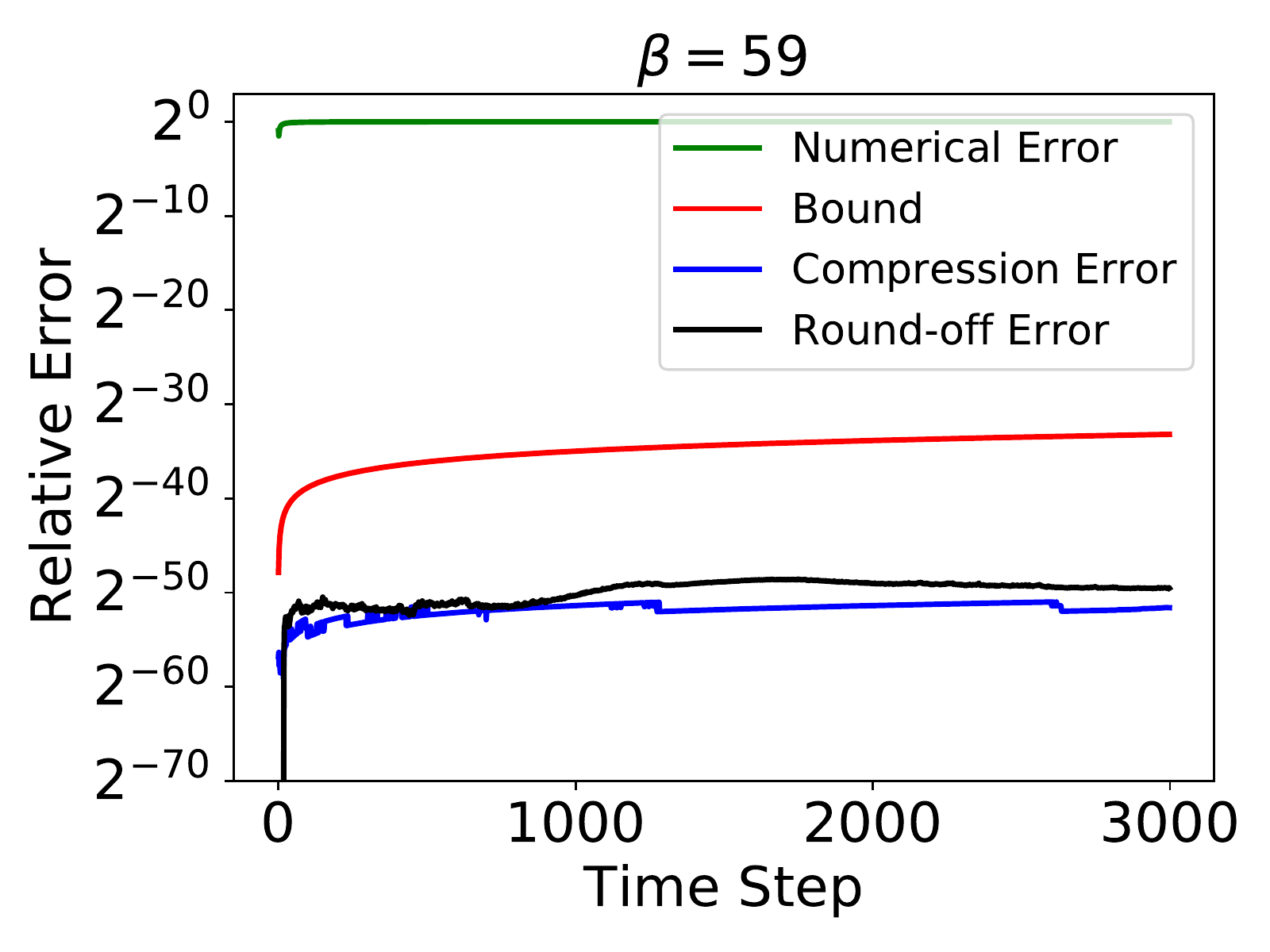}
		\caption{$\beta = 59$}
		\label{fig:heat59}
	\end{subfigure}
	\begin{subfigure}[b]{0.4\textwidth}
		\includegraphics[width=\textwidth]{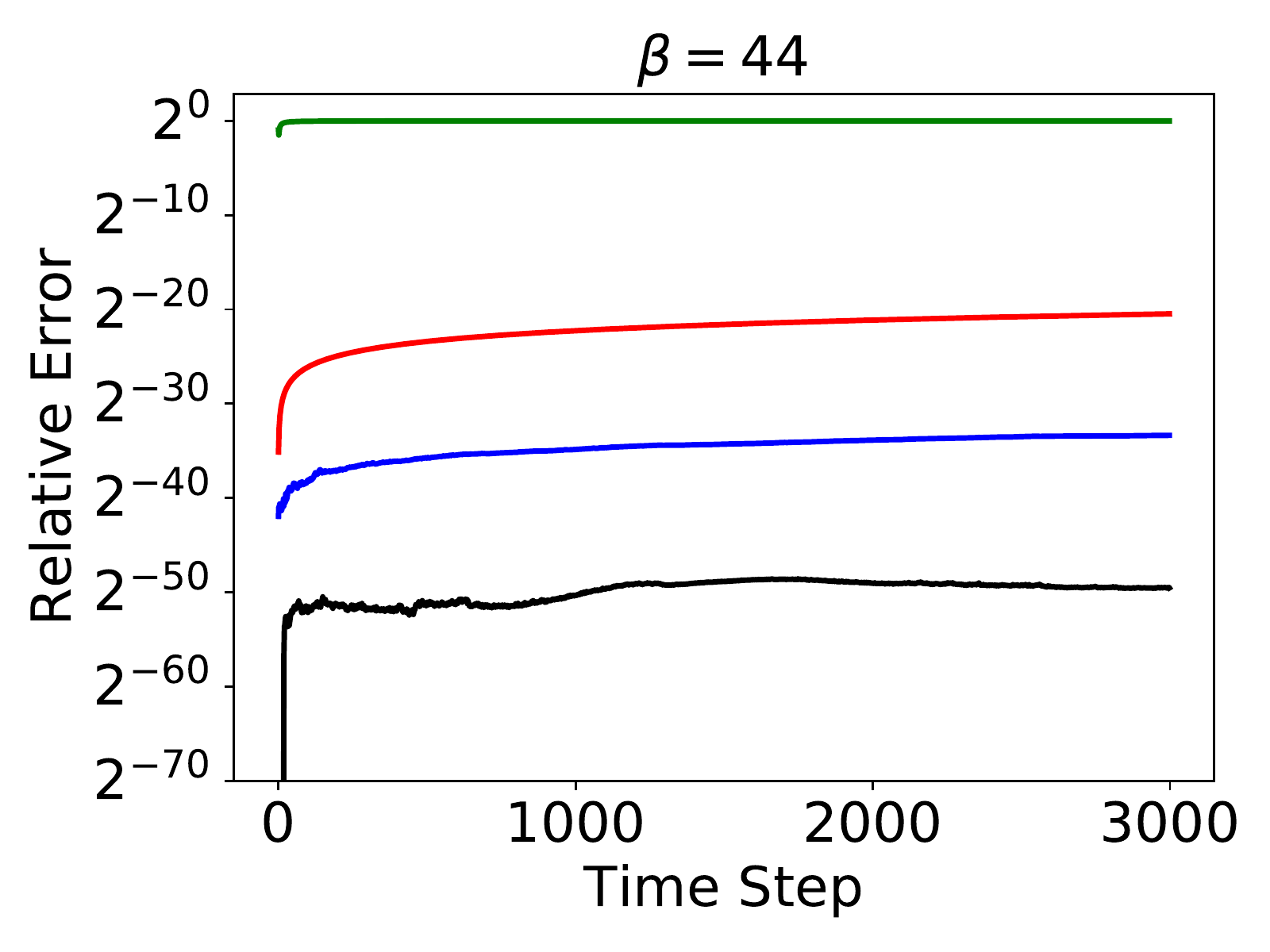}
		\caption{$\beta = 44$}
		\label{fig:heat44}
	\end{subfigure}
	\begin{subfigure}[b]{0.4\textwidth}
		\includegraphics[width=\textwidth]{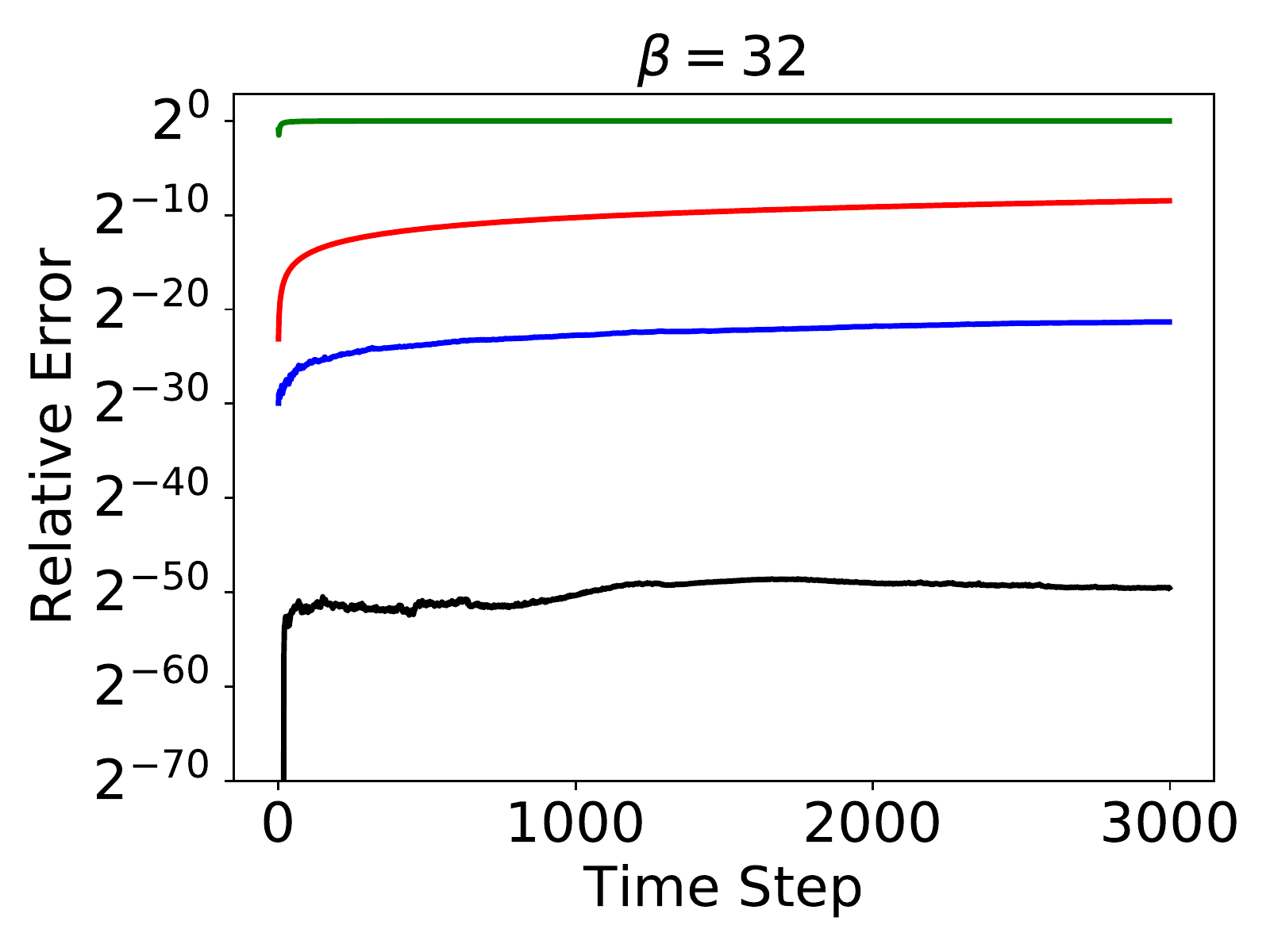}
		\caption{$\beta = 32$}
		\label{fig:heat32}
	\end{subfigure}
	\begin{subfigure}[b]{0.4\textwidth}
		\includegraphics[width=\textwidth]{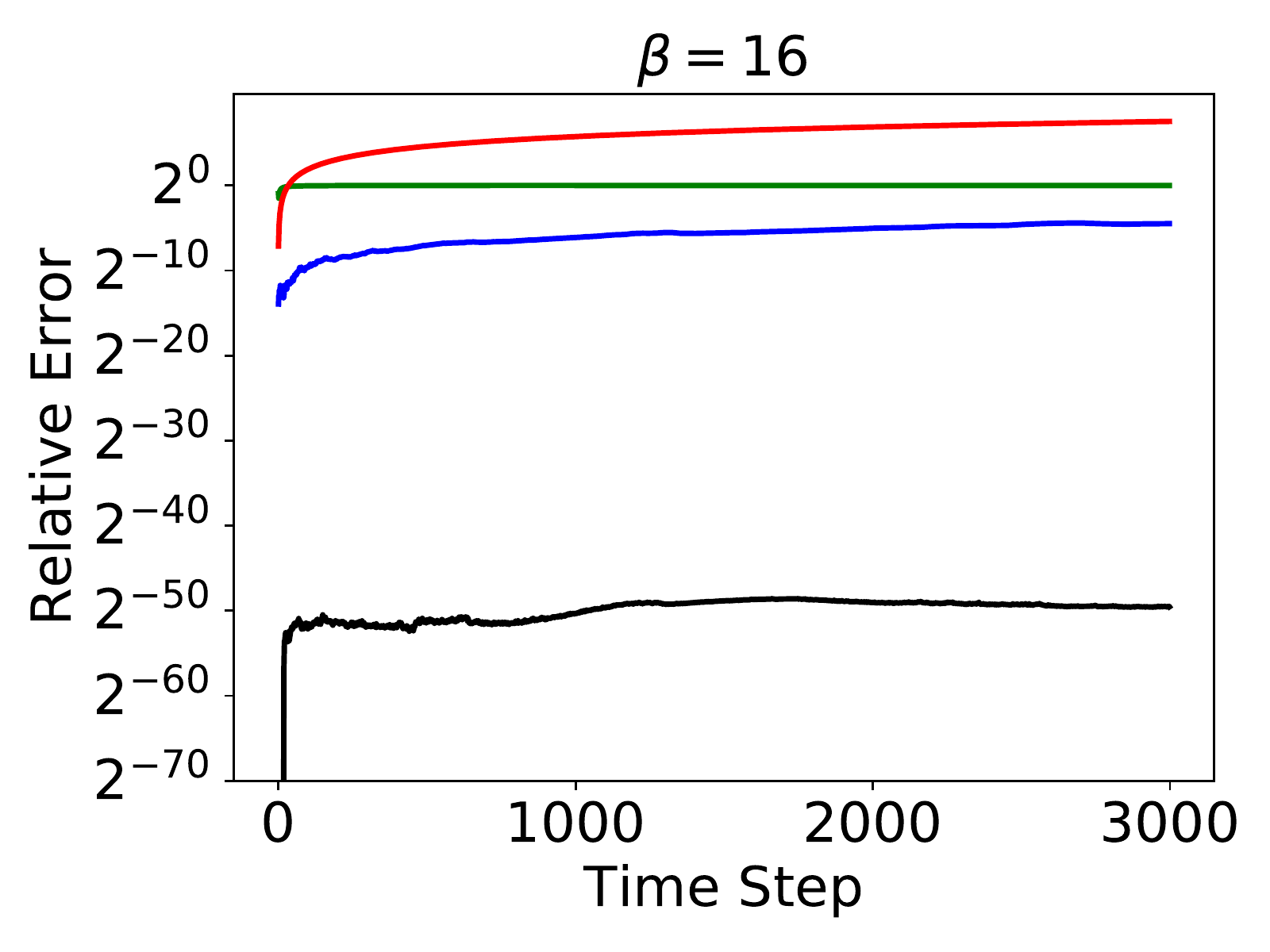}
		\caption{$\beta = 16$}
		\label{fig:heat16}
	\end{subfigure}
	\begin{subfigure}[b]{0.4\textwidth}
			\includegraphics[width=\textwidth]{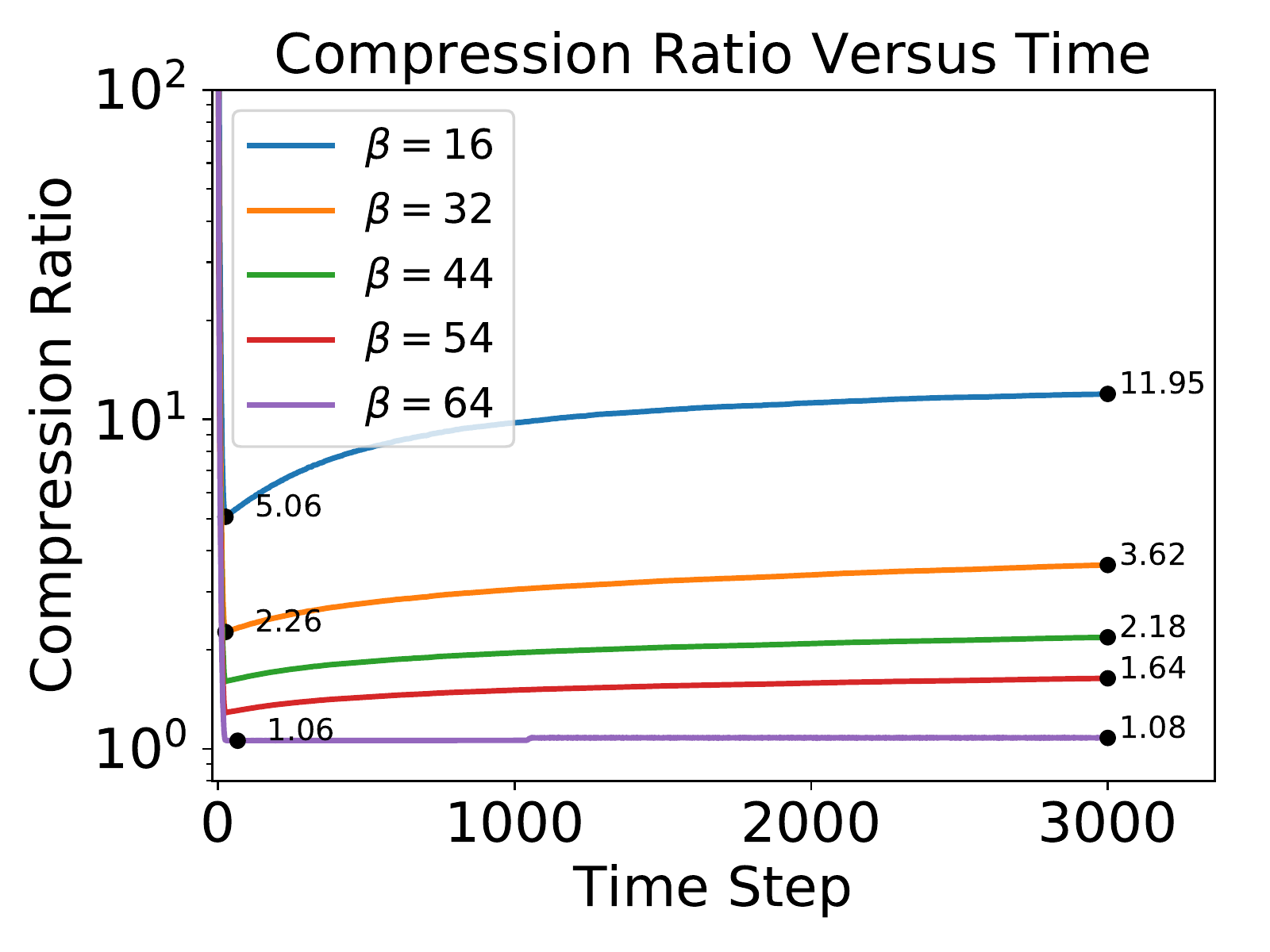}
		\caption{Compression Ratio}
		\label{fig:heatcompression}
	\end{subfigure}
	\caption{2D Diffusion Example: The blue line represents the compression error caused by ZFP, the green line represents the truncation-dominated total error, the red line represents the theoretical bound, and the black line represents the IEEE double precision floating-point round-off error. 
	}
	\label{fig:heat}
\end{figure}

For all $\beta_j$, the theoretical prediction bounds the relative error caused by the compression. As $\beta_j$ decreases, the relative round-off error from ZFP compression and the theoretical bound both increase slowly during the time-stepping iterations and trend in a similar way. The truncation-dominated numerical error is well above the compression error caused by ZFP, which means that the error is still dominated by the error caused by discretization. When $\beta = 64$, the error caused by compression is lower than the error caused by the double precision round-off \red{estimated by computing the solution using 80-bit precision}, i.e., the total floating-point arithmetic error. \red{Even though $\beta = 64$, the ZFP data-type is using fewer bits, as ZFP uses an encoding scheme to represent the values more efficiently.} From Lemma \ref{lemma:forwarderrorbeta}, assuming $\tilde{c}_{n_1} \approx 10,000$, i.e., the number of unknowns which can be used as an estimate of the condition number, we must have $\tilde{\beta}\geq 44 $ for the round-off caused by ZFP to be less than the total floating-point arithmetic error using the forward error analysis. However, we see this is not the case, in Figure \ref{fig:heat44}. As $A$ is a sparse well conditioned matrix, $\tilde{c}_{n_1} \approx 10,000$ is likely a pessimistic approximation. If instead we assume $\tilde{c}_{n_1} \approx 1$, the most optimistic approximation, then from Lemma \ref{lemma:forwarderrorbeta}, we must have $\tilde{\beta} \geq 59$, which is consistent with Figure \ref{fig:heat59}. For each iterate, the error caused by ZF remains below the floating-point arithmetic error validating Lemma \ref{lemma:forwarderrorbeta}. If instead we were concerned with the total numerical error, we can use Lemma \ref{lemma:truncerror2} to estimate the value $\tilde{\beta}$ that will result in the error bound to be approximately the same magnitude as the truncation-dominated numerical error. The numerical error from the numerical experiment can be concluded to be $\approx 2^{-14}$ at $n =3000.$ From Lemma \ref{lemma:truncerror2}, we have that $\tilde{\beta} \geq 32$, which is depicted in Figure \ref{fig:heat32}. However, the actual accumulated error caused by ZFP will be much less than  $ 2^{-14}$.  We can conclude that even though ZFP is introducing error into the numerical simulation, the method remains stable, as the theorems predict. When $\beta_j = 16$, depicted in Figure \ref{fig:heat16}, the total numerical error still dominates the ZFP round-off error. \red{Our bound, though it encapsulates the introduced error, is more conservative than is necessary; 16 bit planes are sufficient to obtain a solution with the same level of total numerical error.}

Figure \ref{fig:heatcompression} presents the compression ratio, i.e., the ratio between the uncompressed size and compressed size at time $n$, for varying $\beta$ values.  For the first few time steps, ZFP is able to efficiently compress the data as there is a single point heat source while the remaining domain is approximately zero. However, there is a sharp decrease in the compression ratio until approximately  $n\approx 80$, where the minimum compression ratio is attained for each $\beta$ value. The minimum compression ratio is 1.09 at $n = 83$ with $\beta = 64$. Then as $n \geq 83$ increases, the compression ratio increases. The final compression ratio at $n= 3000$ is plotted for each $\beta$ value. For $\beta = 64$ the final compression ratio is 1.31. 

%

ZFP is known to produce higher compression ratios for higher dimensions. Thus, we perform the same numerical example in three dimensions 
using the same setup as before with  $\Delta x_1 = \Delta x_2 =\Delta x_3 = 1/40$, and $\Delta t= 2.5\times 10^{-5}$. 
The Dirichlet periodic boundary conditions are enforced by $u_{i,j,0}^n=u_{0,j,h}^n =u_{i,0,h}^n= 0$, for all $i,j,h$, and the initial condition is $u_{19,19,19} = 1$. Figure \ref{fig:heat3d} displays similar results as the 2D for the 3D diffusion example, and similar conclusions from the 2D case can be drawn. The total numerical error from the numerical experiment can be concluded to be $\approx 2^{-14}$ at $n =3000.$ Similarly, from Lemma \ref{lemma:truncerror2}, we need $\tilde{\beta} \geq 32$ for the ZFP round-off error to remain below the numerical error, which is depicted in Figure \ref{fig:heat3d32}.  Figure \ref{fig:heat3dcompression} presents the compression ratio for varying $\beta$ values. As expected, the compression ratios tend to be higher in the 3D case than the 2D case (although not substantially).
\begin{figure}[h!]
	\centering
	\begin{subfigure}[b]{0.4\textwidth}
		\includegraphics[width=\textwidth]{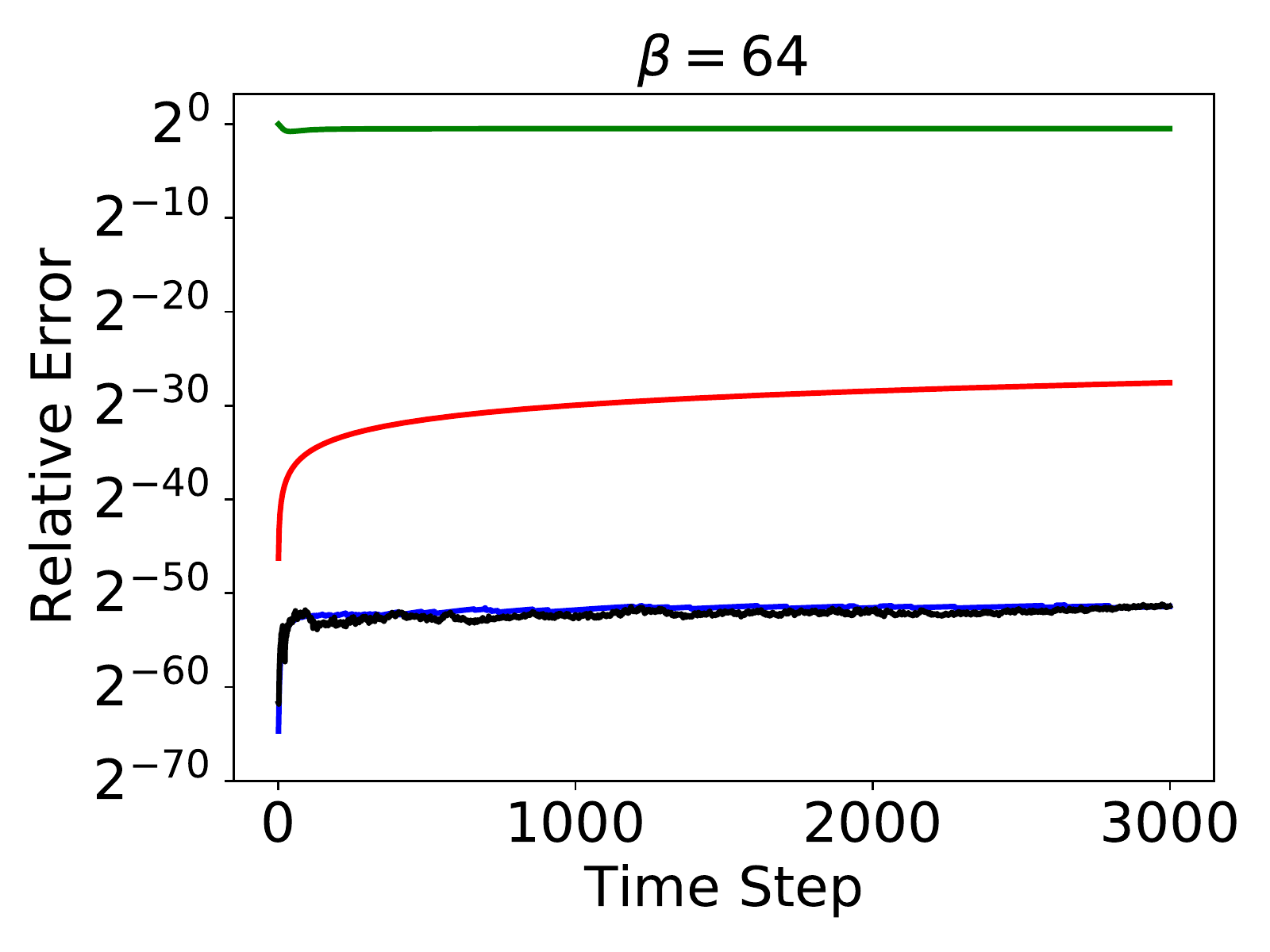}
		\caption{$\beta= 64$}
				\label{fig:heat3d64}
	\end{subfigure}
	\begin{subfigure}[b]{0.4\textwidth}
		\includegraphics[width=\textwidth]{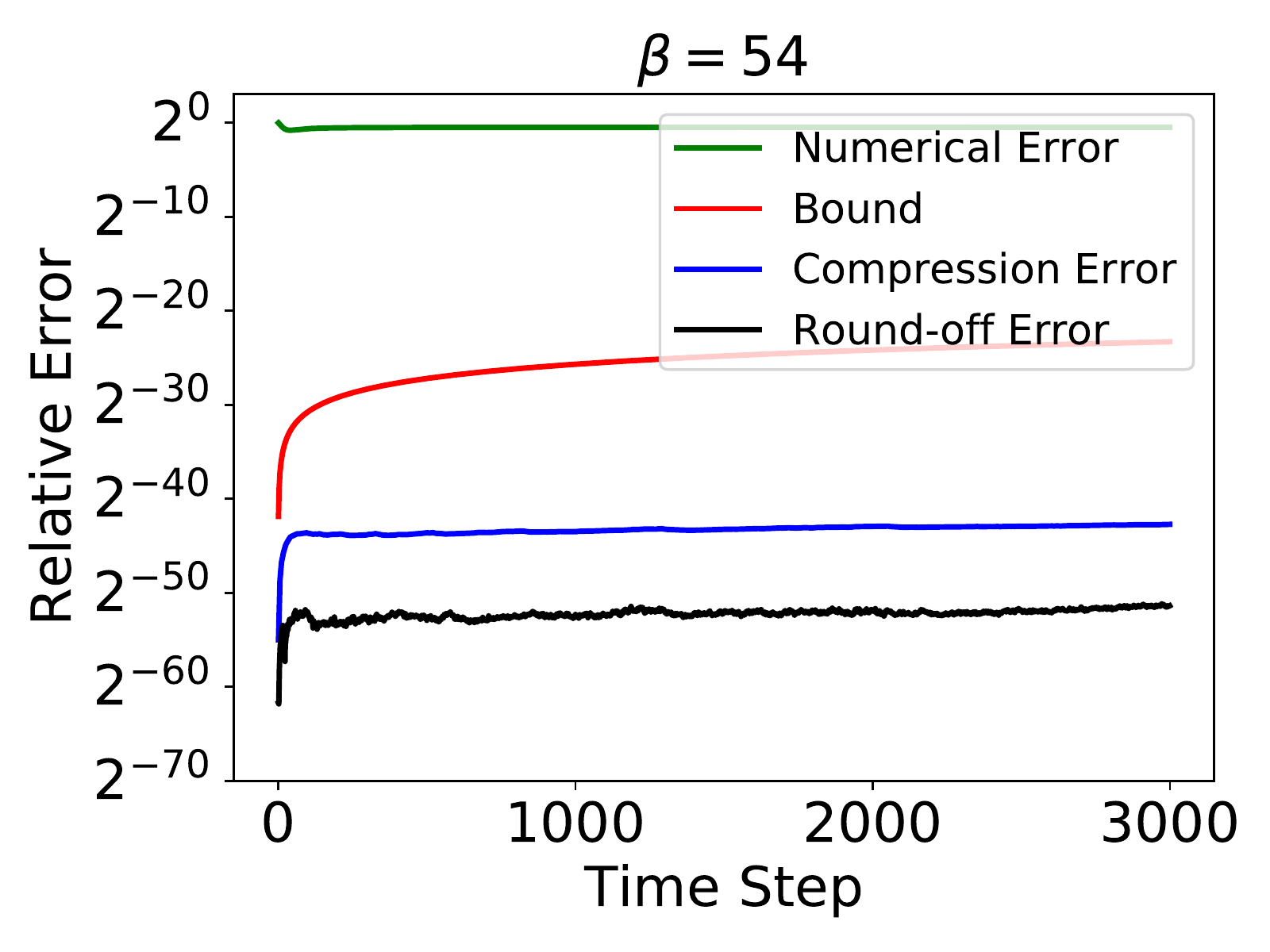}
		\caption{$\beta = 54$}
	\end{subfigure}
	\begin{subfigure}[b]{0.4\textwidth}
		\includegraphics[width=\textwidth]{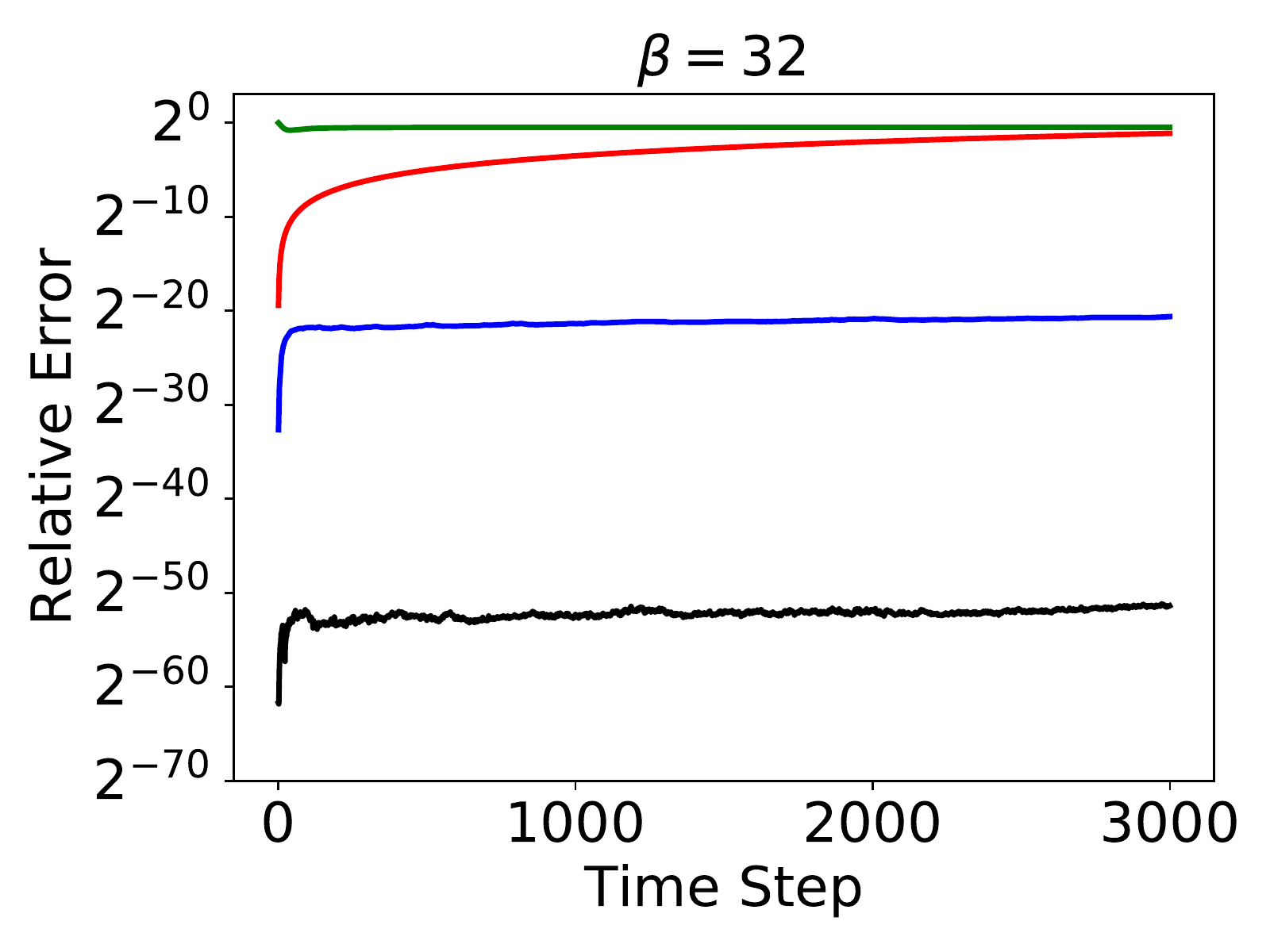}
		\caption{$\beta = 32$}
		\label{fig:heat3d32}
	\end{subfigure}
	\begin{subfigure}[b]{0.4\textwidth}
		\includegraphics[width=\textwidth]{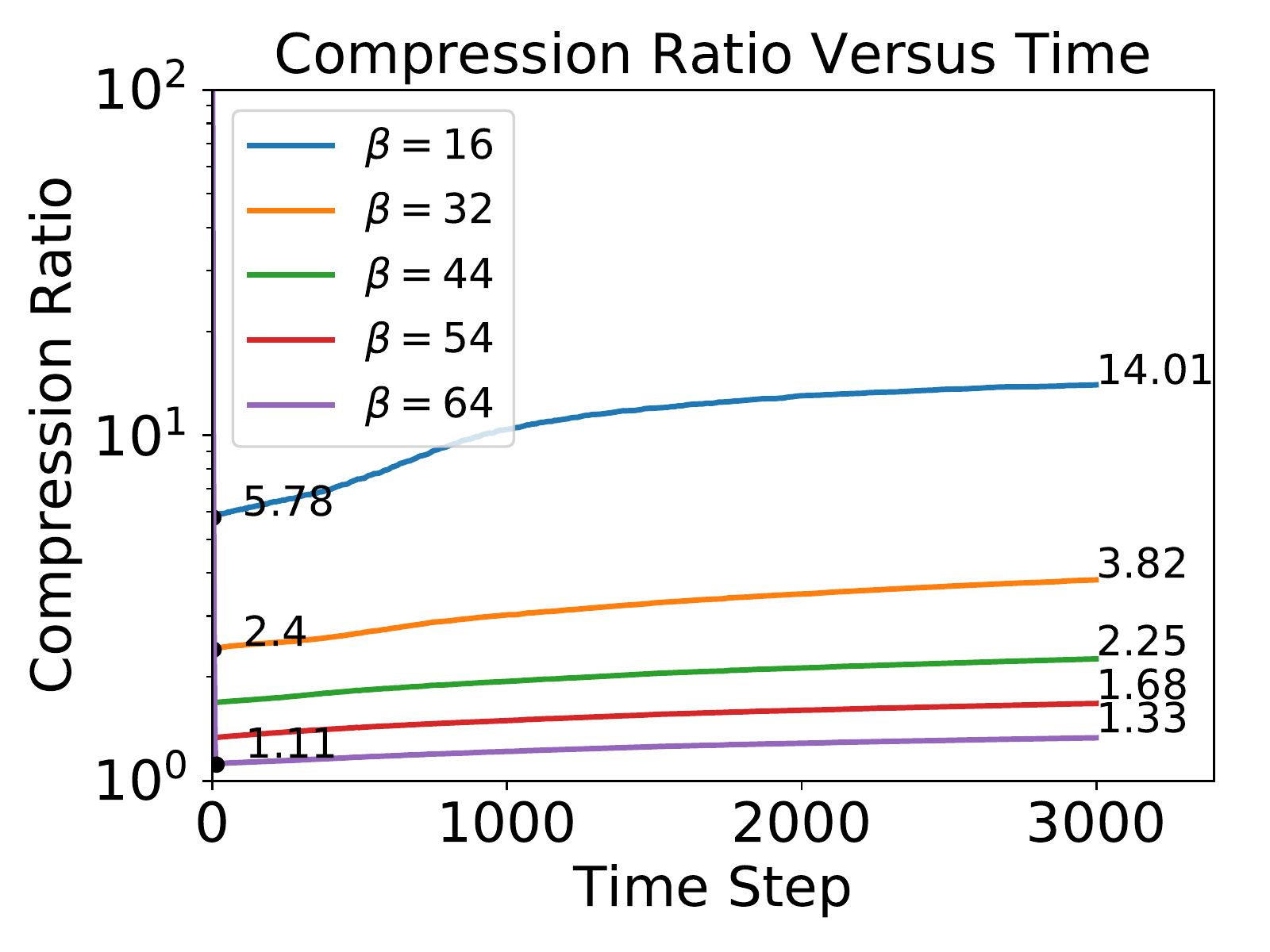}
		\caption{Compression Ratio}
		\label{fig:heat3dcompression}
	\end{subfigure}
	\caption{{3D Diffusion Example{: Figure (a-e) The blue line represents the compression error caused by ZFP, the green line represents the truncation-dominated total error, the red line represents the theoretical bound, and the black line represents the IEEE double precision round-off error for $\beta = \{64,54,44,32,16\}$ (f)} The compression ratio for varying $\beta = \{64,32,26,16\}$ is displayed as a function of time.} }
	\label{fig:heat3d}
\end{figure}

\subsection{Advection Example}
\label{sec:laxwend}
In our next example, we have selected the hyperbolic advection
equation in 1D, as any error generated by ZFP compression will be less damped
than in a system that involves explicit diffusion.    
Classical discretizations of this PDE, such as the Lax-Wendroff method, contain a
small amount of implicit dissipation that helps to stabilize the
scheme. However, the leading-order truncation error in this case is actually dispersive.  
Solving a discrete system with weak diffusion results in error that may
grow over time like the worst-case theoretical bound of the iterative
method expressed in (\ref{thm:lip}) does, although the rate of
growth is typically nowhere near the worst case for reasonably smooth
initial data.

The specific PDE we solve in $d$ dimensions, with its initial and boundary conditions, is
\begin{equation}
\begin{array}{rclcrcl}
\partial_t u(x,t) +a \nabla u(x,t)  & = & 0,  && (x,t) &\in& [0,1]^d \times (0,1], \\
u(x+me_d,t) & = & u(x,t), &&m& =&  0,\pm 1,\pm 2,\dots \\ 
u(x, t=0) & = & f(x), && 
\end{array}
\end{equation}
%
for constant $a$, where $e_d$ is the Cartesian unit vector in the $d^{th}$ direction. For $d= 1$, the general analytic solution is $u(x,t) = f(x - at)$, representing a
wave moving in the positive $x$-direction when $a>0$.  The Lax-Wendroff
scheme is a second-order accurate finite differencing method, expressed
as   
\begin{align} u_i^{n+1} = u_i^n -a\frac{\Delta t}{2\Delta x } \left [ u_{i+1}^n - u_{i-1}^n \right ] +a^2 \frac{\Delta t^2}{2\Delta x^2 } \left [ u_{i+1}^n -2u_i^n+ u_{i-1}^n \right ],  \end{align}
for $i = 1, ..., 100$ and periodicity enforced by $u^n_0:=u^n_{100}$ and
$u^n_{101} := u^n_1$.   

For initial condition $u(x,0) = \sin(2\pi x)$, let $a =1$, $\Delta x$ = 1/90, $\Delta t = 1/100$, such that the CFL
number is $\sigma=a\Delta t /\Delta x = 9/10$. 
The Lax-Wendroff scheme can again be written as $\bu^{n+1} = A\bu^{n}$,
where $\bu^{n} = [ u_1^n, \cdots, u_{100}^n]^{{T}}$.  In this case, the Lipschitz 
constant is $\|A\|_\infty =|1-\sigma^2|+|\sigma|\approx 1.09 $, so for any $0<\sigma<1$, $\|A\|_\infty>1$, the bound of Theorem~\ref{thm:lip} will
grow exponentially with each time step (iteration).  However, Theorem \ref{thm:boundedIterative} only depends on the Kreiss constant, which is more appropriate. It is known that for the Lax-Wendroff scheme the Kreiss constant is $ L_k = 2$ (\cite{kriess}, p.89); thus, we can investigate the validity of Theorem \ref{thm:boundedIterative}. 
In this example, the red lines represent the theoretical
bound from Theorem \ref{thm:boundedIterative}, 
\begin{equation}
\sum_{j = 0}^{n} L_k K_{\beta_j} \|\bu^{j}\|_{\infty},
\end{equation}
where $\beta_j = \{ 64,32,24,16 \}$ is constant for each experiment
presented in Figure \ref{fig:laxwend}.
\begin{figure}[h!]
	\centering    
	\begin{subfigure}[b]{0.4\textwidth}
		\includegraphics[width=\textwidth]{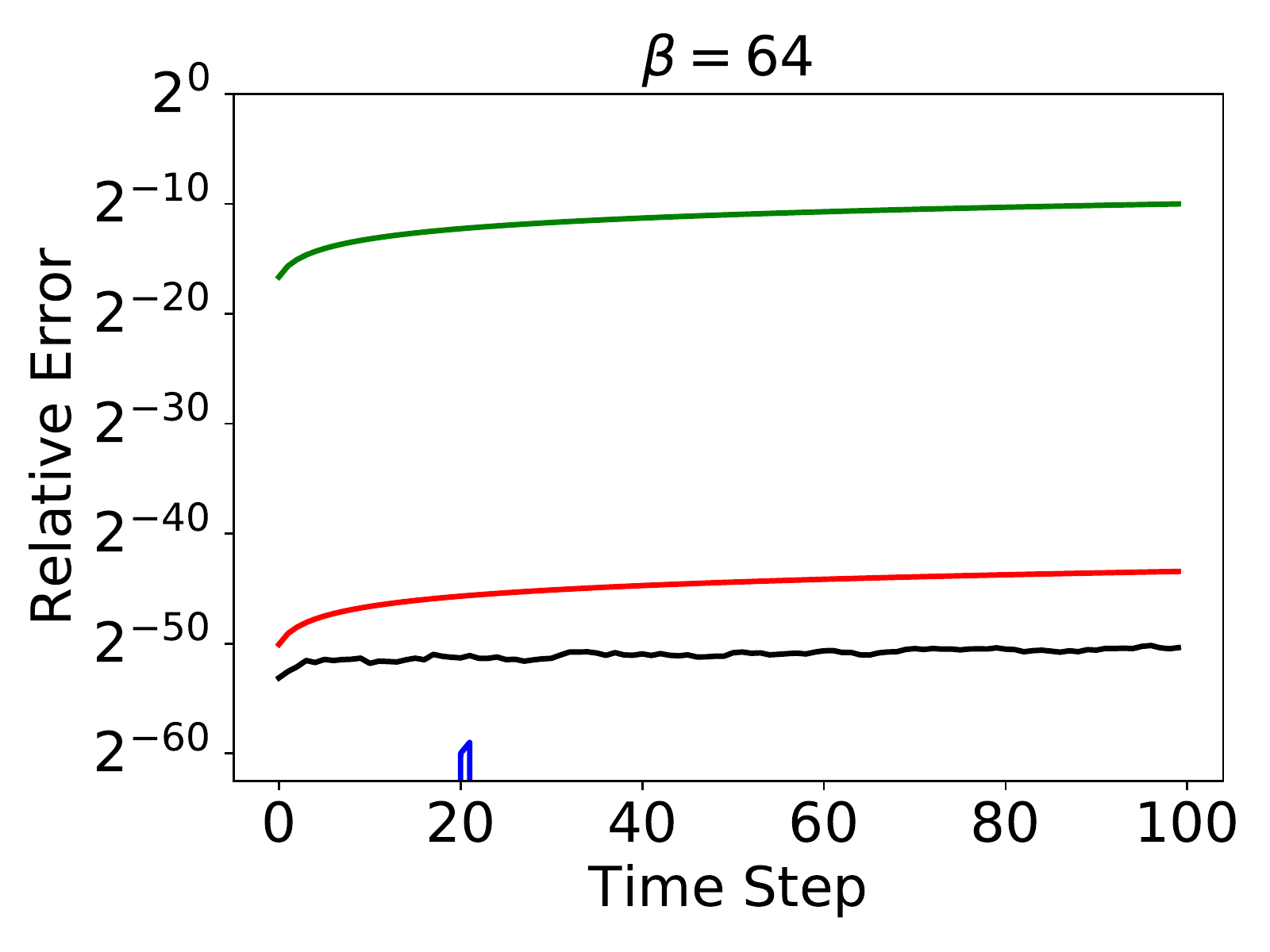}
		\caption{$\beta= 64$}
	\end{subfigure}
	\begin{subfigure}[b]{0.4\textwidth}
		\includegraphics[width=\textwidth]{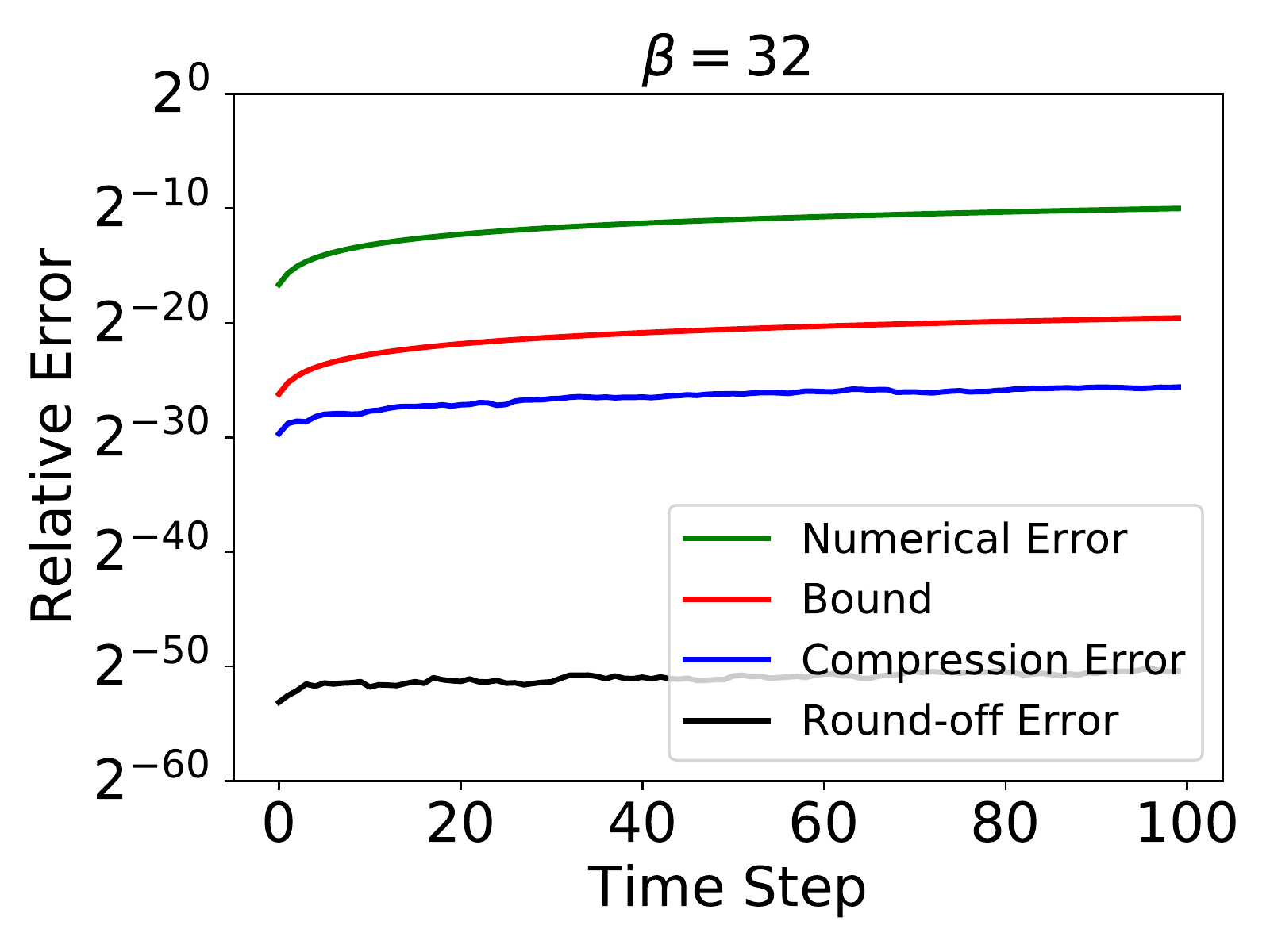}
		\caption{$\beta = 32$}
	\end{subfigure}
	\begin{subfigure}[b]{0.4\textwidth}
		\includegraphics[width=\textwidth]{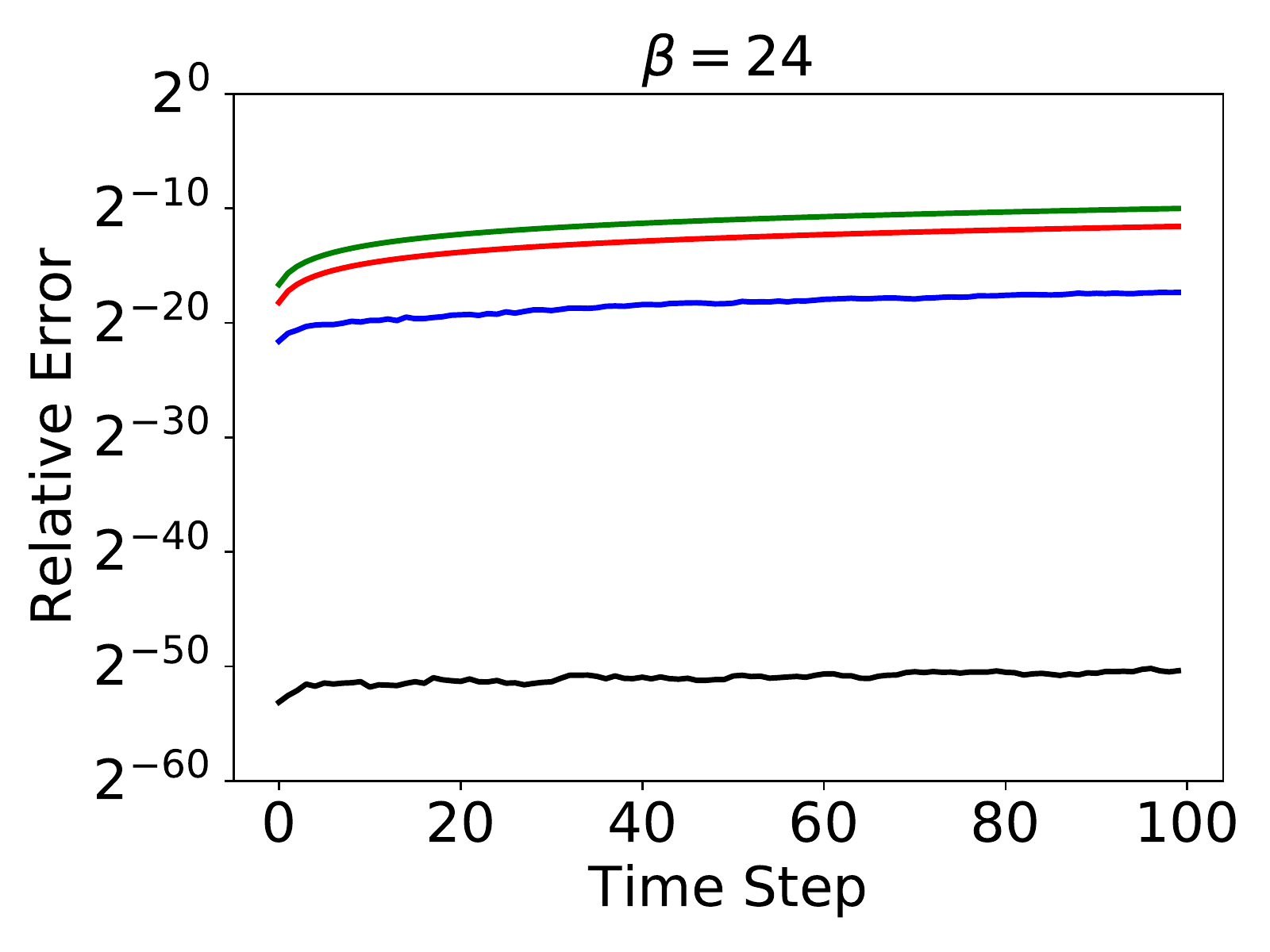}
		\caption{$\beta= 24$}
		\label{fig:laxwend26}
	\end{subfigure}
	\begin{subfigure}[b]{0.4\textwidth}
		\includegraphics[width=\textwidth]{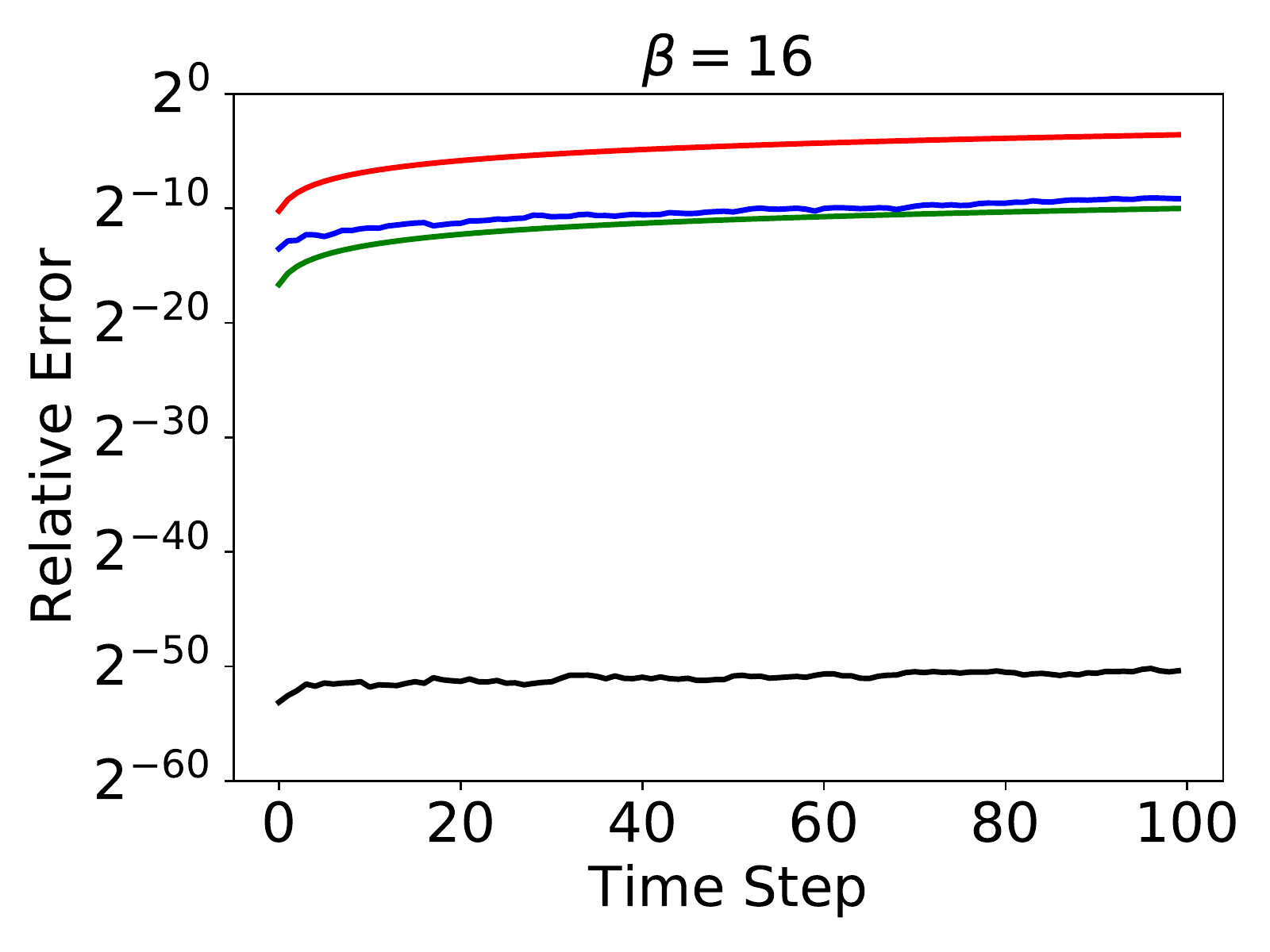}
		\caption{$\beta = 16$}
	\end{subfigure}
	\caption{{1D Lax-Wendroff Example{: The blue line represents the compression error caused by ZFP, the green line represents the truncation-dominated total error, the red line represents the theoretical bound, and the black line represents the IEEE double precision round off error estimated by using a solution calculated in  $80$-bit precision.} } }
	\label{fig:laxwend}
\end{figure}

For all $\beta_j$, the error caused by the compression algorithm is less than the total numerical error
dominated by discretization, and the theoretical bound is able to fully capture the error caused by the compression. For $\beta_j = 64$, ZFP produces no discernible additional error.  
The maximum truncation-dominated total error is approximately $\approx 2^{-8}$ at $n = 1000$; thus, from Lemma \ref{lemma:truncerror2}, we have $\tilde{\beta}\approx 24$.  Figure \ref{fig:laxwend26} depicts the same plots as before where $\beta_j = 24$ is held constant. One can see that, at time-step $n= 1000$, the theoretical bound is less than the total numerical error, but the actual error caused by the compression is much smaller than both the bound and total numerical error. \red{Unlike the diffusion example, when $\beta = 16$ the total numerical error is less than the ZFP round-off error. However, the use of ZFP compressed data-types still did not destabilize the method and our theoretical bound still bounds the additional error ensuring the method remained stable. }Figure \ref{fig:laxcompression2} represents the compression ratio for varying $\beta$ values. Since the initial condition is a sine wave with periodic boundary conditions, the sine wave propagates as time varies. The shape of the sine wave remains nearly constant and thus the compression ratio remains constant as time varies. 

As ZFP does not produce significant compression ratios in 1D we must consider ZFP in higher dimensions.  Thus, we perform the same numerical example in two dimensions. 
Let $a =1$, $\Delta x_1 = \Delta x_2$ = 1/90, and  $\Delta t = 1/100$. For initial condition $u(x,0) = \cos(2\pi(x_1+x_2))$, we used the two-step Lax-Wendroff scheme \cite{Zwas1972}.  Figure \ref{fig:laxwend2d} displays similar results for the 2D advection example, and similar conclusions to the 1D case can be drawn. The total numerical error from the numerical experiment is $\approx 2^{-10}$ at $n =100.$ From Lemma \ref{lemma:truncerror2}, we need $\tilde{\beta} \geq 26$ for the ZFP round-off error to remain below the total numerical error, which is depicted in Figure \ref{fig:laxwend262d}.  Figure \ref{fig:lax2ddcompression} presents the compression ratio for varying $\beta$ values. As expected, the compression ratios tend to be higher in the 2D case than the 1D case. \red{When $\beta = 16$, round-off error caused by ZFP does not destabilize the method which still converges to a solution. However, since the round-off error is over an order of magnitude greater than the total numerical error, the compression ratio degrades over time, as seen in Figure \ref{fig:lax2ddcompression}.  }
\begin{figure}[h!]
	\centering    
	\begin{subfigure}[b]{0.4\textwidth}
		\includegraphics[width=\textwidth]{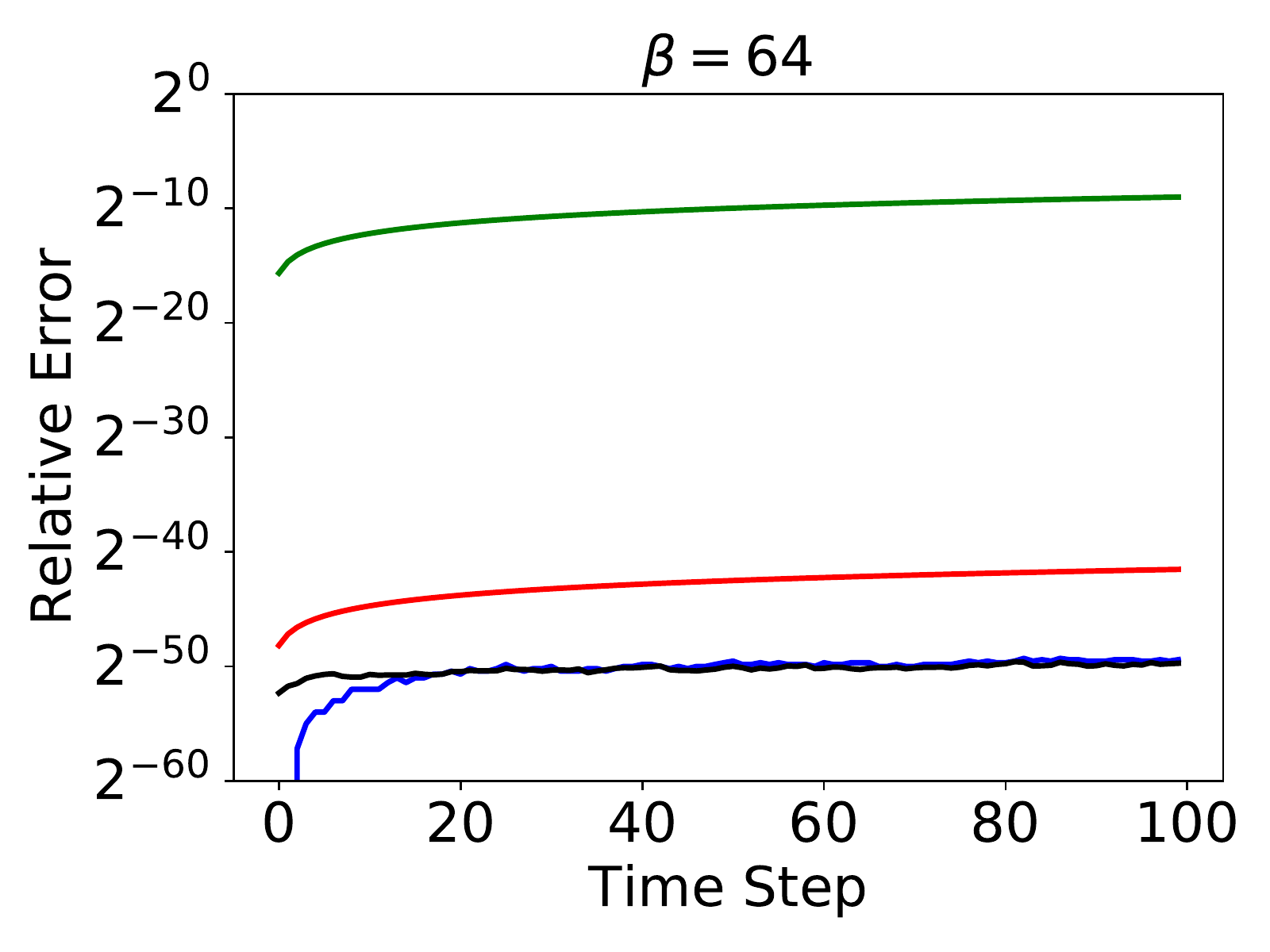}
		\caption{$\beta= 64$}
	\end{subfigure}
	\begin{subfigure}[b]{0.4\textwidth}
		\includegraphics[width=\textwidth]{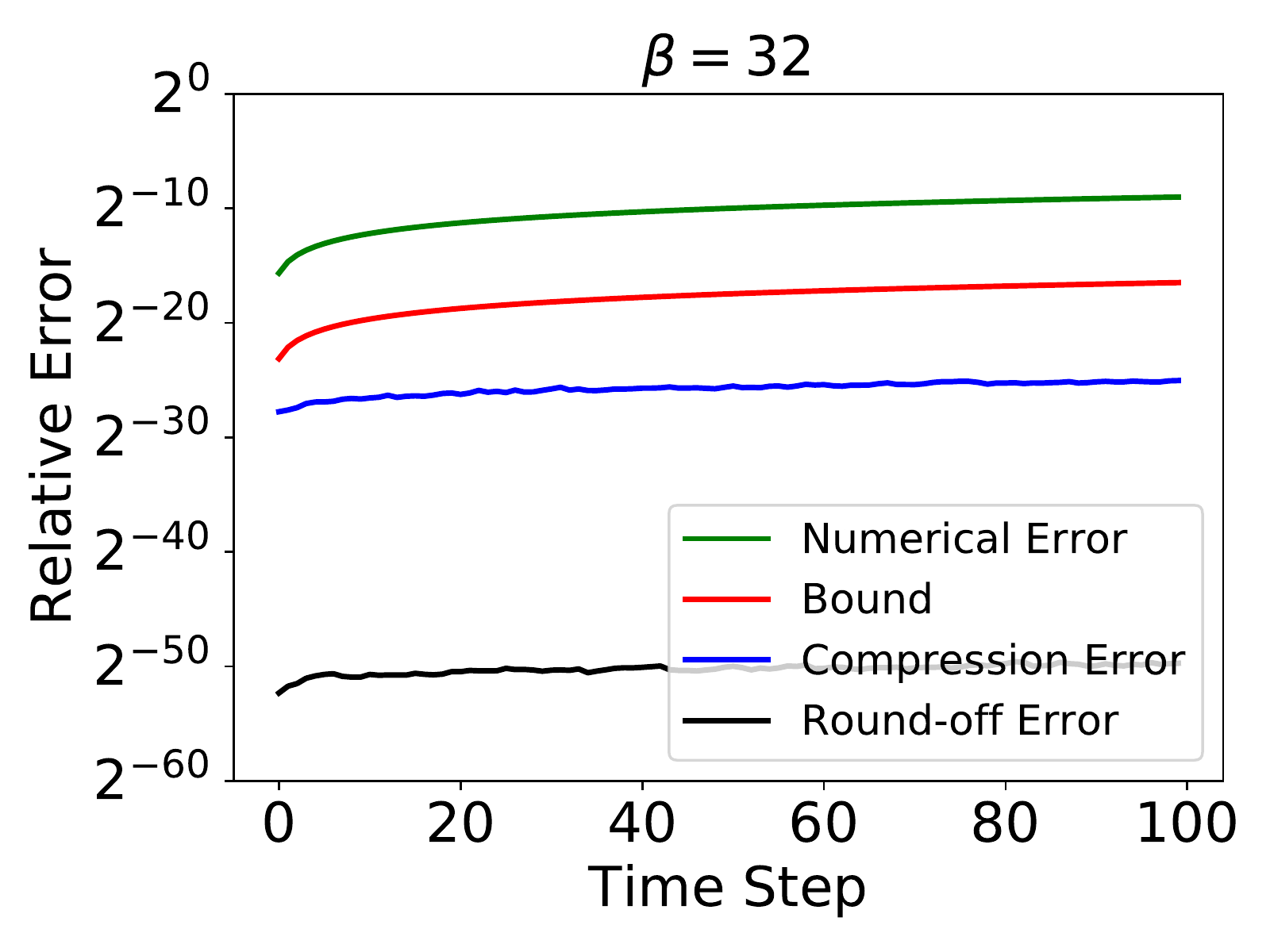}
		\caption{$\beta = 32$}
	\end{subfigure}
	
	\begin{subfigure}[b]{0.4\textwidth}
		\includegraphics[width=\textwidth]{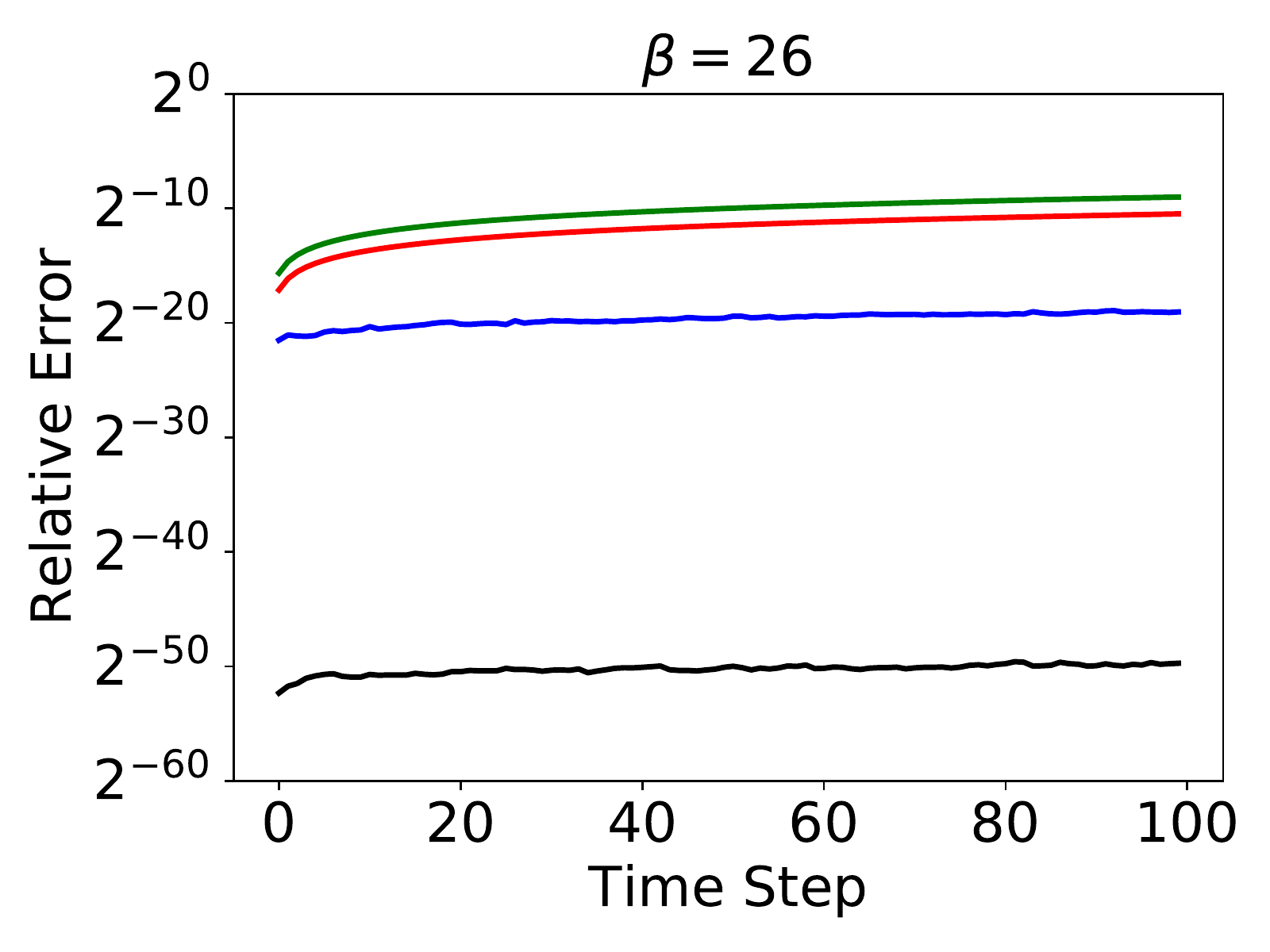}
		\caption{$\beta = 26$}
		\label{fig:laxwend262d}
	\end{subfigure}
	\begin{subfigure}[b]{0.4\textwidth}
		\includegraphics[width=\textwidth]{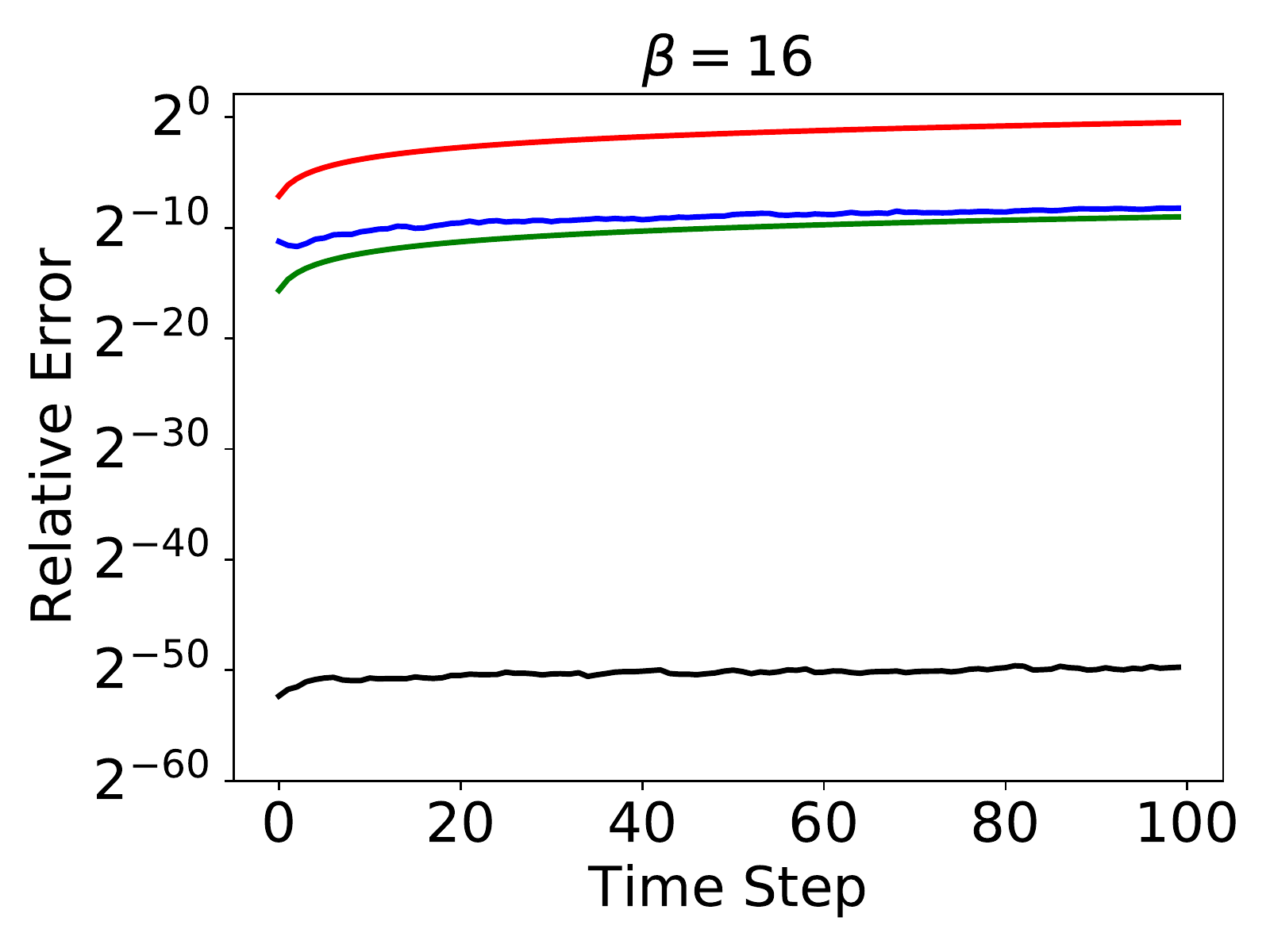}
		\caption{$\beta= 16$}
	\end{subfigure}
	\caption{{2D Lax-Wendroff Example{: The blue line represents the compression error caused by ZFP, the green line represents the truncation-dominated total error, the red line represents the theoretical bound, and the black line represents the IEEE double precision round off error estimated by using a solution calculated in $80$-bit precision.} } }
	\label{fig:laxwend2d}
\end{figure}
\begin{figure}
	\centering
		\begin{subfigure}[b]{0.4\textwidth}
		\includegraphics[width=\textwidth]{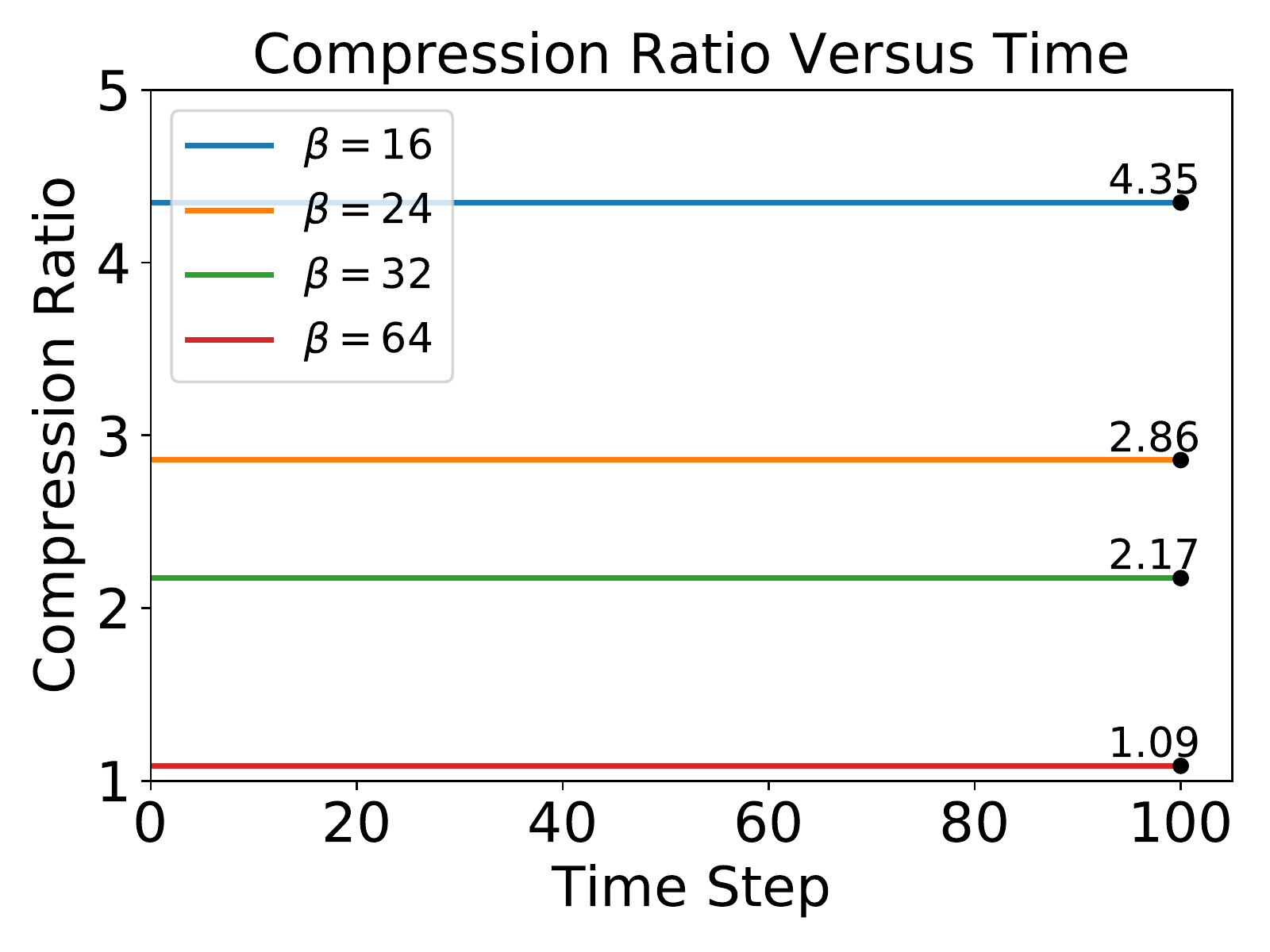}
		\caption{1D Lax-Wendroff Example } 
		\label{fig:laxcompression2}
	\end{subfigure}
		\begin{subfigure}[b]{0.4\textwidth}
		\includegraphics[width=\textwidth]{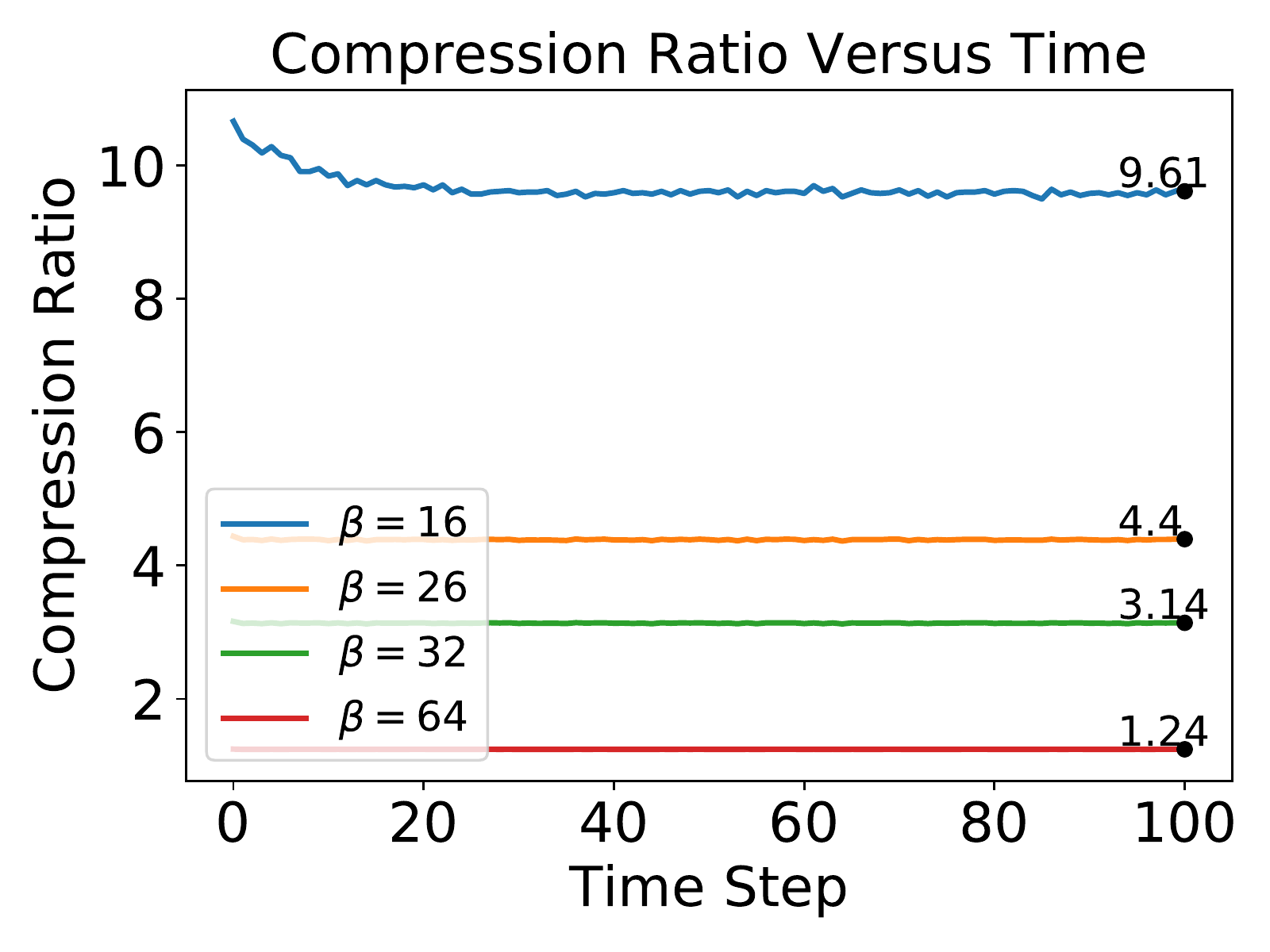}
		\caption{2D Lax-Wendroff Example } 
		\label{fig:lax2ddcompression}
	\end{subfigure}
		\caption{ The compression ratio for varying $\beta = \{64,32,24,16\}$ is displayed over time. (a) 1D Lax-Wendroff Example: (b) 2D Lax-Wendroff Example: The minimum compression ratio is $1.24$ when $\beta = 64$.   } 
\end{figure}

\section{Poisson Example}
In our last example, we have selected the Poisson equation in 2D, which will be solved using the Jacobi method. In this example, we wish to test Theorem \ref{thm:extraits}, as the discrete Poisson equation is a linear system with a fixed point solution and the Jacobi method is a stationary iterative method. Using similar notation as the previous example, let $x,y$ be the spatial variables, $u$ the continuous solution to the Poisson equation, and $\bu$ the uncompressed solution to the discretized Poisson equation. The specific PDE we solve with its initial and boundary conditions, is 
\begin{equation}
\begin{array}{rclcrcl}
&\partial^2_x u(x,y) + \partial^2_y u(x,y)  =  4,  &(x,y) \in [0,1] \times [0,1],& \\
\end{array}
\end{equation}
with initial condition 
\begin{equation}
u(x,y ) = \begin{cases}
1+y^2& x=0  \\
1+x^2& y = 0\\
2+y^2 &x = 1\\ 
2+x^2 & y= 1\\ 
0& \text{otherwise}\\
\end{cases}. 
\end{equation}
Using central differences in $x$ and $y$ directions, the algebraic system is expressed as 
\begin{equation}
\frac{u_{i-1,j}-2u_{i,j}+u_{i+1,j}}{\Delta x_1^2}+ \frac{u_{i,j-1}-2u_{i,j}+u_{i,j+1}}{\Delta x_2^2} = 4. 
\end{equation}
Let $\Delta x_1 = \Delta x_2 $. Using Jacobi iteration, we have 
\begin{equation}
u^{n+1}_{i,j} = \frac{1}{4}\left(u^n_{i-1,j}+u^n_{i+1,j}+u^n_{i,j-1}+u^n_{i,j+1} -4\Delta x_1^2\right). 
\end{equation}
Then the finite difference scheme for Poisson using Jacobi can be written as $\bu^{n+1}= A\bu^n,$ with Lipschitz constant $L_l \approx .9996$ for $\Delta x_1 = 0.01$. Define $\bv^{n+1} =
A({D}{C}(\bv^{n}) )$, where $\bv^{0} = \bu^{0}$.
\begin{figure}
	\centering   
	\begin{subfigure}[b]{0.40\textwidth}
		\includegraphics[width=\textwidth]{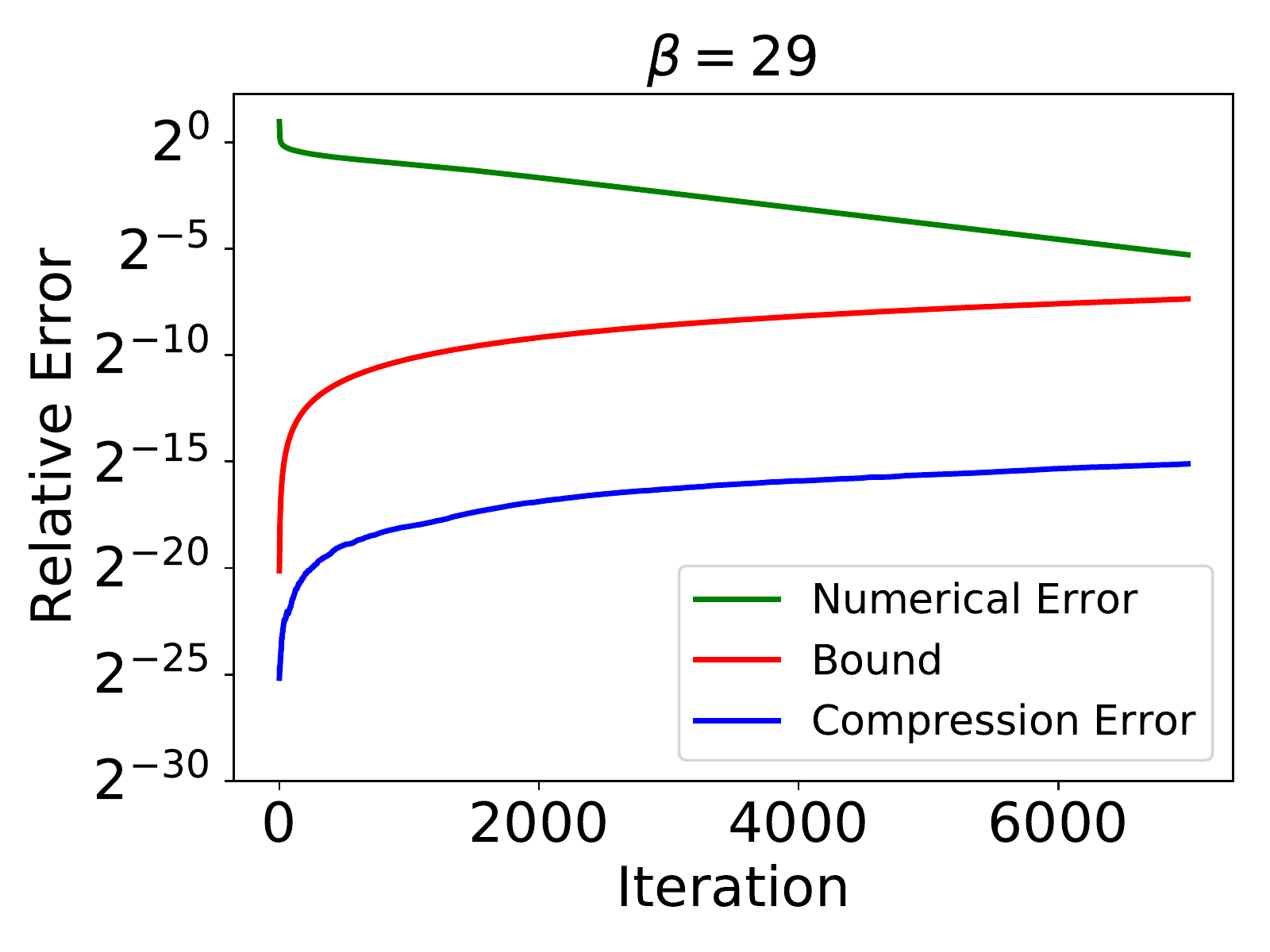}
		\caption{$\beta = 29$}
		\label{fig:poisson}		
	\end{subfigure}
	\begin{subfigure}[b]{0.40\textwidth}
		\includegraphics[width=\textwidth]{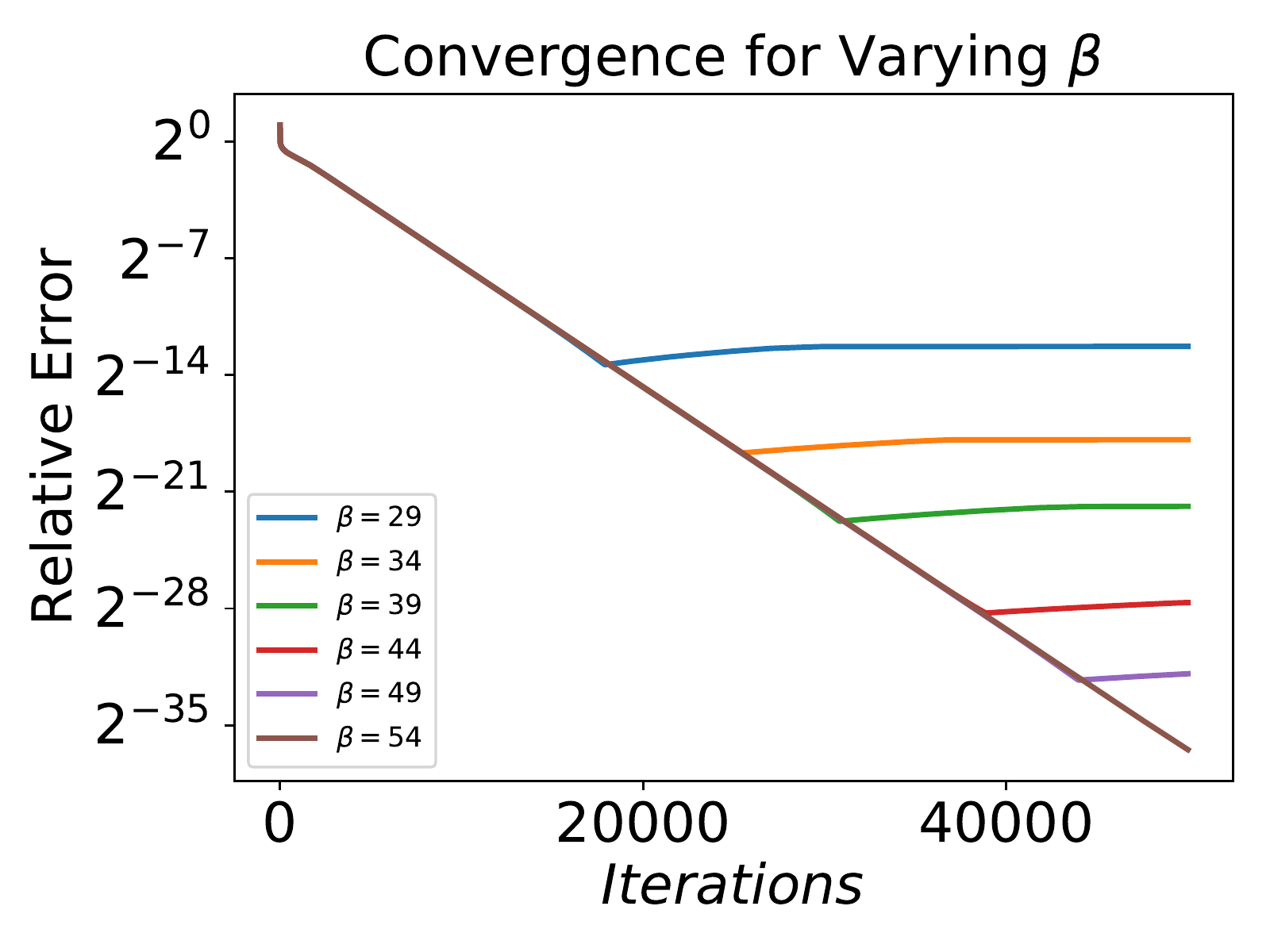}
		\caption{Convergence for varying $\beta$}
		\label{fig:poisson_convergence}		
	\end{subfigure}
	
	\begin{subfigure}[b]{0.40\textwidth}
		\includegraphics[width=\textwidth]{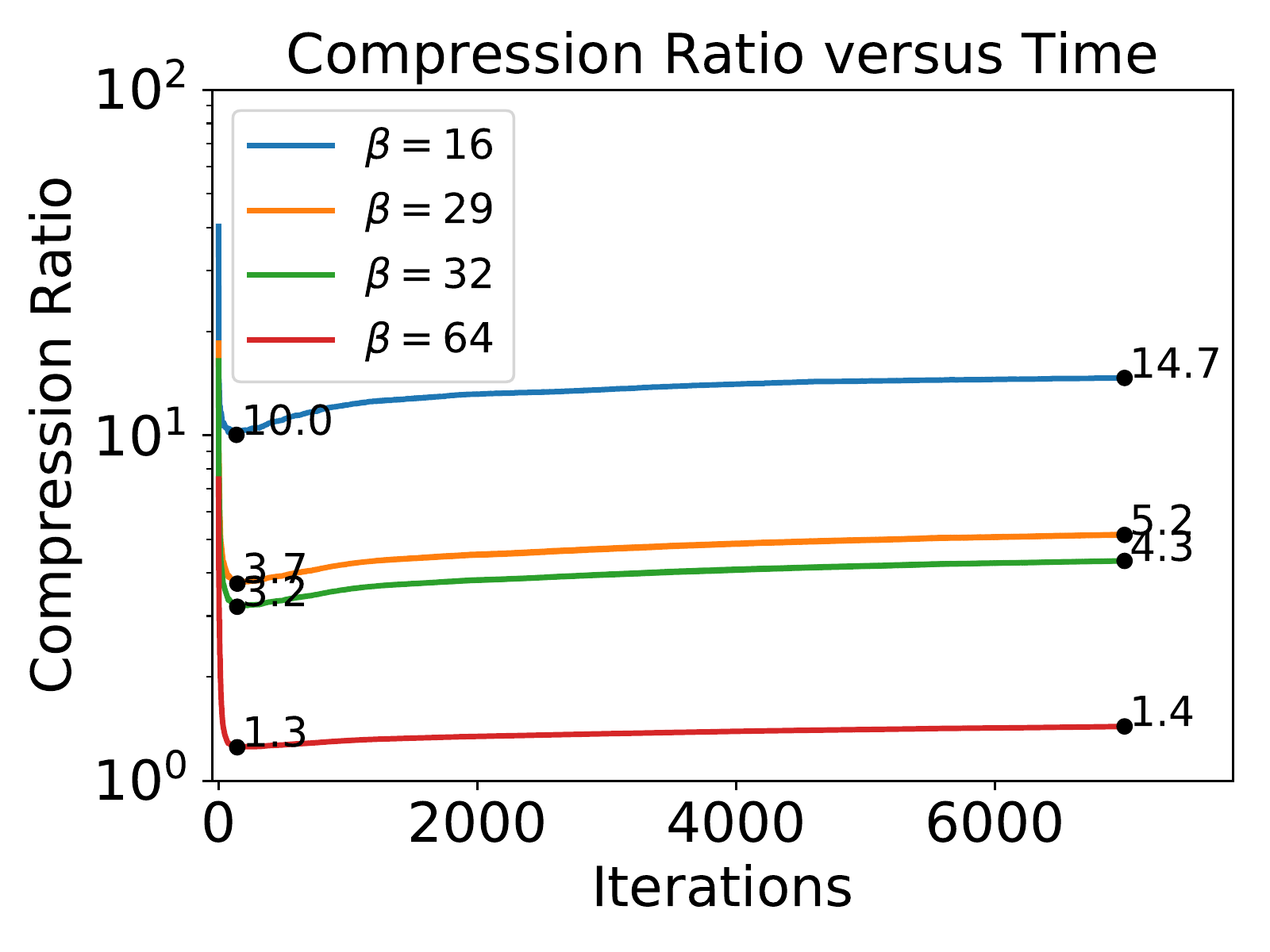}
		\caption{Compression Ratio}
		\label{fig:poissoncompression}
	\end{subfigure}
	\begin{subfigure}[b]{0.40\textwidth}
		\includegraphics[width=\textwidth]{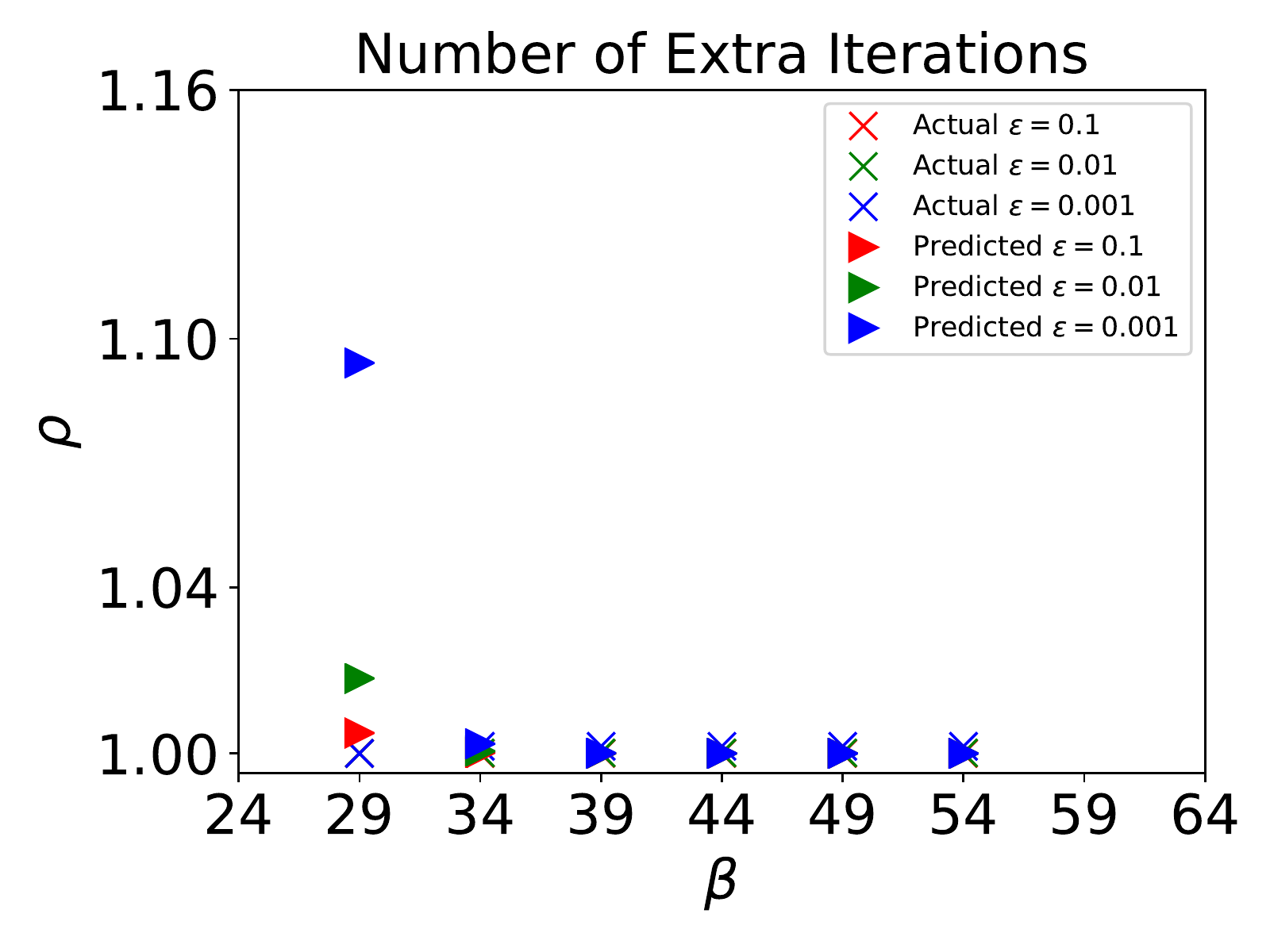}
		\caption{Number of extra iterations
		}
		\label{fig:number extra}
	\end{subfigure}
	
	\centering    
	\caption{Poisson Equation Example{: (a) The blue line represents ZFP
			error, the green line represents the total numerical error, and the red
			line represents the theoretical bound.} (b) \red{The error between the true solution, i.e., $\hat{u}(x,y) = 1+x^2+y^2$, calculated in double precision} and the ZFP solution for varying $\beta$ values. (c) The compression ratio for varying $\beta = \{64,32,29,16\}$ is displayed over time. (d) \red{For a user-defined error tolerance, $\epsilon = \{ 0.1,0.01,0.001\}$, the ratio of the number of extra iterations ($\rho_{\mathit{actual}}$ and $\rho_{\mathit{predicted}}$) provided by Theorem \ref{thm:extraits} is displayed as a function of the fixed precision parameter. From the assumptions of Theorem \ref{thm:extraits}, $\beta$ must be greater than 28.}} 
	
\end{figure}
In this example, the red lines represent the theoretical
bound from Theorem~\ref{ZFPFixedPointThm}. 
\begin{equation}
\sum_{j = 0}^{n} L_l^{n-j+1} K_{\beta_j} \|\bu^{j}\|_{\infty},
\end{equation}
where $\beta_j = 29$ is held constant. \red{In this example, the Lipschitz constant is less than one, thus Theorem \ref{thm:lip} implies Theorem \ref{thm:boundedIterative}.} {Figure \ref{fig:poisson_convergence} displays the convergence of the ZFP solution to the fixed-point,  $\|\bv^{n+1} -\hat{\bu}\|_\infty$, for varying $\beta$ values. For each $\beta$, the convergence stalls once $\|\bv^{n+1} -\hat{\bu}\|_\infty \approx \mathcal{O}\left(\frac{K_\beta}{1-L_l}\right)$. } {Figure \ref{fig:poissoncompression} represents the compression ratio for varying $\beta$ values. Due to the initial conditions, we again see a sharp decrease in the compression ratio at the start of the simulation; then, the ratio slightly increases as $n \geq 144$ increases.}

Since this is a fixed point numerical example, we can study the validity of Theorem \ref{thm:extraits}. For some $\epsilon$, let $n$ and $n_{\mathit{actual}}$ be the index such that $ \|\bu^{n+1} -  \hat{\bu}\|\leq \epsilon$ and $ \|\bv^{n_\mathit{actual}+1} -  \hat{\bu}\|\leq \epsilon$,  respectively. Define the constants
$$\rho_{\mathit{actual}} = \frac{n_{\mathit{actual}}}{n} \qquad {\text{and} } \qquad \rho_{\mathit{predicted}} = \frac{ n +m}{n},$$
where $m = \log_{L_l} \frac{ L_l^{n+1}-C}{1-C}  - (n+1) $ and  $C =\frac{K_{\tilde{\beta}}}{1-L_l}$, as defined in Theorem \ref{thm:extraits}.  For the bound in Theorem \ref{thm:extraits} to be valid, it must be the case that $\rho_{actual}\leq \rho_{\mathit{predicted}}$. Figure \ref{fig:number extra} displays $\rho_{\mathit{actual}}$ and $\rho_{\mathit{predicted}}$ for varying $\epsilon = \{0.1,0.01,0.001\}$, and $\beta$ values. For all $\beta$ values, the actual number of extra iterations, represented by $\rho_{\mathit{actual}}$, is approximately zero. As $\beta$ decreases, an exponential increase occurs for the predicted number of extra iterations, as represented by $\rho_{\mathit{predicted}}$. However, even for $\beta = 29$, $\rho_{\mathit{predicted}} \leq 1.16$ meaning that only $16\%$ extra iterations are required to ensure the error caused by ZFP is less than the theoretical worst case bound for the double precision solution. The required assumption from Theorem \ref{thm:extraits}, that $K_\beta \leq L_l^{n+1}(1-L_l)$, informs us that if $n = 7,000$, then we must use a $\beta \geq 28$ to produce a meaningful result.

\section{Conclusion}
In this paper, we addressed the accumulated round-off error introduced by the use of a compressed array data-type in fixed-point and time-stepping methods. An important contribution of this paper is the extension of the single use round-off error bound that was first formulated in \cite{errorzfp} to bound the accumulated error introduced by ZFP to both time-evolving and fixed-point iterative methods. Under reasonable assumptions on the advancement operators, Theorems \ref{thm:lip} and \ref{thm:boundedIterative} bound the error between the double precision IEEE solution state and the ZFP solution state at some iterate $n$ for general Lipschitz continuous operators and Kreiss bounded linear operators. Lemmas \ref{lemma:truncerror} and \ref{lemma:truncerror2} extend the theorems  to predict the fixed precision parameter, $\beta$, to obtain a certain accuracy over the course of the simulation. Theorem \ref{ZFPFixedPointThm} proves that if the fixed precision parameter is chosen with respect to the Lipschitz constant, then fixed-point methods on ZFP compressed arrays will converge to the same fixed-point as fixed-point methods without ZFP. Theorem \ref{thm:extraits} provides a bound on the number of extra iterations required to obtain the same accuracy as the non-compressed solution. Our results indicated that we can achieve the required precision with an appropriate fixed precision parameter. For comparison purposes, IEEE double precision was used for all the numerical examples presented. It can be seen that when $\beta =64$, the error caused by the repeated use of compression was less than the round-off error caused by IEEE floating-point round-off; however, the compression ratio remained above one, indicating that ZFP solution always used fewer bits than the IEEE double precision solution. 
 To assist in choosing a suitable $\beta$ value, we provided Lemmas \ref{lemma:forwarderrorbeta} and \ref{lemma:backwarderrorbeta}, which state the minimal fixed precision parameter, $\beta$, required to ensure that the round-off error caused by ZFP is less than traditional floating-point arithmetic error. To conclude, we have presented theoretical rationale that the continued investigation and research into ZFP is advantageous to the HPC community.   

We limited our results and analysis to the fixed precision implementation of ZFP. However, the error bounds from \cite{errorzfp} for the fixed accuracy and fixed rate implementation can be used to extend the bounds presented in our paper. Use of new efficient number representations, lossy compression algorithms, and mixed precision algorithms are all potentially beneficial  methods to reduce the memory capacity and bandwidth demands in simulation codes. The key advantage of using ZFP as a number representation over mixed precision algorithms in IEEE is that we can achieve bandwidth reduction without changing the underlying structure of the algorithm, whereas mixed precision algorithms typically require restructuring of the code in order to obtain the same accuracy. In addition, due to the flexibility of precision in ZFP via the fixed precision parameter, mixed-precision algorithms can be easily achieved using ZFP with access to any level of precision desired without the restriction to double, single, and half precision, as in IEEE. 
\section*{Acknowledgments}
This work was funded by LLNL Laboratory Directed Research and
Development as Project 17-SI-004: \emph{Variable Precision Computing}
and was performed under the auspices of the U.S. Department of Energy by
Lawrence Livermore National Laboratory under Contract
DE-AC52-07NA27344.

\bibliographystyle{siamplain}
\bibliography{compression}

\end{document}